\documentclass{preprint}
\textwidth=\paperwidth
\evensidemargin=\oddsidemargin
\topmargin=-\headheight
\vfuzz\hfuzz
\makeatletter
\newcommand\footnotetext@\relax
\let\footnotetext@\@footnotetext
\renewcommand{\thanks}[1]{%
  \unskip\protected@xdef\@thefnmark{}%
  \protected@xdef\@thanks{\@thanks\protect\footnotetext@{#1}}}
\newcommand{\MSC}[2][2010]{%
  \unskip\protected@xdef\@thefnmark{}%
  \protect\footnotetext@{\kern-1.8em{\slshape #1 MSC:}\/ #2.}}
\newcommand{\keywords}[1]{%
  \unskip\protected@xdef\@thefnmark{}%
  \protect\footnotetext@{\kern-1.8em{\slshape Keywords:}\/ #1.}}
\newcommand{\address}[1]{\\[1.67ex]%
  \parbox{.96\textwidth}{\centering\footnotesize\slshape #1}}
\gdef\@date{}
\makeatother
\usepackage{preprint}
\entrymodifiers={!!<.em,.6ex>+}
\renewcommand{\bydef}[1][=]{\overset{\mathrm{def}}#1}
\newcommand{\tto}{\rightrightarrows}
\renewcommand{\GL}{\operatorname{GL}}
\renewcommand{\SO}{\operatorname{SO}}
\newcommand{\void}{\bgroup\phantom e\egroup}
\makeatletter
\newmuskip\d@emphskip
\newmuskip\t@emphskip
\newmuskip\medemphmuskip \medemphmuskip=5mu%
\newmuskip\thickemphmuskip \thickemphmuskip=7mu%
\newcommand\m@themphpunct[1]{\mathchoice{%
  \mathord{#1}\mkern\d@emphskip}%
 {\mathord{#1}\allowbreak\mskip\t@emphskip}{#1}{#1}}
\newcommand\@emphpunct[2][3000]{%
  \ifmmode
    \m@themphpunct{#2}%
  \else
    #2\spacefactor#1{}%
  \fi}
\renewcommand{\emphpunct}{%
  \d@emphskip=\thickemphmuskip
  \t@emphskip=\thickmuskip
  \@ifstar
   {\d@emphskip=\medemphmuskip
    \t@emphskip=\medmuskip
    \@emphpunct}\@emphpunct}
\newcommand\sidetext@[3]{%
  \mskip\thickmuskip
  \mathopen{\@amstext{\upshape#1}}%
  \ifinner
    \def\text{\quitm@th}#2%
    \def\text{\@amstext}%
  \else #2\fi
  \mathclose{\@amstext{\upshape#3}}}
\def\sidetext@p(#1){\sidetext@({#1})}
\def\sidetext@b[#1]{\sidetext@[{#1}]}
\newcommand{\subsetto}{\overset{\textstyle\subset\mkern5mu}{\sm@shrelto.5ex\to}}
\makeatother
\theoremstyle{definition}
  \newtheorem{exmp}[stmt]{Example}
  \newtheorem{prob}[stmt]{Problem}
\newcommand{\Mfd}{\mathit{Mfd}}
\newcommand{\Set}{\mathit{Set}}
\newcommand{\op}[1]{#1^\mathit{op}}
\DeclareMathOperator{\Conn}{\mathit{Conn}}
\DeclareMathOperator{\Econ}{\mathit{Econ}}
\DeclareMathOperator{\Mcon}{\mathit{Mcon}}
\DeclareMathOperator{\Psr}{\mathit{Psr}}
\DeclareMathOperator{\Rep}{\mathit{Rep}}
\usepackage{accents}
\newcommand{\nef}[1]{\accentset{\ast}{#1}}
\newcommand{\idc}[1]{\accentset{\circ}{#1}}
\newcommand{\free}[1]{#1^\mathrm{free}}
\newcommand{\res}{\mathit{res}}
\DeclareMathOperator{\Func}{Func}
\makeatletter
\DeclareSymbolFont{wasyfonts}{U}{wasy}{m}{n}% math font variant
\re@DeclareMathSymbol{\intop}{\mathop}{wasyfonts}{"72}
\re@DeclareMathSymbol{\iintop}{\mathop}{wasyfonts}{"73}
\renewcommand{\int}{\intop}
\renewcommand{\iint}{\iintop}
\newcommand\avg@power\relax
\newcommand\@raiseto@power{%
  \ifx\avg@power\@empty\else^{(\avg@power)}\fi}
\newcommand\@avg[1]{\hat{#1}\@raiseto@power}
\newcommand\@wideavg[1]{\widehat{#1}\@raiseto@power}
\newcommand\p@avg\relax
\def\p@avg(#1){%
  \ifinner
    (#1)^{\boldsymbol\wedge}%
  \else
    \@wideavg{(#1)}%
  \fi}
\newcommand{\avg}[1][]{%
  \def\avg@power{#1}%
  \@ifnextchar\bgroup\@wideavg
 {\@ifnextchar(\p@avg\@avg}}
\makeatother
\title{A fast convergence theorem %
    for nearly multiplicative connections %
    on proper Lie groupoids%
  \MSC{Primary 58H05; %
    Secondary 53C05, 58C30}
  \keywords{Lie groupoid, proper groupoid, %
    multiplicative connection, pseudo-rep\-re\-sen\-ta\-tion, %
    normalized Haar system, iterative averaging, $C^k$\mdash topology}
}%
\author{Giorgio Trentinaglia%
  \address{Center for Mathematical Analysis, Geometry and Dynamical Systems, %
    Instituto Superior T\'ecnico, Universidade de Lisboa, %
    Av.~Rovisco Pais, 1049-001 Lisboa, Portugal}
  \thanks{Part of the material contained in this article was elaborated %
    while the author was a guest of the Max Planck Institute for Mathematics %
    in Bonn, Germany. The author acknowledges support %
    from the Portuguese Foundation for Science and Technology %
    (Fun\-da\-\c c\~ao pa\-ra a Ci\-\^en\-cia e a Tec\-no\-lo\-gia) %
    through the grants %
    SFRH/\allowbreak BPD/\allowbreak 81810/\allowbreak 2011 and %
    UID/\allowbreak MAT/\allowbreak 04459/\allowbreak 2013.}
}%
\begin{document}
\maketitle
%%
% document %%%%%%%%%%%%%%%%%%%%%%%%%%%%%%%%%%%%%%

\begin{abstract} Motivated by the study of a certain family of classical geometric problems we investigate the existence of multiplicative connections on proper Lie groupoids. We show that one can always deform a given connection which is only approximately multiplicative into a genuinely multiplicative connection. The proof of this fact that we present here relies on a recursive averaging technique. As an application we point out that the study of multiplicative connections on general proper Lie groupoids reduces to the study of longitudinal representations of regular groupoids. We regard our results as a preliminary step towards the elaboration of an obstruction theory for multiplicative connections. \end{abstract}

\tableofcontents

\section*{Introduction}

This paper is supposed to be the first of a series devoted to the study of multiplicative connections on Lie groupoids, also known as \emph{Cartan connections} in the literature. Examples of \emph{Cartan groupoids} i.e.~of Lie groupoids endowed with multiplicative connections abund in nature. Quite often in situations when a notion of ``finite order symmetry'' can be associated with a certain kind of geometric structures naturally, it is also possible to associate a Cartan groupoid with each structure of that kind\textemdash just as naturally. In many cases, the Cartan groupoids arising in this way happen to be proper, or even compact. The list of classical geometric structures which give rise to proper Cartan groupoids includes e.g.~``absolute parallelisms'' on manifolds (i.e.~tangent bundle trivializations), Riemannian metrics, $G$\mdash structures endowed with compatible connections \cite{Chern,Koba,Stern} for compact $G$, and more generally, ``Cartan geometries'' i.e.~Cartan forms on principal $G$\mdash bundles \cite{Bla16,Ehr50,Sharpe}, again for compact $G$. All these structures which we have just mentioned give rise to {\em transitive}\/ Cartan groupoids, but there are also many {\em intransitive}\/ examples, among which we highlight the Cartan groupoids associated with general Lie pseudo-groups of finite type \cite{Bl16a,Ehr58}.

The point of view that Cartan connections (and other closely related multiplicative structures on Lie groupoids and Lie algebroids) provide a convenient setting for formulating and studying a variety of problems in geometry has inspired a flare of research in the course of the last decade or so; we mention the papers of Blaom \cite{Bl06,Bl12,Bl13,Bla16,Bl16a}, Crainic, Salazar and Struchiner \cite{CSS}, Jotz and Ortiz \cite{JO}, Salazar's thesis \cite{Salaz}, Yudilevich's \cite{Yudil}, and also the recent book \cite{CrampS}. All these works either promote the aforesaid point of view through the study of old and new examples and the reinterpretation of classical problems, methods and results in the new language, and/or focus on the correspondence between Cartan connections on Lie groupoids and their ``infinitesimal counterparts'' on Lie algebroids. The intuitive idea behind the latter glob\-al-to-lo\-cal correspondence can easily be described. Given an arbitrary source connected Lie groupoid and an arbitrary open cover of its base, any multiplicative connection will uniquely be determined by its restrictions over the open sets of the cover, because any given arrow can (by source connectedness) be broken into some path of composable arrows each one having its end-points lying both within the same set of the cover\textemdash so that the value of the connection on the given arrow can be recovered by ``multiplying'' its values on the individual arrows of some such path. Conversely, providing that our groupoid is source simply connected, for any collection of local multiplicative connections subordinated to some open cover of the groupoid's base which agree on the overlaps of their domains of definition there will be a unique multiplicative connection whose restrictions over the open sets of the cover coincide with the given connections, because the results of ``multiplying'' local values along different paths that compose up to the same arrows will coincide. In the limit when you make the diameters of the open sets shrink to zero, any such ``net'' of local multiplicative connections will approximately look like the Lie algebroid version of a Cartan connection.

The perspective of the present paper is somewhat different. On the one hand, we are interested in a problem whose nature is essentially global. We want to be able to deal with Lie groupoids which possibly have non simply connected, or even disconnected, source fibers. On the other hand, the tools that we use narrow the scope of our theory down to the proper case; in fact, the non-prop\-er case exhibits pathologies (see below) which make it unrealistic to hope for systematic results similar to those that we are going to present in this paper. The problem that we want to consider is a far-reach\-ing generalization of the classical problem of finding the obstructions to the existence of $G$\mdash structures on manifolds \cite{Ehr50,Koba,Steen}. Specifically, we are interested in the following sort of questions:
\begin{enumerate}
\itemsep=0pt%
\def\labelenumi{(\alph{enumi})}%
 \item When does a Lie groupoid admit Cartan connections? Better, given any Lie groupoid, decide whether it admits Cartan connections or not.
 \item When can two Cartan connections on a Lie groupoid be deformed into each other through Cartan connections?
 \item Same questions, but for \emph{flat} i.e.~Frobenius integrable (as vector distributions) Cartan connections.
\end{enumerate}
The objects mentioned in (c) are sometimes called ``pseudo-ac\-tions'' \cite{Bl16a,Tang}. The above questions underlie a number of seemingly unrelated problems in geometry. Suppose we know that certain geometric structures canonically give rise to Cartan groupoids. Then, the obstructions to the existence of Cartan connections automatically translate into obstructions to the existence of those structures. Similarly, the obstructions to deforming Cartan connections into one another translate into information about the classification of the same structures up to homotopy. Another example of a problem where the questions (a)\textendash (c) are relevant is that of determining whether a singular foliation arises as the orbit foliation of some Lie group action; this is simply one instance of a more general type of problem\textemdash that of finding the global symmetries of an abstract geometric object.

The present paper will exclusively be concerned with the first two questions, (a) and (b); answering (c) would seem to require some kind of multiplicative analogue of Haefliger\textendash Thurston's theory for foliations \cite{Haefli,Thur74,Thur76}. Of course, understanding the first two questions is propaedeutic to answering the third one. We observe that in this respect, (b) appears to be particularly relevant.

Both (a) and (b) can be viewed as special cases of a more general extension problem (Problem~\ref{prob:16a.1.9}). Our main result (Theorem~\ref{prop:14A.5.2}) implies that in the proper case, the study of this extension problem reduces to the study of its regular subcase. Now for proper regular groupoids the problem can be so rephrased that it takes the form of a standard problem in equivariant (orbifold) obstruction theory\textemdash whose methods and techniques may then be put to use. Our main result is essentially a direct application of a fast convergence theorem for nearly multiplicative connections (Theorem~\ref{thm:12B.15.1}), which, we think, is of interest in its own right. The reader is referred to \S\ref{sec:16a.1}\textemdash which can be regarded as a continuation of this introduction\textemdash for a detailed description of all these ideas and results.

\subsection*{General conventions}

With few exceptions, in what follows all the manifolds will be ($C^\infty$\babelhyphen{nobreak})dif\-fer\-en\-tia\-ble, non-emp\-ty, of constant dimension, separated, and will possess a countable basis of open sets. We shall call every differentiable manifold with these properties a \emph{smooth} manifold. All the maps between smooth manifolds, as well as all the vector bundles, will be \emph{smooth}, i.e., ($C^\infty$\babelhyphen{nobreak})dif\-fer\-en\-tia\-ble. For any vector bundle $E$ over a manifold $X$, we shall let $\Gamma^\infty(X;E)$ denote the (in\-fi\-nite-di\-men\-sion\-al) vector space of differentiable cross-sec\-tions of $E$. We shall regard $\Gamma^\infty(X;E)$ as a Fr\'echet space, the topology being that of uniform convergence on compact sets up to all orders of derivation (this being also known as the \emph{\(C^\infty\)\mdash topology}).

In the course of the paper we shall meet operators which are ``differentiable'' in the intuitive sense that they are defined through differentiable expressions, but which operate between spaces whose differentiable structures, typically in\-fi\-nite-di\-men\-sion\-al, are much less intuitively defined. In order to rigorously keep track of the differentiability properties of these operators, but at the same time avoid unnecessary technical complications, we shall adopt the following standard conceptual shortcut. Let $\Mfd$ denote the category of smooth manifolds. For any given ``space'' $\mathscr X$, let $\Set(-,\mathscr X):\op\Mfd\to\Set$ denote the presheaf which to every smooth manifold $S$ assigns the set $\Set(S,\mathscr X)$ of maps from $S$ to $\mathscr X$. By a \emph{\(C^\infty\)\mdash structure} on $\mathscr X$, we mean a subsheaf $\mathcal S$ of $\Set(-,\mathscr X)$ that contains all the locally constant maps (the sheaf property being with respect to arbitrary open covers). We refer to $\mathscr X=(\mathscr X,\mathcal S)$ as a \emph{$C^\infty$\mdash space}. A \emph{$C^\infty$ mapping} between two $C^\infty$\mdash spaces will be one which by forward composition gives rise to a natural transformation between the corresponding $C^\infty$\mdash structures. For any smooth vector bundle $E\to X$, there is on the space $\Gamma^\infty(X;E)$ a canonical $C^\infty$\mdash structure consisting of those maps $S\to\Gamma^\infty(X;E)$ which give rise to smooth maps $S\times X\to E$. Most of the $C^\infty$\mdash spaces that we shall encounter will be $C^\infty$\mdash spaces of this kind, or subspaces thereof.

We shall assume familiarity with the theory of Lie groupoids at the level say of Chapters~5\textendash 6 of the textbook \cite{MM}. The notations which we shall be using will be fairly standard, with the only possible exception that we shall be writing $\varGamma^x$ resp.~$\varGamma_x$ for the source fiber $s^{-1}(x)$ resp.~target fiber $t^{-1}(x)$ of a groupoid $\varGamma$ at a base point $x$, and $\varGamma^x_y$ for the set $\varGamma^x\cap\varGamma_y=\varGamma(x,y)$ of arrows with source $x$ and target $y$, in particular $\varGamma^x_x$ for the isotropy group at $x$. Another minor departure from \cite{MM} concerns the terminology. We shall be using the term `Lie groupoid' in a more restrictive sense, namely as a synonym to `smooth groupoid', meaning that the manifold of objects and the manifold of arrows are both smooth. (This usage seems to be closer to the standard practice of Lie group theory.) On the few occasions when we actually meet examples which do not fit this more restrictive definition, we shall use the term `differentiable groupoid'.

% sections %%%%%%%%%%%%%%%%%%%%%%%%%%%%%%%%%%%%%%
%%

\section{Outline}\label{sec:16a.1}

Let \(\varGamma\tto M\) be a Lie groupoid. Let \(s,t:\varGamma\to M\) denote the groupoid source resp.~target map. Let \(E\) be any vector bundle over \(M\). A \emph{pseudo-rep\-re\-sen\-ta\-tion} of \(\varGamma\tto M\) on \(E\) is a vector bundle morphism from \(s^*E\) to \(t^*E\), in other words, a global cross-sec\-tion of the vector bundle \(L(s^*E,t^*E)\). To every arrow \(g\in\varGamma\) a pseudo-rep\-re\-sen\-ta\-tion \(\lambda:s^*E\to t^*E\) assigns a linear map \(\lambda_g:E_{sg}\to E_{tg}\) between the fibers of \(E\) corresponding to the source and to the target of \(g\). We call \(\lambda\) \emph{invertible} if \(\lambda_g\) is for every \(g\) a linear isomorphism of \(E_{sg}\) onto \(E_{tg}\). We call \(\lambda\) \emph{unital} if \(\lambda_{1_x}=\id_{E_x}\) for all \(x\) in \(M\). We call \(\lambda\) a \emph{representation} if \(\lambda\) is unital and the identity \(\lambda_{g'g}=\lambda_{g'}\lambda_g\) holds for every pair of arrows \(g',g\) which are composable in other words satisfy the condition \(sg'=tg\).

A \emph{connection} on \(\varGamma\tto M\) is a right splitting \(\eta\) of the following short exact sequence of morphisms of vector bundles over the manifold \(\varGamma\):
\begin{equation}
\label{eqn:12B.8.1}
\xymatrix@=1.67em{%
 0 \ar[r]
 &	\ker{\d s}
	\ar[r]
	&	T\varGamma
		\ar[r]^-{\d s}
		&	s^*TM \ar@/^1pc/[l]_\eta
			\ar[r]
			&	0\text.
}\end{equation}
We identify \(\eta\) with the subbundle \(H:=\im\eta\) of \(T\varGamma\), and refer to \(\eta=(\d s\mathbin|H)^{-1}\) as the \emph{horizontal lift} associated with \(H\), also written \(\eta^H\). Letting \(1:M\to\varGamma\) denote the groupoid unit map \(x\mapsto 1_x\), we call \(\eta\) \emph{unital} when the condition \(\eta_{1x}=T_x1\) is satisfied for all \(x\) in \(M\). Lie groupoids always admit unital connections. One can compose the horizontal lift \(\eta^H:s^*TM\to T\varGamma\) with the vec\-tor-bun\-dle morphism \(\d t:T\varGamma\to t^*TM\) so as to obtain a pseudo-rep\-re\-sen\-ta\-tion \(\lambda^H\) of \(\varGamma\tto M\) on \(TM\); we call \(\lambda^H\) the \emph{effect} of \(H\). By an \emph{effective} connection we mean one whose effect is a representation.

The \emph{tangent groupoid} of \(\varGamma\tto M\) is the Lie groupoid \(T\varGamma\tto TM\) whose structure maps are obtained by differentiating those of \(\varGamma\tto M\). A connection \(H\) on \(\varGamma\tto M\) is said to be \emph{multiplicative} if \(H\subset T\varGamma\) constitutes a subgroupoid of \(T\varGamma\tto TM\) (of necessity over the whole of \(TM\)). Trivially, multiplicative connections are unital. A unital connection \(H\) on \(\varGamma\tto M\) is multiplicative if and only if the identity below holds for every composable pair of arrows \(g',g\) for all tangent vectors \(\varv\in T_{sg}M\).%
\begin{subequations}
\label{eqn:16a.2}
\begin{equation}
\label{eqn:16a.2a}
	\eta^H_{g'g}\varv=(\eta^H_{g'}\lambda^H_g\varv)\eta^H_g\varv
\end{equation}
Multiplicative connections are always effective, as one can see by applying the linear map \(T_gt\) to both sides of this identity.

Let \(\varGamma_\div\) denote the submanifold of \(\varGamma\times\varGamma\) formed by the pairs of arrows $(g,h)$ such that \(sg=sh\). We shall call any such pair a \emph{divisible} pair. Let \(q_\div\) stand for the \emph{division map} \(\varGamma_\div\to\varGamma\emphpunct,(g,h)\mapsto gh^{-1}\). For any divisible pair of arrows \(\varw_1,\varw_2\) in the tangent groupoid \(T\varGamma\tto TM\) let \(\varw_1\div\varw_2\) denote the ``ratio'' \(\varw_1{\varw_2}^{-1}=(Tq_\div)(\varw_1,\varw_2)\). Then for any unital connection \(H\) on \(\varGamma\tto M\) the above multiplicativity equations \eqref{eqn:16a.2a} take the following, alternative form, where \(g,h\) is any divisible pair of arrows in \(\varGamma\tto M\) and \(\varv\in T_{sg=sh}M\) is any tangent vector:
\begin{equation}
\label{eqn:12B.11.4}
	\eta^H_{gh^{-1}}\lambda^H_h\varv=\eta^H_g\varv\div\eta^H_h\varv\text.
\end{equation}
\end{subequations}

We shall let \(\Conn(\varGamma)\) denote the space of connections on \(\varGamma\tto M\). This is the affine subspace of \(\Gamma^\infty\bigl(\varGamma;L(s^*TM,T\varGamma)\bigr)\) formed by those differentiable cross-sec\-tions \(\eta\) of the vector bundle \(L(s^*TM,T\varGamma)\) which are solutions for the equation \(\d s\circ\eta=\id_{s^*TM}\). We shall view \(\Conn(\varGamma)\) as a \(C^\infty\)\mdash space throughout. We shall moreover let \(\Conn_1(\varGamma)\), \(\Econ(\varGamma)\), and \(\Mcon(\varGamma)\) denote the subspaces of \(\Conn(\varGamma)\) respectively formed by the unital, effective, and multiplicative connections.

\subsection*{Averaging}

Call \emph{non-de\-gen\-er\-ate} a connection \(H\) on \(\varGamma\tto M\) whose effect \(\lambda^H\) is an invertible pseudo-rep\-re\-sen\-ta\-tion. Let \(\Conn_\div(\varGamma)\) denote the subspace of \(\Conn(\varGamma)\) formed by the non-de\-gen\-er\-ate connections.

\begin{defn}\label{defn:12B.11.1} For any non-de\-gen\-er\-ate connection \(H\) on \(\varGamma\tto M\), and for any divisible pair \((g,h)\in\varGamma_\div\), put%
\begin{subequations}
\label{eqn:12B.11.5}
\begin{equation}
\label{eqn:12B.11.5a}
	\delta^H(g,h)\bydef(\eta^H_g\div\eta^H_h)\circ(\lambda^H_h)^{-1}\in L(T_{th}M,T_{gh^{-1}}\varGamma)\text.
\end{equation}
Let \(s_\div\) stand for the map of \(\varGamma_\div\) into \(M\) given by \((g,h)\mapsto th\). The global cross-sec\-tion
\begin{equation}
\label{eqn:12B.11.5b}
	\delta^H\in\Gamma^\infty\bigl(\varGamma_\div;L(s_\div^*TM,q_\div^*T\varGamma)\bigr)
\end{equation}
\end{subequations}
shall be referred to as the ``division cocycle'' associated with \(H\). \end{defn}

From now on, let us assume that \(\varGamma\tto M\) is {\em proper}\/ and hence can be endowed with a \emph{normalized Haar system}; recall that any such system \(\nu=\{\nu_x\}\) assigns each base point \(x\in M\) a positive Radon measure \(\nu_x\) on the target fiber \(\varGamma_x=t^{-1}(x)\) in such a way that the following three conditions are satisfied:%
\begin{subequations}
\label{eqn:12B.10.3}
\begin{enumerate}
\def\labelenumi{(\Alph{enumi})}%
\itemsep=0pt%
 \item For some differentiably varying family \(\tau=\{\tau_x\}\) of volume densities on the target fibers and for some differentiable non-neg\-ative function \(c\) on \(M\) with the property that for each compact subset \(K\) of \(M\) the intersection \(\supp c\cap\varGamma K\) is compact, one has \(\d\nu_x=(c\circ s_x)\d\mu_x\) for all \(x\), where \(\mu_x\) denotes the positive Radon measure on \(\varGamma_x\) associated with the volume density \(\tau_x\), and \(s_x\) denotes the restriction of the map \(s\) along \(\varGamma_x\).
 \item For every arrow \(g\), and for all Borel subsets \(A\) of the target fiber \(\varGamma_{sg}\),
\begin{equation}
\label{eqn:12B.10.3a}
	\nu_{tg}(gA)=\nu_{sg}(A)\text.
\end{equation}
 \item \(\nu_x(\varGamma_x)=1\) for every \(x\).
\end{enumerate}
The first condition implies that \(C^0(\varGamma_x)\subset L^1(\nu_x)\). The second condition, \emph{left invariance}, can be rephrased by saying that for every continuous function \(\varphi\) on \(\varGamma_{sg}\)
\begin{equation}
\label{eqn:12B.10.3b}
	\integral\varphi(g^{-1}h)d\nu_{tg}(h)=\integral\varphi(h)d\nu_{sg}(h)\text.
\end{equation}
\end{subequations}

Next, let \(f:P\to M\) be a map of some manifold of ``parameters'' into the base of our groupoid. If \(E\) is any vector bundle over \(P\) then, letting \(\pr_P\) denote the projection from the fiber product \(P\ftimes ft\varGamma:=\{(y,h)\in P\times\varGamma:f(y)=th\}\) onto \(P\), every cross-sec\-tion \(\vartheta\) of the vector bundle \(\pr_P^*E\) can be turned into a cross-sec\-tion \(\integral\vartheta d\nu\) of \(E\) by \emph{integration along the target fibers:}%
\begin{subequations}
\label{eqn:12B.10.6}
\begin{equation}
\label{eqn:12B.10.6a}
	P\ni y\mapsto\bigl(\integral*\vartheta d\nu\bigr)(y)\bydef\integral\vartheta(y,h)d\nu_{f(y)}(h)\in E_y
\end{equation}
(the integrand here being a vec\-tor-val\-ued differentiable function on \(\varGamma_{f(y)}\) with values in the fi\-nite-di\-men\-sion\-al vector space \(E_y\)). The resulting ``integration functional''
\begin{equation}
\label{eqn:12B.10.6b}
	\Gamma^\infty(P\ftimes ft\varGamma\emphpunct*;\pr_P^*E)\longto\Gamma^\infty(P;E)\emphpunct, \vartheta\mapsto\integral*\vartheta d\nu
\end{equation}
\end{subequations}
is continuous (as a linear map between Fr\'e\-chet spaces).

Appendix~\ref{sec:16a.B} may be consulted for additional information.

\begin{defn}\label{defn:12B.11.2} Let \(H\) be a non-de\-gen\-er\-ate connection on \(\varGamma\tto M\). For every arrow \(g\in\varGamma\) let \(\avg\eta^H_g\) denote the linear map%
\begin{subequations}
\label{eqn:12B.11.6}
\begin{equation}
\label{eqn:12B.11.6a}
	T_{sg}M\ni\varv\mapsto\avg\eta^H_g\varv\bydef\integral_{tk=sg}\delta^H(gk,k)\varv dk\in T_g\varGamma\text.
\end{equation}
[This expression makes sense because \(\delta^H(gk,k)\) is a linear map of \(T_{tk=sg}M\) into \(T_{gkk^{-1}=g}\varGamma\) for all \(\xfro g\,\xfro k\).] We shall refer to the global cross-sec\-tion
\begin{equation}
\label{eqn:12B.11.6b}
	\avg\eta^H\in\Gamma^\infty\bigl(\varGamma;L(s^*TM,T\varGamma)\bigr)
\end{equation}
\end{subequations}
as the \emph{multiplicative average} of \(H\). (Our definition depends of course on the choice of a normalized Haar system on \(\varGamma\tto M\).) \end{defn}

The multiplicative average \(\avg\eta^H\) of every non-de\-gen\-er\-ate groupoid connection \(H\) on \(\varGamma\) happens to be itself the horizontal lift for a unique groupoid connection \(\avg H\) on \(\varGamma\). This new connection, \(\avg H\), is always unital. We shall call it too the multiplicative average of \(H\). We have the following integral formula for \(\avg\lambda^H_g:=T_gt\circ\avg\eta^H_g\), the effect of the multiplicative average of \(H\), in terms of the effect of \(H\).
\begin{equation}
\label{eqn:12B.11.7}
	\avg\lambda^H_g=\integral_{tk=sg}\lambda^H_{gk}\circ(\lambda^H_k)^{-1}dk
\end{equation}
Details will be supplied in \S\ref{sec:16a.3}, along with a proof of the following result:

\begin{prop}\label{prop:12B.11.7} Let\/ \(\varPhi\) be any effective\textemdash a~fortiori, non-de\-gen\-er\-ate\textemdash connection on a proper Lie groupoid\/ \(\varGamma\tto M\). Then, for an arbitrary choice of normalized Haar systems on\/ \(\varGamma\tto M\), the multiplicative average\/ \(\avg\varPhi\) of\/ \(\varPhi\) is a multiplicative connection. \end{prop}

In fact, for any proper Lie groupoid \(\varGamma\) endowed with some choice of normalized Haar systems the process that to each non-de\-gen\-er\-ate groupoid connection \(\varPhi\) on \(\varGamma\) assigns the multiplicative average of \(\varPhi\) defines a \(C^\infty\) mapping between \(C^\infty\)\mdash spaces
\begin{equation}
\label{eqn:16a.8}
	\Conn_\div(\varGamma)\longto\Conn_1(\varGamma)\emphpunct, \varPhi\mapsto\avg\varPhi,
\end{equation}
which we call the \emph{averaging operator} for groupoid connections on \(\varGamma\). The multiplicative connections lie within the fixed-point set of this operator. By the proposition, this operator carries effective connections into multiplicative connections. Notice that since whenever \(\varPhi\) is effective so is every connection on the line segment \(\{\varPhi+t(\avg\varPhi-\varPhi):0\leq t\leq 1\}\) the averaging operator must provide a strong \(C^\infty\) deformation retraction of the \(C^\infty\)\mdash space \(\Econ(\varGamma)\) onto its subspace \(\Mcon(\varGamma)\). In general, the statement that \(\Econ(\varGamma)\) retracts onto \(\Mcon(\varGamma)\) is false for non-prop\-er Lie groupoids; cf.~Example~\ref{npar:14A.3.4} below.

Evidently, every connection on a \emph{Lie bundle} is effective. (By a Lie bundle we mean a Lie groupoid whose source and target map coincide.) More in general, the same is true of every connection on any Lie groupoid whose associated Lie algebroid has zero anchor map. It follows from Proposition~\ref{prop:12B.11.7} that in the proper case these groupoids always admit multiplicative connections. Generalizing further, it is easy to see that every regular Lie groupoid whose \emph{longitudinal bundle} is trivializable admits effective connections. (Cf.~Example~\ref{prop:14A.3.2} below; by definition, the longitudinal bundle of a regular Lie groupoid \(\varGamma\tto M\) is the subbundle of \(TM\) consisting of all the vectors tangent to the orbits of \(\varGamma\).) Again, Proposition~\ref{prop:12B.11.7} implies that any such groupoid which is also proper admits multiplicative connections. These remarks serve to illustrate the usefulness of our averaging method already at the relatively low level of sophistication of this section thus far.

\subsection*{Fast convergence}

Next, for a general (non effective) non-de\-gen\-er\-ate connection \(\varPhi\) we want to consider the sequence of connections \(\avg\varPhi,\avg{\avg\varPhi},\mathellipsis\) that one obtains by repeatedly averaging \(\varPhi\) (provided this sequence is at all defined), and understand its limiting behavior.

Let \(E\) be an arbitrary vector bundle over the base \(M\) of our proper Lie groupoid \(\varGamma\). Let \(\Psr(\varGamma;E)\) denote the \(C^\infty\)\mdash space \(\Gamma^\infty\bigl(\varGamma;L(s^*E,t^*E)\bigr)\) of pseudo-rep\-re\-sen\-ta\-tions of \(\varGamma\) on \(E\). Also, let \(\Psr_\div(\varGamma;E)\subset\Psr(\varGamma;E)\) denote the subspace of invertible pseudo-rep\-re\-sen\-ta\-tions, and \(\Psr_1(\varGamma;E)\) that of unital pseudo-rep\-re\-sen\-ta\-tions. Motivated by the formula \eqref{eqn:12B.11.7}, for every \(\lambda\in\Psr_\div(\varGamma;E)\emphpunct, g\in\varGamma\) we set%
\begin{equation}
\label{eqn:12B.12.4}
	\avg\lambda(g)\bydef\integral_{tk=sg}\lambda(gk)\circ\lambda(k)^{-1}dk\text.
\end{equation}
It follows from the fundamental properties of Haar integrals depending on parameters recalled previously that \(\avg\lambda\) belongs to \(\Psr_1(\varGamma;E)\) i.e.~is a unital pseudo-rep\-re\-sen\-ta\-tion of \(\varGamma\) on \(E\). The integral formula \eqref{eqn:12B.12.4} gives rise to a \(C^\infty\)\mdash mapping between \(C^\infty\)\mdash spaces
\begin{equation}
\label{eqn:16a.10}
	\Psr_\div(\varGamma;E)\longto\Psr_1(\varGamma;E)\emphpunct, \lambda\mapsto\avg\lambda,
\end{equation}
which we call the \emph{averaging operator} for pseudo-rep\-re\-sen\-ta\-tions of \(\varGamma\) on \(E\).

Let us endow \(E\) with some vec\-tor-bun\-dle metric \(\phi\) of class \(C^\infty\) (Riemannian or Hermitian, depending on whether \(E\) is real or complex). For every pair of base points \(x,y\in M\) we have a norm \(\|\void\|_{x,y}\) on \(L(E_x,E_y)\) given by
\begin{equation}
\label{eqn:16a.11}
	\|\lambda\|_{x,y}:=\sup_{|\varv|_x=1}|\lambda\varv|_y,
\end{equation}
where \(|\void|_x\) stands for the usual norm on \(E_x\) defined by \(|\varv|_x:=\sqrt{\phi_x(\varv,\varv)}\). Now for every \(\lambda\in\Psr(\varGamma;E)\) let us set%
\begin{subequations}
\label{eqn:12B.12.8}
\begin{align}
\label{eqn:12B.12.8a}
	b(\lambda)&:=\sup_{g\in\varGamma}\|\lambda(g)\|_{sg,tg}	\text{\qquad	and}
\\%
\label{eqn:12B.12.8b}
	c(\lambda)&:=\sup_{(g',g)\in\varGamma\ftimes st\varGamma}\|\lambda(g'g)-\lambda(g')\circ\lambda(g)\|_{sg,tg'}\text.
\end{align}
\end{subequations}
For all \(\lambda\in\Psr_1(\varGamma;E)\) satisfying the condition \(c(\lambda)<1\), it is possible to show that \(\lambda\) must be invertible and that the following estimates hold (compare \S\ref{sec:16a.4}).%
\begin{subequations}
\label{eqn:12B.12.10}
\begin{gather}
\label{eqn:12B.12.10a}
	\|\avg\lambda(g)\|_{sg,tg}\leq\frac{b(\lambda)}{1-c(\lambda)}
\\%
\label{eqn:12B.12.10b}
	\|\avg\lambda(g'g)-\avg\lambda(g')\circ\avg\lambda(g)\|_{sg,tg'}\leq 2\left(\frac{b(\lambda)}{1-c(\lambda)}\right)^2c(\lambda)^2
\end{gather}
\end{subequations}

\begin{lem}\label{lem:12B.12.8} Let\/ \(\{b_0,b_1,\dotsc,b_l\}\) and\/ \(\{c_0,c_1,\dotsc,c_l\}\) be two finite sequences of non-neg\-a\-tive real numbers of length say \(l+1\geq 2\). Suppose that for every index\/ \(i\) between\/ \(0\) and\/ \(l-1\) the following implication is true:
\begin{equation}
\label{eqn:12B.12.11}
 c_i<1\seq\left\{%
\begin{aligned}
	b_{i+1}&\leq\frac{b_i}{1-c_i}	\text{\quad	and}
\\	c_{i+1}&\leq 2\left(\frac{b_i}{1-c_i}\right)^2c_i^2\text.
\end{aligned}\right.%
\end{equation}
Also suppose that\/ \(b_0\geq 1\) and that\/ \(\varepsilon:=6b_0^2c_0\leq\frac 23\). Then, the following inequalities must hold for every index\/ \(i=0,1,\dotsc,l\).%
\begin{subequations}
\label{eqn:12B.12.12}
\begin{gather}
\label{eqn:12B.12.12a}
	c_i\leq\frac{\varepsilon^{2^i}}{6b_0^2}
\\%
\label{eqn:12B.12.12b}
	\frac{b_i}{1-c_i}\leq\sqrt 3b_0
\end{gather}
\end{subequations} \end{lem}

The proof will be given in \S\ref{sec:16a.4}. For an arbitrary open subset \(U\) of \(M\) let \(\lambda\mathbin|U\) denote the pseudo-rep\-re\-sen\-ta\-tion of the open subgroupoid \(\varGamma\mathbin|U\tto U\) of \(\varGamma\tto M\) on \(E\mathbin|U\) induced by \(\lambda\) upon restriction. The above lemma motivates our next definition.

\begin{defn}\label{defn:12B.14.1} A unital pseudo-rep\-re\-sen\-ta\-tion \(\lambda\in\Psr_1(\varGamma;E)\) is \emph{nearly multiplicative} or a \emph{near representation} if for each point in \(M\) one can find an invariant open neighborhood \(U=\varGamma U\) with the property that the inequality below holds for some choice of \(C^\infty\) metrics on \(E\mathbin|U\).
\begin{equation}
\label{eqn:12B.14.1}
	c(\lambda\mathbin|U)\leq\mathinner\frac 19b(\lambda\mathbin|U)^{-2}
\end{equation}
A unital connection \(\varPhi\in\Conn_1(\varGamma)\) is \emph{nearly effective} if the associated pseudo-rep\-re\-sen\-ta\-tion \(\lambda^\varPhi\in\Psr_1(\varGamma;TM)\) (i.e.~the effect of \(\varPhi\)) is a near representation. \end{defn}

\noindent We might occasionally refer to nearly effective connections improperly as ``nearly multiplicative'' connections, as we did in the title.

Near representations turn out to be always invertible (cf.~\S\ref{sec:16a.5}) so for any near representation \(\lambda\) it makes sense to consider the pseudo-rep\-re\-sen\-ta\-tion \(\avg\lambda\) defined by our averaging formula \eqref{eqn:12B.12.4}. One can then show that \(\avg\lambda\) itself must be a near representation. One therefore obtains a whole sequence of \emph{averaging iterates} \(\avg[i]\lambda\) of \(\lambda\) which one constructs recursively by setting \(\avg[0]\lambda:=\lambda\) and \(\avg[i+1]\lambda:=\avg(\avg[i]\lambda)\) for all \(i\).

\begin{thm}[\bfseries Fast Convergence Theorem~A\normalfont ]\label{thm:12B.14.2} Let\/ \(\varGamma\tto M\) be a proper Lie groupoid. Let\/ \(\lambda\in\Psr_1(\varGamma;E)\) be a unital pseudo-rep\-re\-sen\-ta\-tion of\/ \(\varGamma\tto M\) on some vector bundle\/ \(E\) over\/ \(M\). Suppose that\/ \(\lambda\) is nearly multiplicative. Then, for every choice of normalized Haar systems on\/ \(\varGamma\tto M\), the sequence of successive averaging iterates of\/ \(\lambda\) obtained by recursive application of the formula\/ \eqref{eqn:12B.12.4}
\begin{equation*}
	\avg[0]\lambda:=\lambda\emphpunct, \avg[1]\lambda:=\avg\lambda\emphpunct, \dotsc\emphpunct, \avg[i+1]\lambda:=\avg(\avg[i]\lambda)\emphpunct, \dotsc\in\Psr_1(\varGamma;E)
\end{equation*}
converges within the Fr\'e\-chet space\/ \(\Gamma^\infty\bigl(\varGamma;L(s^*E,t^*E)\bigr)\) (endowed with the\/ \(C^\infty\)\mdash topology) towards a unique representation, say, \(\avg[\infty]\lambda\), of\/ \(\varGamma\tto M\) on\/ \(E\). \end{thm}

\noindent The proof can be found in \S\ref{sec:16a.5}.

Next, our formula \eqref{eqn:12B.11.7} says that \(\lambda^{\avg\varPhi}\) equals \(\avg(\lambda^\varPhi)\) for any non-de\-gen\-er\-ate connection \(\varPhi\). We thus have that any nearly effective connection \(\varPhi\) gives rise, by recursive averaging, to a sequence of nearly effective connections \(\within*\{\avg[i]\varPhi\}_{i=0}^\infty\). Our next result\textemdash whose proof is to be found in \S\ref{sec:16a.6}\textemdash is essentially a corollary of the preceding one.

\begin{thm}[\bfseries Fast Convergence Theorem~B\normalfont ]\label{thm:12B.15.1} Let\/ \(\varGamma\tto M\) be a proper Lie groupoid. Let\/ \(\varPsi\in\Conn_1(\varGamma)\) be a unital connection on\/ \(\varGamma\tto M\). Suppose that\/ \(\varPsi\) is nearly effective. Then for any choice of normalized Haar systems on\/ \(\varGamma\tto M\) the sequence of successive averaging iterates of\/ \(\varPsi\) constructed by recursive application of the averaging operator\/ \eqref{eqn:16a.8}
\begin{equation*}
	\avg[0]\varPsi:=\varPsi\emphpunct, \avg[1]\varPsi:=\avg\varPsi\emphpunct, \dotsc\emphpunct, \avg[i+1]\varPsi:=\avg(\avg[i]\varPsi)\emphpunct, \dotsc\in\Conn_1(\varGamma)
\end{equation*}
converges within the affine Fr\'e\-chet submanifold\/ \(\Conn(\varGamma)\subsetto\Gamma^\infty\bigl(\varGamma;L(s^*TM,T\varGamma)\bigr)\emphpunct, H\mapsto\eta^H\) towards a unique multiplicative connection\/ \(\avg[\infty]\varPsi\) on\/ \(\varGamma\tto M\). \end{thm}

\subsection*{Applications}

Let \(\varGamma\tto M\) be any proper Lie groupoid. The orbit \(\varGamma x\) corresponding to each base point \(x\) will be a smooth submanifold of \(M\). For every non-neg\-a\-tive integer \(q\) let
\begin{gather*}
	M_q:=\{x\in M:\dim\varGamma x=q\}
\\	M_{\leq q}:=\{x\in M:\dim\varGamma x\leq q\}
\end{gather*}
denote the set of base points \(x\) which lie on orbits \(\varGamma x\) of dimension equal to\textemdash resp., not greater than\textemdash \(q\); also, let \(M_{<q}\) denote the set-the\-o\-ret\-ic difference \(M_{\leq q}\smallsetminus M_q\). Every \(M_{\leq q}\) is obviously an invariant closed subset of \(M\). Moreover, since \(\varGamma\tto M\) is proper, every \(M_q\) is an invariant differentiable submanifold of \(M\), in which case we refer to \(M_q\), or to the differentiable subgroupoid \(\varGamma\mathbin|M_q\tto M_q\) of \(\varGamma\tto M\), as the \(q\)-th \emph{regular stratum} of \(\varGamma\tto M\). (Note that the restriction of \(\varGamma\mathbin|M_q\tto M_q\) over each invariant component of \(M_q\) is an ordinary Lie groupoid, but different invariant components of \(M_q\) might have different dimensions.)

Let \(Z\subset M\) be an arbitrary invariant differentiable submanifold. Every connection \(H\) on \(\varGamma\tto M\) restricts to a connection on \(\varGamma\mathbin|Z\tto Z\), which we designate \(H\mathbin|Z\) and refer to as the connection \emph{induced} by \(H\) along \(Z\). Explicitly,
\begin{equation*}
	\eta^{H|Z}_g:=\eta^H_g\mathbin|T_{sg}Z:T_{sg}Z\to(T_gs)^{-1}(T_{sg}Z)=T_g(\varGamma\mathbin|Z).
\end{equation*}
Because of the invariance of \(Z\), the effect under \(H\) of each arrow \(g\in\varGamma^Z_Z\) must carry the linear subspace \(T_{sg}Z\subset T_{sg}M\) into the linear subspace \(T_{tg}Z\subset T_{tg}M\). Hence
\begin{equation}
\label{eqn:16a.17}
	\lambda^{H|Z}_g=\lambda^H_g\mathbin|T_{sg}Z.
\end{equation}
Moreover, \(H\mathbin|Z\) will be effective or multiplicative whenever so is \(H\).

The next result is substantially a direct consequence of Theorem~\ref{thm:12B.15.1}:

\begin{thm}\label{prop:14A.5.2} Let\/ \(\varGamma\tto M\) be a Lie groupoid that is proper. Suppose\/ \(C\) is an invariant closed subset of\/ \(M\) and\/ \(U\) is an open neighborhood of\/ \(C\). Let\/ \(Z\) be an invariant differentiable submanifold of\/ \(M\). Suppose further that a multiplicative connection\/ \(\varPhi\) is given on\/ \(\varGamma\mathbin|U\tto U\) and that\/ \(\varTheta\) is a multiplicative connection on\/ \(\varGamma\mathbin|Z\tto Z\) whose restriction over\/ \(Z\cap U\) coincides with the connection induced by\/ \(\varPhi\) on\/ \(\varGamma\mathbin|U\cap Z\tto U\cap Z\). Then there exist open neighborhoods\/ \(V\) of\/ \(C\cup Z\) and multiplicative connections\/ \(\varPsi\) on\/ \(\varGamma\mathbin|V\tto V\) which induce\/ \(\varTheta\) along\/ \(Z\) and agree with\/ \(\varPhi\) over some open neighborhood of\/ \(C\) within\/ \(U\cap V\). \end{thm}

The proof is to be found in \S\ref{sec:16a.7}. The main application of this theorem is intended to be to the case \(C=M_{<q}\emphpunct, Z=M_q\). Whenever we are given some multiplicative connection \(\varPhi\) defined around \(M_{<q}\) which we know how to extend along \(M_q\), the theorem enables us to extend \(\varPhi\) over a whole open neighborhood of \(M_{\leq q}\) at the expense of shrinking the domain of definition of \(\varPhi\) around \(M_{<q}\). Note however that the shrinkage need not exceed the size of any invariant open neighborhood of \(M_{<q}\) whose closure is contained within the domain of definition of \(\varPhi\).

For every point \(s\) of a smooth manifold \(S\) and open subset \(U\) of the product \(M\times S\) let \(U_s\subset M\) denote the open set \(\{x\in M:(x,s)\in U\}\). Let us call \(U\) \emph{invariant} if so is \(U_s\) for every \(s\). By a \emph{\(C^\infty\) parametric family} \(H=\{H(s)\in\Conn(\varGamma\mathbin|U_s)\}\) of \emph{local connections} on \(\varGamma\tto M\) indexed over \(S\) with domain \(U\subset M\times S\) we simply mean an ordinary connection \(H\) on the restriction over \(U\) of the product Lie groupoid \(\varGamma\times S\tto M\times S\). Notice that the horizontal lift of any such connection \(H\) is necessarily of the form \(\eta^H_{(g,s)}=\eta^{H(s)}_g\times\id_{T_sS}\), where \(H(s)\) stands for the connection on \(\varGamma\mathbin|U_s\tto U_s\) which \(H\) induces along the invariant submanifold \(U_s=U\cap(M\times\{s\})\subset U\), so \(H\) is completely determined by the associated family \(H(s)\sidetext(s\in S)\) of ``partial'' connections on \(\varGamma\tto M\) with ``variable domains''.

\begin{prob}\label{prob:16a.1.9} Let \(\varGamma\tto M\) be any Lie groupoid. Let \(S\) be any smooth manifold. Let \(U\),~\(V\) be open subsets of \(M\times S\) such that \(\overline U\subset V\), and let \(\varPhi=\{\varPhi(s)\in\Mcon(\varGamma\mathbin|V_s)\}\) be any \(C^\infty\) parametric family of local multiplicative connections on \(\varGamma\) indexed over \(S\) with domain of definition \(V\). What are the precise obstructions to extending \(\varPhi\mathbin|U\) to a \(C^\infty\) parametric family \(S\to\Mcon(\varGamma)\) of (globally defined) multiplicative connections on \(\varGamma\)? \end{prob}

As we anticipated in the course of the introduction the two questions below are special cases of Problem~\ref{prob:16a.1.9} of particular interest to us:
\begin{enumerate}
\def\labelenumi{(\alph{enumi})}%
\itemsep=0pt%
 \item What are the obstructions to the existence of multiplicative connections on \(\varGamma\)?
 \item What are the obstructions to deforming two multiplicative connections \(\varPhi_0\) and \(\varPhi_1\) on \(\varGamma\) smoothly into each other through multiplicative connections?
\end{enumerate}
The first question, of course, corresponds to the pa\-ram\-e\-ter-free case \(S=\ast\) with \(U=V=\emptyset\), the second one, to the case \(S=\R\) with, e.g., \(U=M\times(\R\smallsetminus[0,1])\emphpunct, V=M\times(\R\smallsetminus\{\frac 12\})\), and \(\varPhi(s)\) given for \(s<\frac 12\) by \(\varPhi_0\) and for \(s>\frac 12\) by \(\varPhi_1\).

Let us now confine attention to the situation of {\em proper}\/ \(\varGamma\). We contend that in this case we simply need to understand the answer to Problem~\ref{prob:16a.1.9} for \(\varGamma\) {\em regular}. (Then under the hypothesis of regularity the problem appears amenable to treatment by the methods of standard equivariant obstruction theory, a viewpoint which we try to substantiate further in \S\ref{sec:16a.8}.) Our strategy for the reduction of the problem to the regular case goes as follows. To begin with, we observe that at the expense of replacing \(\varGamma\) with the product groupoid \(\varGamma\times S\tto M\times S\) we can always assume that we are in the pa\-ram\-e\-ter-free case \(S=\ast\). We can then try to build an extension of \(\varPhi\mathbin|U\) inductively by means of Theorem~\ref{prop:14A.5.2}, starting with an arbitrary multiplicative connection defined around \(\overline U\cup M_0\) whose restriction over \(U\) coincides with \(\varPhi\mathbin|U\), and then extending this connection along the regular strata \(M_q\) of higher orbit dimension \(q\geq 1\), one dimension at the time.

Further in detail, imagine that you have constructed some extension \(\varPhi_q\) of \(\varPhi\mathbin|U\) over a suitable open neighborhood \(V_q\) of \(\overline U\cup M_{<q}\) within \(M\). (When \(q=0\), you may of course take \(V_0=V\emphpunct, \varPhi_0=\varPhi\).) Suppose that for all proper, rank~\(q\), regular Lie groupoids of a certain type you know what the obstructions of Problem~\ref{prob:16a.1.9} are, and that (every invariant component of) the \(q\)-th regular stratum \(\varGamma\mathbin|M_q\tto M_q\) of \(\varGamma\) is a groupoid of that type. (We shall provide examples featuring complete descriptions of such obstructions in \S\ref{sec:16a.8} for \(q=1,2\).) Fix an arbitrary open neighborhood \(U_q\) of \(\overline U\cup M_{<q}\) so that \(\overline U_q\subset V_q\). Whenever the obstructions to the existence of a multiplicative connection \(\varTheta_q\in\Mcon(\varGamma\mathbin|M_q)\) whose restriction over \(M_q\cap U_q\) equals \(\varPhi_q\mathbin|U_q\cap M_q\) vanish, Theorem~\ref{prop:14A.5.2} applied to \(C=\overline U\cup M_{<q}\emphpunct, Z=M_q\emphpunct, \varPhi_q\mathbin|U_q\), and \(\varTheta_q\) implies that there ought to be \(\varPhi_{q+1}\in\Mcon(\varGamma\mathbin|V_{q+1})\) defined over some open neighborhood \(V_{q+1}\) of \((\overline U\cup M_{<q})\cup M_q=\overline U\cup M_{\leq q}\) which coincides with \(\varPhi_q\) around \(\overline U\cup M_{<q}\) and hence, a fortiori, extends \(\varPhi\mathbin|U\). You may then transfer the problem to the regular stratum of the next higher orbit dimension\textemdash until you eventually meet some non vanishing obstruction or else exhaust the whole of \(M\).

Notice that Theorem~\ref{prop:14A.5.2} also guarantees that in the above inductive process the germs of any two extensions \(\varPhi_{q+1}\) around \(\overline U\cup M_{\leq q}\) (for given \(\varPhi_q\) and \(\varTheta_q\)) have to be homotopic, so one's particular choice of \(\varPhi_{q+1}\) will not influence the outcome (feasibility) of the next step of the process. The choice of \(\varTheta_q\) however might. Question~(b) is therefore relevant even if one is only interested in Question~(a). The two questions are tightly interwoven.
% 30575 Nov 26 18:03

\section{First examples}\label{sec:16a.2}

In order to help the reader digest the results described above and get some intuition about the topics under discussion, we pause for a while and give a list of basic examples. More advanced examples will be supplied in the final section of the paper. We start by looking into the simple case of action groupoids, which is already instructive.

Let \(G\) be an arbitrary Lie group. Let \(\mathfrak g\) stand for its Lie algebra. Recall that the \emph{Mau\-rer\textendash Car\-tan form} on \(G\) is the Lie-al\-ge\-bra valued 1-form \(\omega:TG\to\mathfrak g\) given at each group element \(g\in G\) by \(\omega_g:=T_g{\tau_g}^{-1}:T_gG\simto T_1G=\mathfrak g\), where \(\tau_g\) denotes right translation by \(g\). The differential \(T_{(g,h)}m\) of the group multiplication law \(m:G\times G\to G\) at any pair of elements \(g,h\in G\) is easily seen to be given by the following expression.
\begin{equation}
\label{eqn:12B.9.10}
	T_{(g,h)}m={\omega_{gh}}^{-1}\circ(\omega_g\circ\pr_1+\Ad_G(g)\circ\omega_h\circ\pr_2):T_gG\oplus T_hG\to T_{gh}G
\end{equation}

\begin{exmp}[Action groupoids]\label{exmp:16a.2.1} Now suppose that our Lie group \(G\) is acting on some smooth manifold \(U\). Let \(\pr\) stand for the projection \(G\times U\to U\). An arbitrary connection \(H\) on the action groupoid \(G\ltimes U\tto U\) gives rise to a Lie-al\-ge\-bra valued map
\begin{equation*}
	X^H:\pr^*TU\to\mathfrak g
\end{equation*}
obtained by composing the connection horizontal lift \(\eta^H:\pr^*TU\to T(G\times U)=TG\times TU\), first, with the projection on the \(TG\)~factor, and then, with the Mau\-rer\textendash Car\-tan form of \(G\). Note that given any smooth map \(X:\pr^*TU\to\mathfrak g\) which fibrates into a family of linear maps \(X_{g,u}:T_uU\to\mathfrak g\sidetext(g\in G\emphpunct, u\in U)\) there is exactly one connection \(H\) on \(G\ltimes U\tto U\) for which \(X=X^H\).

Now, in virtue of \eqref{eqn:12B.9.10}, the condition of multiplicativity of \(H\) \eqref{eqn:16a.2a} can be reformulated as a system of cocycle equations to be satisfied by the linear maps \(X^H_{g,u}\):%
\begin{subequations}
\label{eqn:12B.9.15}
\begin{equation}
\label{eqn:12B.9.15a}
	\Ad_G(g)\circ X^H_{h,u}-X^H_{gh,u}+X^H_{g,hu}\circ\lambda^H_{h,u}=0.
\end{equation}
The condition of unitality of \(H\) can likewise be reformulated as
\begin{equation}
\label{eqn:12B.9.15b}
	X^H_{1,u}=0.
\end{equation}
\end{subequations}
We point out that the zero morphism is always trivially a solution to \eqref{eqn:12B.9.15}, hence the connection \(\varPhi\) on \(G\ltimes U\tto U\) characterized by the condition \(X^\varPhi=0\) is always multiplicative. \end{exmp}

\begin{exmp}[Bundles of compact abelian Lie groups]\label{exmp:16a.2.2} Next, let us restrict attention to the case of a trivial \(G\)~action; in such case, our action groupoid \(G\ltimes U\tto U\) is simply the trivial Lie-group bundle \(G\times U\) over \(U\) with fiber \(G\).

Given a multiplicative connection \(\varPhi\) on \(G\times U\), let us choose any base tangent vector \(\varv\in T_uU\) and consider the Lie-al\-ge\-bra valued function \(Z:G\to\mathfrak g\emphpunct, g\mapsto X^\varPhi_{g,u}\varv\). Since \(\lambda^H\) is trivial for every groupoid connection \(H\) on \(G\times U\), \eqref{eqn:12B.9.15a} tells that \(Z\) must be a \(1\)\mdash cocycle for the Lie-group cohomology of \(G\) with coefficients in the adjoint representation:
\begin{equation*}
	\Ad_G(g)Z(h)-Z(gh)+Z(g)=0.
\end{equation*}

If our Lie group \(G\) is commutative, so that \(\Ad_G\) is the trivial representation, \(1\)\mdash cocycles \(Z\) will be the same thing as \emph{additive} functions: \(Z(gh)=Z(g)+Z(h)\). Now if \(\alpha:\R\to G\) is an arbitrary one-pa\-ram\-e\-ter subgroup of \(G\) then the composition \(Z\circ\alpha\) is a continuous additive map of \(\R\) into \(\mathfrak g\) and hence is an \(\R\)\mdash linear map. If the image of \(\alpha\) is a compact subgroup of \(G\) then the image of this subgroup under \(Z\) is a compact linear subspace of \(\mathfrak g\) and hence is necessarily equal to \(\{0\}\). We conclude that \(Z(g)=0\) for every \(g\) lying on a compact one-pa\-ram\-e\-ter subgroup of \(G\).

If our Lie group \(G\) is further compact (besides being commutative) then the set of all those group elements that lie on compact one-pa\-ram\-e\-ter subgroups of \(G\) is dense within the identity component \(G_0\) of \(G\). Every additive function \(Z\) must then vanish identically over \(G_0\). Being additive, \(Z\) can take only finitely many values, one for each component of \(G\). But the image of \(Z\) also has to be a \(\Z\)\mdash sublattice of \(\mathfrak g\); it can therefore only be zero. We conclude that on any compact commutative Lie group, the only \(1\)\mdash cocycle is the zero function.

So, in summary, the only multiplicative connection \(\varPhi\) on any trivial Lie bundle \(G\times U\) with compact abelian fiber \(G\) is the one characterized by the condition \(X^\varPhi=0\). It follows at once that every locally trivial bundle of compact abelian Lie groups admits exactly one multiplicative connection. \end{exmp}

\begin{exmp}[Pair groupoids]\label{exmp:16a.2.3} On a pair groupoid \(M\times M\tto M\) the multiplicative connections are easily recognized to be the same as the global trivializations of the tangent bundle \(TM\). Thus \(M\times M\tto M\) possesses multiplicative connections iff \(M\) is parallelizable. In fact, the pair groupoids over non-par\-al\-lel\-iz\-able manifolds constitute the simplest examples of proper Lie groupoids which do not admit any multiplicative (not even non-de\-gen\-er\-ate) connections. \end{exmp}

\begin{exmp}[Counterexample]\label{npar:14A.6.9} The nonexistence of multiplicative connections can be due to a variety of independent reasons. In our next example, \(\varGamma\) is a compact Lie groupoid with orbits of dimension~\(\leq 1\) over the two-sphere \(S^2\) which does not admit multiplicative connections, even though its restriction over the complement of each orbit does.

Let the non-triv\-i\-al element of \(\Z/2\) act as inversion on the circle group \(\mathbb T:=\{z\in\C:\abs z=1\}\). Let the corresponding semi-di\-rect product (Lie) group \(\mathbb T\rtimes\Z/2\) act on complex numbers \(\zeta\) by the rule $(z,-1)\zeta:=z\bar\zeta$. Also let the product group $\mathbb T\times\Z/2$ act on complex numbers by the rule $(z,\nu)\zeta:=z\zeta$. Glue the two action groupoids $(\mathbb T\rtimes\Z/2)\ltimes\C\tto\C$ and $(\mathbb T\times\Z/2)\ltimes\C\tto\C$ together along the Lie-group\-oid isomorphism
\begin{equation*}
	(\mathbb T\rtimes\Z/2)\ltimes\C\mathord\smallsetminus 0\ni(z,\nu;\zeta)\mapsto(z\zeta^\nu\bar\zeta/\abs\zeta^2,\nu;\zeta/\abs\zeta^2)\in(\mathbb T\times\Z/2)\ltimes\C\mathord\smallsetminus 0.
\end{equation*}
The resulting Lie groupoid \(\varGamma\tto S^2\) is proper, hence compact. Its orbits, namely the north pole, the south pole, and the regular level sets for the ``latitude'' function, are connected. However, the effect of the isotropic arrow $(1,-1;1)$ under any multiplicative connection on $(\mathbb T\rtimes\Z/2)\ltimes\C\tto\C$ must be non-triv\-i\-al (of period two), whereas its effect under any multiplicative connection on $(\mathbb T\times\Z/2)\ltimes\C\tto\C$ must be trivial. Hence $\varGamma\tto S^2$ cannot admit any (globally defined) multiplicative connection. \end{exmp}

\begin{exmp}[Counterexample]\label{npar:14A.2.4} Our next example features a {\em source-con\-nect\-ed}, compact Lie groupoid over the three-sphere \(S^3\) which does not admit multiplicative connections.

Let the Lie group $\SO(3)\times\R/\Z$ act on $\R^3$ by the law \[%
	(P,[\theta])\cdot\left[%
\begin{smallmatrix}
                    x
\\[.4\baselineskip] y
\\[.4\baselineskip] z
\end{smallmatrix}\right]:=P\left[%
\begin{smallmatrix}
                    x
\\[.4\baselineskip] y
\\[.4\baselineskip] z
\end{smallmatrix}\right].
\] Identify $\R^3\mathord\smallsetminus 0$ with the product $\R_{>0}\times S^2$~(``spherical coordinates''). Then, consider the following automorphism of the action groupoid $\bigl(\SO(3)\times\R/\Z\bigr)\ltimes\R^3\mathord\smallsetminus 0\tto\R^3\mathord\smallsetminus0$,
\begin{equation}
\label{eqn:14A.2.1}
	(P,[\theta];r,\left[%
\begin{smallmatrix}
                    x
\\[.4\baselineskip] y
\\[.4\baselineskip] z
\end{smallmatrix}\right])\mapsto(P\exp\left(2\pi\theta\left[%
\begin{smallmatrix}
                     0 &-z & y
\\[.3\baselineskip]  z & 0 &-x
\\[.3\baselineskip] -y & x & 0
\end{smallmatrix}\right]\right),[\theta]\emphpunct*;1/r,\left[%
\begin{smallmatrix}
                    x
\\[.4\baselineskip] y
\\[.4\baselineskip] z
\end{smallmatrix}\right]).
\end{equation}
(That this is a well-de\-fined Lie-group\-oid homomorphism, follows from the basic relation \[%
	P\exp\left[%
\begin{smallmatrix}
                     0 &-z & y
\\[.3\baselineskip]  z & 0 &-x
\\[.3\baselineskip] -y & x & 0
\end{smallmatrix}\right]P^{-1}=\exp\left[%
\begin{smallmatrix}
                     0  &-z' & y'
\\[.2\baselineskip]  z'	& 0  &-x'
\\[.2\baselineskip] -y' & x' & 0
\end{smallmatrix}\right]\emphpunct, \text{where}\emphpunct{ }\left[%
\begin{smallmatrix}
                    x'
\\[.4\baselineskip] \smash[t]{y'}
\\[.4\baselineskip] \smash[t]{z'}
\end{smallmatrix}\right]=P\left[%
\begin{smallmatrix}
                    x
\\[.4\baselineskip] y
\\[.4\baselineskip] z
\end{smallmatrix}\right]\text{.)}
\] The Lie groupoid \(\varGamma\tto S^3\) obtained by gluing two copies of the action groupoid $\bigl(\SO(3)\times\R/\Z\bigr)\ltimes\R^3$ along the automorphism \eqref{eqn:14A.2.1} is compact and source connected. Its orbits are the two poles and the ``isolatitudinal'' spheres. We leave it as an exercise for the reader to confirm that \(\varGamma\tto S^3\) can admit no multiplicative connections, even though\textemdash like in our previous example\textemdash its restriction over the complement of each orbit does. \end{exmp}

\subsection*{The regular case}

The remainder of this section is devoted to a preliminary survey of effective connections in the regular case. By a \emph{longitudinal representation} of a regular Lie groupoid we intend a representation on the vector distribution tangent to the orbit foliation associated with the groupoid. As it is not hard to imagine, effective connections on any regular groupoid bear a close relationship to longitudinal representations. We start by making such relationship precise.

\begin{lem}\label{lem:14A.3.1} Let\/ \(\varGamma\tto M\) be an arbitrary regular Lie groupoid. Let\/ \(\varLambda\subset TM\) denote the longitudinal bundle of\/ \(\varGamma\tto M\). Suppose\/ \(\rho:\varGamma\to\GL(TM)\) is a tangent representation of\/ \(\varGamma\tto M\) having the property that for all arrows\/ \(g\) the linear maps\/ \(\rho(g):T_{sg}M\simto T_{tg}M\) carry longitudinal subspaces\/ \(\varLambda_{sg}\subset T_{sg}M\) into longitudinal subspaces\/ \(\varLambda_{tg}\subset T_{tg}M\) and the resulting quotient linear maps\/ \(T_{sg}M\mathbin/\varLambda_{sg}\simto T_{tg}M\mathbin/\varLambda_{tg}\) coincide with the\/ {\em infinitesimal effects} of the arrows~\cite[\S 1]{Tre6}. Also suppose that\/ \(\varPhi\) is a groupoid connection defined on the restriction of\/ \(\varGamma\tto M\) over some open subset\/ \(V\subset M\) whose effect\/ \(\lambda^\varPhi\) coincides with\/ \(\rho\mathbin|V\). Then, for every open subset\/ \(U\subset M\) such that\/ \(\overline U\subset V\), there is some connection on\/ \(\varGamma\tto M\) whose effect coincides with\/ \(\rho\) and whose restriction over\/ \(U\) coincides with\/ \(\varPhi\mathbin|U\). \end{lem}

\begin{proof} Endow \(M\) with some Riemannian metric. For every \(x\in M\), let \(N_x:={\varLambda_x}^\bot\subset T_xM\) stand for the orthogonal complement of \(\varLambda_x\subset T_xM\) with respect to the metric. The matrix of \(\rho\) relative to the vec\-tor-bun\-dle decomposition \(TM\cong\varLambda\oplus N\) reads \[%
	\rho(g)=:%
\begin{pmatrix}
	\alpha(g) & \ast
\\	0         & \nu(g)
\end{pmatrix},\] where \(\nu(g):N_{sg}\simto N_{tg}\) matches the infinitesimal effect of \(g\) under the canonical vec\-tor-bun\-dle identification \(N\cong TM/\varLambda\).

Next, pick an arbitrary connection \(H\) on \(\varGamma\tto M\) whose restriction over \(U\) coincides with \(\varPhi\mathbin|U\), and consider the difference \[%
	\delta^{H,\rho}(g):=\lambda^H(g)-\rho(g)=%
\begin{pmatrix}
	\lambda^H_\varLambda(g) & \ast
\\	0                       & \lambda^H_N(g)
\end{pmatrix}-%
\begin{pmatrix}
	\alpha(g) & \ast
\\	0         & \nu(g)
\end{pmatrix}=%
\begin{pmatrix}
	\ast & \ast
\\	0    & 0
\end{pmatrix},\] to be regarded as a vec\-tor-bun\-dle morphism \(\delta^{H,\rho}:s^*TM\to t^*\varLambda\). Now \(\d t:\ker{\d s}\to t^*\varLambda\) is an epimorphism of vector bundles over \(\varGamma\) and, therefore, splits. Fix an arbitrary splitting \(\xi:t^*\varLambda\to\ker{\d s}\), and set \[%
	\eta^\varPsi_g:=\eta^H_g-\xi_g\circ\delta^{H,\rho}_g.
\] Clearly this is the horizontal lift for some effective connection, say, \(\varPsi\), with effect \(\lambda^\varPsi=\rho\). Moreover since \(\delta^{H,\rho}\) vanishes identically over \(U\), it follows that \(\varPsi\) coincides with \(H\) and, hence, \(\varPhi\) over \(U\). \end{proof}

As a consequence of the lemma, every regular Lie groupoid \(\varGamma\tto M\) which admits representations on its own longitudinal bundle \(\varLambda\subset TM\) necessarily also admits effective connections. Indeed if one is given any such representation \(\alpha:\varGamma\to\GL(\varLambda)\), one can fix an arbitrary Riemannian metric on \(M\) and then, in the notations of the preceding proof, define \(\rho:\varGamma\to\GL(TM)\) with respect to the decomposition \(TM\cong\varLambda\oplus N\) to be \[%
	\rho:=%
\begin{pmatrix}
	\alpha & 0
\\	0      & \nu
\end{pmatrix}.\]

\begin{exmp}\label{prop:14A.3.2} Every regular Lie groupoid \(\varGamma\tto M\) whose longitudinal bundle \(\varLambda\subset TM\) is trivializable admits effective connections, because any vec\-tor-bun\-dle trivialization \(\tau:\R^q\times M\simto\varLambda\) gives rise to a longitudinal representation, \(\varGamma\to\GL(\varLambda)\emphpunct, g\mapsto\tau_{tg}\circ{\tau_{sg}}^{-1}\). \end{exmp}

\begin{exmp}\label{cor:14A.3.3} A transitive Lie groupoid \(\varGamma\tto M\) over a parallelizable manifold \(M\) always admits effective connections. \end{exmp}

\begin{exmp}[Counterexample]\label{npar:14A.3.4} If in the last two examples \(\varGamma\tto M\) is also taken to be proper, then, by Proposition~\ref{prop:12B.11.7}, \(\varGamma\tto M\) admits multiplicative connections. In general, in the absence of properness, the existence of effective connections does not\textemdash even in the regular case\textemdash imply that of multiplicative connections, as our next example shows.

Let us define \(\varGamma:=\varDelta/K\to\R\) to be the quotient of the trivial Lie bundle \(\varDelta:=\R\times\R\xto{\pr_2}\R\) by the \'etale \emph{Lie kernel}~(cf.~\cite[appendix]{Tre6}) \[%
	K:=\{\mkern 1mu(2\pi n/t,t):n\in\Z\et t\in\mathopen]0,\infty\mathclose[\mkern 1mu\}\cup 0\times\R.
\] If \(\varGamma\) admitted a multiplicative connection, the inverse image of that connection along the quotient projection (local diffeomorphism) \(\varDelta\to\varGamma\) would be a multiplicative connection on \(\varDelta\). Now, the restriction of \(\varGamma\) over the positive half-line \(\mathopen]0,\infty\mathclose[\) is a circle bundle, so, by Example~\ref{exmp:16a.2.2}, it admits a unique multiplicative connection, whose inverse image along the quotient projection necessarily coincides with the vector distribution tangent to the hyperbolae \(t\mapsto(\theta/t,t)\sidetext(\theta\in\R)\). But no groupoid connection on \(\varDelta\) could possibly be an extension of that vector distribution. \end{exmp}
% 15935 Nov 30 08:26

\section{The averaging operator}\label{sec:16a.3}

Throughout the present section and in the next four, we shall be dealing with a fixed (but otherwise arbitrary) proper Lie groupoid, say, \(\varGamma\tto M\). Our goal in this section is to provide a proof of Proposition~\ref{prop:12B.11.7} and, with that, establish the equivalence of the following two properties\emphpunct: (i)~\em \(\varGamma\) admits an effective connection\em\emphpunct; (ii)~\em \(\varGamma\) admits a multiplicative connection\em. As a matter of fact, the definition of our averaging operator \eqref{eqn:16a.8} was originally inspired to us by the deformation argument used in \cite[Sections 2.3~to 2.5]{CS} to give a new proof of the linearization theorem for proper Lie groupoids. A rather special case of our definition is also implicit in Weinstein's proof of the local triviality of proper Lie bundles \cite[Theorem~7.1]{Wein}. The history of these ideas goes back at least as far as to Palais and Stewart \cite{PS}.

Just as a warm-up, before starting proving Proposition~\ref{prop:12B.11.7}, let us verify the assertions that precede its statement; the computations are straightforward, but exceptionally in this case we give the details because we think they are useful to further clarify our definitions. Let us begin with the equation \(\d s\circ\avg\eta^H=\id_{s^*TM}\). We have
\begin{align*}
	T_gs\circ\avg\eta^H_g&=\integral_{tk=sg}T_gs\circ(\eta^H_{gk}\div\eta^H_k)\circ(\lambda^H_k)^{-1}dk
\\	                     &=\integral_{tk=sg}T_kt\circ\eta^H_k\circ(\lambda^H_k)^{-1}dk
\\	                     &=\integral_{tk=sg}\lambda^H_k\circ(\lambda^H_k)^{-1}dk
\\	                     &=\integral_{tk=sg}\id_{T_{tk}M}dk=\id_{T_{sg}M}\text.
\end{align*}
Next, we must check that \(\avg\eta^H_{1x}\) equals \(T_x1\) for every \(x\in M\) (unitality). We have \((\eta^H_k\div\eta^H_k)\circ(\lambda^H_k)^{-1}=T_{tk}1\circ T_kt\circ\eta^H_k\circ(\lambda^H_k)^{-1}=T_{tk}1\circ\lambda^H_k\circ(\lambda^H_k)^{-1}=T_{tk}1\), whence \(\avg\eta^H_{1x}=\integral_{tk=x}T_{tk}1dk=T_x1\). Finally, there remains to check Eq.~\eqref{eqn:12B.11.7} giving the effect of \(\avg\eta^H\):
\begin{align*}
	T_gt\circ\avg\eta^H_g&=\integral_{tk=sg}T_gt\circ(\eta^H_{gk}\div\eta^H_k)\circ(\lambda^H_k)^{-1}dk
\\	                     &=\integral_{tk=sg}T_{gk}t\circ\eta^H_{gk}\circ(\lambda^H_k)^{-1}dk
\\	                     &=\integral_{tk=sg}\lambda^H_{gk}\circ(\lambda^H_k)^{-1}dk\text.
\end{align*}

\subsection*{Multiplicativity equations relative to a ``background'' connection}

When dealing with multiplicative connections on action groupoids (cf.~Example~\ref{exmp:16a.2.1}) we exploited the natural splitting of the tangent bundle of an action groupoid into its vertical and horizontal subbundles for the purpose of rewriting the condition of multiplicativity in terms of the vertical component of a connection. That proved to be useful, then, from the point of view of computations (cf.~Example~\ref{exmp:16a.2.2}). We now want to carry out an analogous rewriting of the condition of multiplicativity in the more general context of the present section.

Let \(\mathfrak g\to M\) denote the Lie algebroid of \(\varGamma\tto M\) i.e.~the vector bundle over \(M\) given by \(1^*\ker{\d s}\). For every arrow \(g\in\varGamma\) the right translation map \({\tau_g}^{-1}:\varGamma^{sg}\approxto\varGamma^{tg}\emphpunct, h\mapsto hg^{-1}\) is a diffeomorphism which makes \(g\) correspond to \(1_{tg}\). The invertible linear maps \(\omega_g\bydef T_g{\tau_g}^{-1}:T_g\varGamma^{sg}\simto T_{1tg}\varGamma^{tg}\) fit together into an isomorphism of vector bundles over \(\varGamma\),
\begin{equation}
\label{eqn:12B.9.5}
	\omega:\ker{\d s}\simto t^*\mathfrak g
\end{equation}
which generalizes the familiar ``Maurer\textendash Cartan form'' from Lie group theory.

Even though in the case of a general proper Lie groupoid there may be no global\textemdash let alone, canonical\textemdash trivializations of the source map available, we can at least always find such trivializations at the {\em infinitesimal}\/ level. Namely, let us randomly fix some connection \(\varPhi\) on \(\varGamma\tto M\). We shall henceforth refer to \(\varPhi\) as our ``background'' connection. The choice of \(\varPhi\) defines a splitting of the tangent bundle of \(\varGamma\) into a ``vertical'' and a ``horizontal'' component (with respect to the tangent source map) as in the diagram below, where \(\pi^\varPhi\bydef\id_{T\varGamma}-\eta^\varPhi\circ\d s:T\varGamma\to\ker{\d s}\) indicates the \emph{vertical projection} associated with \(\varPhi\).
\begin{equation}
\label{eqn:12B.11.8}
	\sigma^\varPhi\bydef(\omega\circ\pi^\varPhi,\d s):\mkern-\thickmuskip%
\xymatrix@C=-1em@M=.278em{%
 T\varGamma
 \ar[dr]_{\d s}
 \ar@{}[rr]|-*{\vphantom g\simto\vphantom g}
 &	&	t^*\mathfrak g\oplus s^*TM
		\ar[dl]^(.48){\pr_2}
\\&	s^*TM}
\end{equation}
This splitting determines an analogous decomposition of the tangent division map \(Tq_\div:T\varGamma_\div\sidetext(=T\varGamma\ftimes{Ts}{Ts}T\varGamma)\to T\varGamma\); namely, for each divisible pair of arrows \((g,h)\in\varGamma_\div\), there is a linear map of \(\mathfrak g_{tg}\oplus\mathfrak g_{th}\oplus T_{sg=sh}M\) into \(\mathfrak g_{tg}\) [resp.~\(T_{th}M\)] hereafter denoted \(\dot q^\varPhi_{g,h}\) [resp.~\(\dot s^\varPhi_{g,h}\)] characterized through the commutativity of the following diagram.
\begin{equation}
\label{eqn:12B.11.9}
 \begin{split}
\xymatrix@C=1.333em{%
 T_g\varGamma\ftimes{T_gs}{T_hs}T_h\varGamma \ar@{=}[d]
 \ar[r]^-*-<.7ex>{\sim}
 &	(\mathfrak g_{tg}\oplus T_{sg}M)\ftimes{\pr_2}{\pr_2}(\mathfrak g_{th}\oplus T_{sh}M)
	\ar@{=}[d]
\\ T_{(g,h)}\varGamma_\div
 \ar[d]^{T_{(g,h)}q_\div}
 &	\mathfrak g_{tg}\oplus\mathfrak g_{th}\oplus T_{sg=sh}M
	\ar@{.>}[d]^{(\dot q^\varPhi_{g,h},\dot s^\varPhi_{g,h})}
\\ T_{gh^{-1}}\varGamma
 \ar[r]^-*-<.7ex>{\sim}
 &	\mathfrak g_{tg}\oplus T_{th}M
}\end{split}
\end{equation}

Evidently, \(\dot s^\varPhi_{g,h}\) factors as
\begin{equation}
\label{eqn:12B.11.10}
\xymatrix@C=2.333em{%
 \mathfrak g_{tg}\oplus\mathfrak g_{th}\oplus T_{sh}M \ar[r]^-\pr
 &	\mathfrak g_{th}\oplus T_{sh}M \ar[r]^-{(\sigma^\varPhi_h)^{-1}}
	&	T_h\varGamma \ar[r]^-{T_ht}
		&	T_{th}M\text,
}\end{equation}
where \(\pr\) stands for the projection \((X,Y,\varv)\mapsto(Y,\varv)\). The expression \(\dot s^\varPhi_{g,h}(X,Y,\varv)\) must then be independent of \(g\) and \(X\): we may abbreviate it into \(\dot s^\varPhi_h(Y,\varv)\). Let us introduce a bunch of related shorthand, of which we are going to make use at our convenience.%
\begin{subequations}
\label{eqn:12B.11.11}
\begin{gather}
\label{eqn:12B.11.11a}
 \makebox[9pc][c]{$\dot q^\varPhi_\updownarrow(g,h):=\dot q^\varPhi_{g,h}(-,-,0)$}
\qquad	\makebox[9pc][c]{$\dot s^\varPhi_\updownarrow(h):=\dot s^\varPhi_h(-,0)$}
\\%
\label{eqn:12B.11.11b}
 \makebox[9pc][c]{$\dot q^\varPhi_\leftrightarrow(g,h):=\dot q^\varPhi_{g,h}(0,0,-)$}
\qquad	\makebox[9pc][c]{$\dot s^\varPhi_\leftrightarrow(h):=\dot s^\varPhi_h(0,-)$}
\end{gather}
\end{subequations}

Any other connection \(H\) on \(\varGamma\tto M\) will entirely be encoded into its \emph{vertical component}\/~\(X^{H:\varPhi}\) relative to the chosen ``background'' connection \(\varPhi\):
\begin{equation}
\label{eqn:12B.11.12}
	X^{H:\varPhi}\bydef\omega\circ\pi^\varPhi\circ\eta^H:s^*TM\to t^*\mathfrak g.
\end{equation}

By the above definitions, we have \(\sigma^\varPhi_g(\eta^H_g\varv)=([\omega_g\circ\pi^\varPhi_g]\eta^H_g\varv\emphpunct*,(T_gs)\eta^H_g\varv)=(X^{H:\varPhi}_g\varv,\varv)\) for all \(g\in\varGamma\emphpunct, \varv\in T_{sg}M\). Since \(\sigma^\varPhi_g\) is an invertible linear map, the multiplicativity condition \eqref{eqn:12B.11.4} on \(H\) for \((g,h)\in\varGamma_\div\emphpunct, \varv\in T_{sg=sh}M\) will be satisfied if, and only if,
\begin{align*}
	(X^{H:\varPhi}_{gh^{-1}}\lambda^H_h\varv,\lambda^H_h\varv)
		&=\sigma^\varPhi_{gh^{-1}}(\eta^H_{gh^{-1}}\lambda^H_h\varv)
\\		&=\sigma^\varPhi_{gh^{-1}}(\eta^H_g\varv\div\eta^H_h\varv)
\\		&=\sigma^\varPhi_{gh^{-1}}\bigl(T_{(g,h)}q_\div(\eta^H_g\varv,\eta^H_h\varv)\bigr)
\\		&=\bigl(\dot q^\varPhi_{g,h}(X^{H:\varPhi}_g\varv,X^{H:\varPhi}_h\varv,\varv)\emphpunct*,\dot s^\varPhi_h(X^{H:\varPhi}_h\varv,\varv)\bigr)
	&&\text{by \eqref{eqn:12B.11.9}}
\end{align*}
Suppressing \(\varv\) from the last identity, and making use of \eqref{eqn:12B.11.11}, we get the following couple of equations.%
\begin{subequations}
\label{eqn:12B.11.13}
\begin{gather}
\label{eqn:12B.11.13a}
	X^{H:\varPhi}_{gh^{-1}}\lambda^H_h=\dot q^\varPhi_\updownarrow(g,h)(X^{H:\varPhi}_g,X^{H:\varPhi}_h)+\dot q^\varPhi_\leftrightarrow(g,h)
\\%
\label{eqn:12B.11.13b}
	\lambda^H_h=\dot s^\varPhi_\updownarrow(h)X^{H:\varPhi}_h+\dot s^\varPhi_\leftrightarrow(h)
\end{gather}
\end{subequations}
We observe that the second of these equations is a tautology; indeed, by the above remark to the effect that \(\dot s^\varPhi_h(Y,\varv)\) equals \((T_ht)(\sigma^\varPhi_h)^{-1}(Y,\varv)\), cf.~\eqref{eqn:12B.11.10}, we have
\begin{align*}
 \lambda^H_h-\dot s^\varPhi_h(X^{H:\varPhi}_h,\id)
		&=T_ht\circ\eta^H_h-T_ht\circ(\sigma^\varPhi_h)^{-1}(X^{H:\varPhi}_h,\id)
\\		&=T_ht\circ(\eta^H_h-\eta^H_h)=0.
\end{align*}
The multiplicativity condition \eqref{eqn:12B.11.4} is therefore tantamount to the following single equation, which one obtains by substituting \eqref{eqn:12B.11.13b} into \eqref{eqn:12B.11.13a}, and which only involves the \(\varPhi\)\mdash vertical component of \(H\).
\begin{equation}
\label{eqn:12B.11.14}
	\dot q^\varPhi_\updownarrow(g,h)(X^{H:\varPhi}_g,X^{H:\varPhi}_h)=X^{H:\varPhi}_{gh^{-1}}\circ\bigl(\dot s^\varPhi_\leftrightarrow(h)+\dot s^\varPhi_\updownarrow(h)X^{H:\varPhi}_h\bigr)-\dot q^\varPhi_\leftrightarrow(g,h)
\end{equation}

In \eqref{eqn:12B.11.14}, the two ``horizontal'' terms \(\dot q^\varPhi_\leftrightarrow(\mathellipsis)\) and \(\dot s^\varPhi_\leftrightarrow(\mathellipsis)\) can be given slightly more intuitive expressions, as follows. For the first, since \(\dot s^\varPhi_h(0,\varv)=(T_ht)(\sigma^\varPhi_h)^{-1}(0,\varv)=(T_ht)\eta^\varPhi_h\varv\) for \(\varv\in T_{sh}M\), we see that%
\begin{subequations}
\begin{equation}
\label{eqn:12B.11.15}
	\dot s^\varPhi_\leftrightarrow(h)=\lambda^\varPhi_h.
\end{equation}
As to the second, assuming that \(\varPhi\) is {\em non-de\-gen\-er\-ate}, and letting \(\Delta^\varPhi\bydef q_\div^*\omega\circ q_\div^*\pi^\varPhi\circ\delta^\varPhi:s_\div^*TM\to q_\div^*t^*\mathfrak g\) denote the \(\varPhi\)\mdash vertical component of the ``division cocycle'' \eqref{eqn:12B.11.5} of \(\varPhi\), it is equally easy to see that
\begin{equation}
\label{eqn:12B.11.16}
	\dot q^\varPhi_\leftrightarrow(g,h)=\Delta^\varPhi(g,h)\lambda^\varPhi_h.
\end{equation} \end{subequations}

For any choice of a {\em non-de\-gen\-er\-ate}\/ ``background'' connection \(\varPhi\) on \(\varGamma\tto M\), a {\em unital}\/ connection \(H\) on \(\varGamma\tto M\) will thus be multiplicative if, and only if, its vertical component \(X^{H:\varPhi}\) relative to \(\varPhi\) satisfies the following equation for every divisible pair of arrows \((g,h)\).
\begin{equation}
\label{eqn:12B.11.17}
	\dot q^\varPhi_\updownarrow(g,h)(X^{H:\varPhi}_g,X^{H:\varPhi}_h)=X^{H:\varPhi}_{gh^{-1}}\circ\bigl(\lambda^\varPhi_h+\dot s^\varPhi_\updownarrow(h)X^{H:\varPhi}_h\bigr)-\Delta^\varPhi(g,h)\lambda^\varPhi_h
\end{equation}

\subsection*{Cocycle equations for the ``background'' connection}

Let \(g,h,k\in\varGamma\) satisfy \(sg=sh=sk\). Then
\begin{equation*}
	q_\div(g,k)=q_\div\bigl(q_\div(g,h)\emphpunct*,q_\div(k,h)\bigr).
\end{equation*}
If for every pair of indices \(i\neq j\in\{1,2,3\}\) we let \(q_{ij}\) denote the map of \(\varGamma\ftimes ss\varGamma\ftimes ss\varGamma\) into \(\varGamma\) given by \((g_1,g_2,g_3)\mapsto q_\div(g_i,g_j)\), we can rewrite the last identity more succinctly as
\begin{equation*}
	q_{13}=q_\div(q_{12},q_{32}).
\end{equation*}
Differentiating the latter identity at \((g,h,k)\), and taking into account the obvious relations \(T_{(g_1,g_2,g_3)}q_{ij}=T_{(g_i,g_j)}q_\div\circ\pr_{ij}\), where \(\pr_{ij}\) denotes the projection \((\varw_1,\varw_2,\varw_3)\mapsto(\varw_i,\varw_j)\), we get
\begin{equation*}
	(T_{(g,k)}q_\div)\pr_{13}=(T_{(g\smash[t]{h^{-1}},k\smash[t]{h^{-1}})}q_\div)\bigl((T_{(g,h)}q_\div)\pr_{12},(T_{(k,h)}q_\div)\pr_{32}\bigr).
\end{equation*}

After composing to the left with the invertible linear map \(\sigma^\varPhi_{g\smash[t]{k^{-1}}}\) and to the right with the linear map \((\eta^\varPhi_g,\eta^\varPhi_h,\eta^\varPhi_k)\), and making repeated use of the commutativity of the diagram \eqref{eqn:12B.11.9}, we obtain the following pair of equations for \(\varv\in T_{sg=sh=sk}M\).
\begin{gather*}
	\dot q^\varPhi_{g,k}(0,0,\varv)=\dot q^\varPhi_{g\smash[t]{h^{-1}},k\smash[t]{h^{-1}}}\bigl(\dot q^\varPhi_{g,h}(0,0,\varv)\emphpunct*,\dot q^\varPhi_{k,h}(0,0,\varv)\emphpunct*,\dot s^\varPhi_h(0,\varv)\bigr)
\\	\dot s^\varPhi_k(0,\varv)=\dot s^\varPhi_{k\smash[t]{h^{-1}}}\bigl(\dot q^\varPhi_{k,h}(0,0,\varv)\emphpunct*,\dot s^\varPhi_h(0,\varv)\bigr)
\end{gather*}
Recalling our shorthand \eqref{eqn:12B.11.11} and the identity \eqref{eqn:12B.11.15} and (for \(\varPhi\) non-de\-gen\-er\-ate) \eqref{eqn:12B.11.16}, we may rewrite these equations as follows.
\begin{align*}
\begin{split}
 \Delta^\varPhi(g,k)\lambda^\varPhi_k\varv
		&=\dot q^\varPhi_\updownarrow(gh^{-1},kh^{-1})\bigl(\dot q^\varPhi_\leftrightarrow(g,h)\varv\emphpunct*,\dot q^\varPhi_\leftrightarrow(k,h)\varv\bigr)
\\		&\justify+\dot q^\varPhi_\leftrightarrow(gh^{-1},kh^{-1})\dot s^\varPhi_\leftrightarrow(h)\varv
\end{split}\\%
\begin{split}
		&=\dot q^\varPhi_\updownarrow(gh^{-1},kh^{-1})\bigl(\Delta^\varPhi(g,h)\lambda^\varPhi_h\varv,\Delta^\varPhi(k,h)\lambda^\varPhi_h\varv\bigr)
\\		&\justify+[\Delta^\varPhi(gh^{-1},kh^{-1})\lambda^\varPhi_{k\smash[t]{h^{-1}}}]\lambda^\varPhi_h\varv
\end{split}
\\ \lambda^\varPhi_k\varv
		&=\dot s^\varPhi_\updownarrow(kh^{-1})\dot q^\varPhi_\leftrightarrow(k,h)\varv+\dot s^\varPhi_\leftrightarrow(kh^{-1})\dot s^\varPhi_\leftrightarrow(h)\varv
\\		&=\dot s^\varPhi_\updownarrow(kh^{-1})\Delta^\varPhi(k,h)\lambda^\varPhi_h\varv+\lambda^\varPhi_{k\smash[t]{h^{-1}}}\lambda^\varPhi_h\varv
\end{align*}
After suppressing \(\varv\) from these equations and setting \(kh^{-1}=:h'\) in the second of them, we are left with the following tautological expressions, which we call ``cocycle equations''.%
\begin{subequations}
\label{eqn:12B.11.18}
\begin{gather}
\label{eqn:12B.11.18a}
\begin{split}
 \dot q^\varPhi_\updownarrow(gh^{-1},kh^{-1})\bigl(\Delta^\varPhi(g,h),\Delta^\varPhi(k,h)\bigr)\lambda^\varPhi_h
		&=\Delta^\varPhi(g,k)\lambda^\varPhi_k
\\		&\justify-\Delta^\varPhi(gh^{-1},kh^{-1})\lambda^\varPhi_{k\smash[t]{h^{-1}}}\lambda^\varPhi_h
\end{split}
\\%
\label{eqn:12B.11.18b}
 \lambda^\varPhi_{h'h}-\lambda^\varPhi_{h'}\lambda^\varPhi_h=\dot s^\varPhi_\updownarrow(h')\Delta^\varPhi(h'h,h)\lambda^\varPhi_h
\end{gather}
\end{subequations}

\subsection*{Proof of Proposition~\ref{prop:12B.11.7}}

We have to confirm the validity of Eq.~\eqref{eqn:12B.11.17} for the (unital) connection \(H:=\avg\varPhi\) relative to the given (non-de\-gen\-er\-ate) ``background'' connection \(\varPhi\). It will be convenient to abridge the expressions \(\lambda^{\avg\varPhi}\) and \(X^{\avg\varPhi:\varPhi}\) into \(\avg\lambda\) and \(\avg X\) respectively; we shall moreover systematically suppress `\(\varPhi\)'~superscripts. For any divisible pair of arrows \((g,h)\), we have%
\begin{align*}
 \dot q_\updownarrow(g,h)(\avg X_g,\avg X_h)
		&=\dot q_\updownarrow(g,h)\circ\integral_{tk=sg=sh}\bigl(\Delta(gk,k),\Delta(hk,k)\bigr)dk
\\		&=\integral\dot q_\updownarrow(g,h)\bigl(\Delta(gk,k),\Delta(hk,k)\bigr)dk
\\	\overset{\txt{[by \eqref{eqn:12B.11.18a}: ]}}{\phantom=}
		&=\integral[\Delta(gk,hk)\lambda_{hk}(\lambda_k)^{-1}-\Delta(g,h)\lambda_h]dk
\\%
\begin{split}
	\overset{\txt{[by \eqref{eqn:12B.11.18b}: ]}}{\phantom=}
		&=\integral\Delta(gk,hk)\circ\bigl(\dot s_\updownarrow(h)\Delta(hk,k)+\lambda_h\bigr)dk
\\		&\justify-\integral\Delta(g,h)\lambda_hdk
\end{split}
\\%
\begin{split}
	\overset{\txt{[by \eqref{eqn:12B.11.13b}, \eqref{eqn:12B.11.15}: ]}}{\phantom=}
		&=\integral\Delta(gk,hk)\circ\bigl(\dot s_\updownarrow(h)\Delta(hk,k)-\dot s_\updownarrow(h)\avg X_h\bigr)dk
\\		&\justify+\integral_{tk=sh}\Delta(gk,hk)\avg\lambda_hdk-\Delta(g,h)\lambda_h
\end{split}
\\%
\begin{split}
	\overset{\txt{[by \eqref{eqn:12B.10.3b}: ]}}{\phantom=}
		&=\integral\Delta(gk,hk)\dot s_\updownarrow(h)\circ\bigl(\Delta(hk,k)-\avg X_h\bigr)dk
\\		&\justify+\integral_{tk=th}\Delta(gh^{-1}k,k)\avg\lambda_hdk-\Delta(g,h)\lambda_h
\end{split}
\\%
\begin{split}
		&=\iintegral\Delta(gk,hk)\dot s_\updownarrow(h)\circ\bigl(\Delta(hk,k)-\Delta(hk',k')\bigr)dkdk'
\\		&\justify+\avg X_{gh^{-1}}\avg\lambda_h-\Delta(g,h)\lambda_h\text.
\end{split}
\end{align*}
Thus far, we have not used the assumption that \(\varPhi\) was effective. Now if that is the case, then by \eqref{eqn:12B.11.18b} the double integral term must vanish. \qed%
% 17093 Dec  5 10:03

\section{Basic recursive estimates}\label{sec:16a.4}

Recursive averaging provides a general method to construct exact solutions of problems which, a~priori, are only known to admit ``approximate'' solutions. In the theory of topological groups, this method has been applied to prove the existence of homomorphisms of compact Lie groups near any given ``almost homomorphism'' \cite{GKR} and, again for compact groups, of representations by bounded Hilbert space operators near any given ``approximate representation'' \cite{delaHK}. The statements to be presented in this section lead to a similar type of result for (continuous) pseudo-rep\-re\-sen\-ta\-tions of proper groupoids. We stress that our computations are by no means a reproduction of arguments from the cited references, compared to which, they appear to be significantly shorter and simpler.

\begin{npar}\label{npar:12B.12.1} Let \(\varGamma\tto M\) be as before. If for any tangent pseudo-rep\-re\-sen\-ta\-tion \(\lambda\) which coincides with the effect of some non-de\-gen\-er\-ate groupoid connection \(\varPhi\) we set
\begin{equation*}
	\Delta^\lambda(h'h,h):=\dot s^\varPhi_\updownarrow(h')\Delta^\varPhi(h'h,h)
\end{equation*}
in the tautological expression \eqref{eqn:12B.11.18b} and then, for \((g,h)\in\varGamma_\div\), make \(h':=gh^{-1}\), we obtain the identity%
\begin{subequations}
\label{eqn:16a.32}
\begin{equation}
\label{eqn:12B.12.2}
	\Delta^\lambda(g,h)=\lambda(g)\lambda(h)^{-1}-\lambda(gh^{-1}).
\end{equation}
Regarding this as our definition of \(\Delta^\lambda\) when \(\lambda\) is an arbitrary invertible pseudo-rep\-re\-sen\-ta\-tion of \(\varGamma\tto M\) on some vector bundle \(E\) over \(M\), and letting \(s_\div,t_\div:\varGamma_\div\to M\) indicate the maps \((g,h)\mapsto th\mathord, \mapsto tg\) [compare \eqref{eqn:12B.11.5b}], we obtain a global cross-sec\-tion
\begin{equation}
\label{eqn:12B.12.3}
	\Delta^\lambda\in\Gamma^\infty\bigl(\varGamma_\div;L(s_\div^*E,t_\div^*E)\bigr).
\end{equation}
\end{subequations} \end{npar}

\begin{stmt}\label{stmt:16a.4.2} The following equations hold:%
\begin{subequations}
\label{eqn:12B.12.5}
\begin{gather}
\label{eqn:12B.12.5b}
\begin{split}
	\avg\lambda(g'g)-\avg\lambda(g')\avg\lambda(g)
		&=\integral_{tk=sg}\Delta^\lambda(g'gk,gk)\Delta^\lambda(gk,k)dk
\\		&\justify-\iintegral_{\substack{tk=sg\\tl=sg}}\Delta^\lambda(g'gk,gk)\Delta^\lambda(gl,l)dkdl
\end{split}
\\%
\label{eqn:12B.12.5a}
	\avg\lambda(g)=\lambda(g)+\integral_{tk=sg}\Delta^\lambda(gk,k)dk
\end{gather}
\end{subequations} \end{stmt}

\begin{proof} The second equation is a straightforward consequence of our Haar integral being normalized\emphpunct: \(\avg\lambda_g=\integral\lambda_{gk}(\lambda_k)^{-1}dk=\integral\lambda_gdk+\integral[\lambda_{gk}(\lambda_k)^{-1}-\lambda_g]*dk=\lambda_g+\integral\Delta^\lambda(gk,k)dk\). As to the first equation, we use both left invariance and normality of the Haar integral:%
\begin{align*}
	\avg\lambda_{g'g}-\avg\lambda_{g'}\avg\lambda_g
		&=\integral\lambda_{g'gk}(\lambda_k)^{-1}dk-\left(\integral\lambda_{g'k'}(\lambda_{k'})^{-1}dk'\right)\circ\left(\integral\lambda_{gk}(\lambda_k)^{-1}dk\right)
\\%
		&=\integral\lambda_{g'gk}(\lambda_k)^{-1}dk-\left(\integral\lambda_{g'gk}(\lambda_{gk})^{-1}dk\right)\circ\left(\integral\lambda_{gl}(\lambda_l)^{-1}dl\right)
\\%
\begin{split}
		&=\integral\lambda_{g'gk}(\lambda_k)^{-1}dk-\integral\lambda_{g'gk}(\lambda_{gk})^{-1}\lambda_gdk
\\		&\justify-\integral\lambda_{g'}\lambda_{gk}(\lambda_k)^{-1}dk+\lambda_{g'}\lambda_g
\\		&\justify-\iintegral\lambda_{g'gk}(\lambda_{gk})^{-1}\lambda_{gl}(\lambda_l)^{-1}dkdl+\integral\lambda_{g'gk}(\lambda_{gk})^{-1}\lambda_gdk
\\		&\justify+\integral\lambda_{g'}\lambda_{gl}(\lambda_l)^{-1}dl-\lambda_{g'}\lambda_g
\end{split}
\\%
\begin{split}
		&=\integral\bigl(\lambda_{g'gk}(\lambda_{gk})^{-1}-\lambda_{g'}\bigr)\circ\bigl(\lambda_{gk}(\lambda_k)^{-1}-\lambda_g\bigr)dk
\\		&\justify-\iintegral\bigl(\lambda_{g'gk}(\lambda_{gk})^{-1}-\lambda_{g'}\bigr)\circ\bigl(\lambda_{gl}(\lambda_l)^{-1}-\lambda_g\bigr)dkdl,
\end{split}
\end{align*}
which happens to be the desired expression. \end{proof}

No matter what the vec\-tor-bun\-dle metric that we put on \(E\), the norms \(\Norm\void_{x,y}\)~\eqref{eqn:16a.11} will satisfy the inequalities
\begin{equation}
\label{eqn:12B.12.6}
	\within\|\mu\circ\lambda\|_{x,z}\leq\within\|\mu\|_{y,z}\within\|\lambda\|_{x,y};
\end{equation}
in particular, \(\End(E_x)\) will be a (unital) Banach algebra under the norm \(\Norm\void_{x,x}\).

\begin{lem}\label{lem:12B.12.3} Let\/ \(A\) be a Banach algebra, with unit element\/ \(e\). Let a real constant\/ \(0\leq c<1\) be given. For every element\/ \(\varv\) of\/ \(A\) such that\/ \(\norm\varv\leq c\), the element\/ \(e-\varv\) is invertible and
\begin{equation*}
	\within*|(e-\varv)^{-1}-e|\leq c(1-c)^{-1}.
\end{equation*} \end{lem}

\begin{proof} Since \(\norm\varv<1\), the element \(e-\varv\) is invertible, with inverse
\begin{equation*}
	(e-\varv)^{-1}=e+\varv+\varv^2+\varv^3+\dotsb,
\end{equation*}
so that \(\within*|(e-\varv)^{-1}-e|\leq\norm\varv+\norm\varv^2+\norm\varv^3+\dotsb=\norm\varv(1-\norm\varv)^{-1}\leq c(1-c)^{-1}\). \end{proof}

\begin{stmt}\label{npar:12B.12.5} Let\/ \(\lambda\in\Psr_1(\varGamma;E)\) be a unital pseudo-rep\-re\-sen\-ta\-tion of\/ \(\varGamma\tto M\) on\/ \(E\). Suppose that\/ \(c(\lambda)<1\)~\eqref{eqn:12B.12.8b}. Then\/ \(\lambda\) is invertible. \end{stmt}

\begin{proof} The assumptions entail that \(1>\within*\|\id-\lambda_{\smash[t]{g^{-1}}}\circ\lambda_g\|_{sg,sg}\) for any \(g\). Since \(\End(E_{sg})\) (with the norm \(\Norm\void_{sg,sg}\)) is a (unital) Banach algebra, \(\lambda_{\smash[t]{g^{-1}}}\circ\lambda_g\) must be an invertible element of \(\End(E_{sg})\) and therefore \(\lambda_g\) must be a left invertible (hence injective) linear map. Similarly \(\lambda_g\) must be right invertible (hence surjective). \end{proof}

\begin{stmt}\label{npar:12B.12.6} The estimates below hold for every unital pseudo-rep\-re\-sen\-ta\-tion\/ \(\lambda\in\Psr_1(\varGamma;E)\) that satisfies the condition\/ \(c(\lambda)<1\) [cf.~\eqref{eqn:12B.12.8}]:%
\begin{subequations}
\label{eqn:12B.12.9}
\begin{gather}
\label{eqn:12B.12.9a}
	\within*\|\lambda(g)^{-1}\|_{tg,sg}\leq\frac{b(\lambda)}{1-c(\lambda)}
\\%
\label{eqn:12B.12.9b}
	\within*\|\Delta^\lambda(g,h)\|_{th,tg}\leq\left(\frac{b(\lambda)}{1-c(\lambda)}\right)c(\lambda)
\end{gather}
\end{subequations} \end{stmt}

\begin{proof} To prove the first inequality, \eqref{eqn:12B.12.9a}, let us apply Lemma~\ref{lem:12B.12.3} to the Banach algebra \(A:=\End(E_{sg})\) (with norm \(\Norm\void_{sg,sg}\)), the number \(c:=c(\lambda)<1\), and the Banach algebra element \(\varv:=\id-\lambda_{\smash[t]{g^{-1}}}\circ\lambda_g\). We get (omitting norm subscripts)
\begin{equation*}
	\within*\|{\lambda_g}^{-1}(\lambda_{\smash[t]{g^{-1}}})^{-1}-\id\|=\within*\|\mkern 1mu[\id-(\id-\lambda_{\smash[t]{g^{-1}}}\circ\lambda_g)]^{-1}-\id\|\leq c(1-c)^{-1}.
\end{equation*}
Using the inequalities \eqref{eqn:12B.12.6}, we obtain
\begin{equation*}
	\within*\|{\lambda_g}^{-1}-\lambda_{\smash[t]{g^{-1}}}\|=\within*\|\mkern 1mu[{\lambda_g}^{-1}(\lambda_{\smash[t]{g^{-1}}})^{-1}-\id]\circ\lambda_{\smash[t]{g^{-1}}}\|\leq c(1-c)^{-1}\within*\|\lambda_{\smash[t]{g^{-1}}}\|
\end{equation*}
whence
\begin{equation*}
	\within*\|{\lambda_g}^{-1}\|\leq\within*\|\lambda_{\smash[t]{g^{-1}}}\|+\within*\|{\lambda_g}^{-1}-\lambda_{\smash[t]{g^{-1}}}\|\leq\left(1+\frac c{1-c}\right)\within*\|\lambda_{\smash[t]{g^{-1}}}\|=\frac{\within*\|\lambda_{\smash[t]{g^{-1}}}\|}{1-c}.
\end{equation*}
As to \eqref{eqn:12B.12.9b}, it is an immediate consequence of \eqref{eqn:12B.12.9a}:
\begin{equation*}
	\within*\|\lambda_g(\lambda_h)^{-1}-\lambda_{g\smash[t]{h^{-1}}}\|=\within*\|(\lambda_g-\lambda_{g\smash[t]{h^{-1}}}\lambda_h)\circ{\lambda_h}^{-1}\|\leq c(\lambda)\within*\|{\lambda_h}^{-1}\|.	\qedhere%
\end{equation*} \end{proof}

The estimates \eqref{eqn:12B.12.10} drop out now as a corollary of the identities \eqref{eqn:12B.12.5} and of the preceding inequalities \eqref{eqn:12B.12.9}; indeed the Haar integrals involved in \eqref{eqn:12B.12.5} are normalized so one can estimate each one of them simply by the sup~norm of the integrand.

\subsection*{Proof of Lemma~\ref{lem:12B.12.8}}

We claim that if the inequality \eqref{eqn:12B.12.12a} holds for every integer \(i\) between zero and a certain non-neg\-a\-tive value of \(n\leq l\) then the inequality \eqref{eqn:12B.12.12b} must also hold for \(i=n\). Indeed, suppose that \(c_i\leq\varepsilon^{2^i}/(6b_0^2)\) for \(i=0,\dotsc,n\) and\textemdash only provisionally\textemdash that \(n>0\). Under such assumptions, we must have \(c_i<1\) for \(i=0,\dotsc,n-1\) (because by hypothesis \(b_0\geq1\) and \(\varepsilon<1\)), and consequently, by \eqref{eqn:12B.12.11}, \(b_{i+1}\leq b_i/(1-c_i)\). Combining all these inequalities recursively as \(i\) runs from zero to \(n-1\), we conclude that \(b_n\leq b_0\big/(1-c_0)\dotsb(1-c_{n-1})\). Hence
\begin{equation}
\label{eqn:12B.12.13}
	b_n/(1-c_n)\leq b_0\big/(1-c_0)\dotsb(1-c_n).
\end{equation}
This inequality is true also when \(n=0\) (trivially). We proceed to study the quantity
\begin{equation*}
\textstyle%
	1\big/\prod_{i=0}^n(1-c_i)=\bigl(\exp\log\prod_{i=0}^n(1-c_i)\bigr)^{-1}=\exp\bigl(-\sum_{i=0}^n\log(1-c_i)\bigr).
\end{equation*}
For every real number \(x\) such that \(\abs x<1\) we have \(\displaystyle-\abs x+\abs{\log(1+x)}\leq\abs{x-\log(1+x)}=\biggl\lvert\frac{x^2}2-\frac{x^3}3+\frac{x^4}4-\dotsb\biggr\rvert\leq\frac{\abs x^2}2+\frac{\abs x^3}2+\frac{\abs x^4}2+\dotsb=\frac{\abs x^2}2\frac 1{1-\abs x}\). This quantity is~\(\leq\abs x^2\) whenever \(\abs x\) is~\(\leq 1/2\). Hence, substituting \(x\) with \(-x\), we see that
\begin{equation}
\label{eqn:12B.12.14}
	0\leq x\leq\frac 12\seq-\log(1-x)\leq x+x^2.
\end{equation}
Since \(2^i\geq2i\) for every integer \(i\geq0\), and since by hypothesis \(b_0\geq1\) and \(\varepsilon\leq2/3<1\), for every integer \(0\leq i\leq n\) it must be true that \(c_i\leq\varepsilon^{2^i}/(6b_0^2)\leq\varepsilon^{2i}/6\), in particular, that \(c_i<1/2\), whence by \eqref{eqn:12B.12.14},
\begin{align*}
	\exp\Biggl(\textstyle\sum\limits_{i=0}^n-\log(1-c_i)\Biggr)
		&\leq\exp\within(\textstyle\sum\limits_{i=0}^nc_i+c_i^2)
\\		&\leq\exp\within(\frac 16\textstyle\sum\limits_{i=0}^n\varepsilon^{2i})\exp\within(\frac 1{6^2}\textstyle\sum\limits_{i=0}^n\varepsilon^{4i})
\\		&\leq\exp\within(\frac 16\frac 1{1-\varepsilon^2})\exp\within(\frac 1{6^2}\frac 1{1-\varepsilon^4})
\\		&\leq\exp\within(\frac 16\frac 95+\frac 1{6^2}\frac{81}{65})
\\		&\leq\exp\within(\frac 12\within[\frac 35+\frac 15\frac 12\frac 9{13}])
\\		&\leq\exp(1/2)\leq\sqrt 3.
\end{align*}
Combining this with \eqref{eqn:12B.12.13}, we obtain the desired inequality: \(b_n/(1-c_n)\leq\sqrt 3b_0\).

To finish the proof of the lemma, we are going to show that the inequality \eqref{eqn:12B.12.12a} holds for every integer \(i\) between zero and \(n\) by reasoning inductively on \(n\). Our claim is valid by hypothesis for \(n=0\). Assume that the claim holds for a certain value~\(\geq0\mathord,\leq l-1\) of \(n\). Then, by the above, \(b_n/(1-c_n)\leq\sqrt 3b_0\), whence by \eqref{eqn:12B.12.11}
\begin{flalign*}
&&	c_{n+1}\leq 2\within(\frac{b_n}{1-c_n})^2c_n^2\leq 2\cdot 3b_0^2\frac{\within*(\varepsilon^{2^n})^2}{\within*(6b_0^2)^2}=\frac{\varepsilon^{2^{n+1}}}{6b_0^2}.
&&	\hbox{\qedsymbol}
\end{flalign*}
% 11489 Dec  6 17:38

\section{Fast convergence theorem~A (pseudo-rep\-re\-sen\-ta\-tions)}\label{sec:16a.5}

The present section is devoted in its entirety to the proof of Theorem~\ref{thm:12B.14.2}. Recall that we are assigned a proper Lie groupoid \(\varGamma\tto M\) endowed with some specific normalized Haar system, say, \(\nu\), a vector bundle \(E\) over \(M\), and a \emph{near representation} (cf.~Definition~\ref{defn:12B.14.1}) \(\lambda\in\Psr_1(\varGamma;E)\). These data shall be kept fixed throughout the section.

To begin with, we observe that \(\lambda\) is necessarily {\em invertible}. Indeed, by our near multiplicativity condition \eqref{eqn:12B.14.1}, since \(b(\lambda\mathbin|U)\geq 1\) in consequence of the unitality of \(\lambda\), we must have \(c(\lambda\mathbin|U)\leq 1/9<1\) so our earlier remark \ref{npar:12B.12.5} implies that \(\lambda\mathbin|U\) is invertible.

In virtue of invertibility, it makes sense to consider the (unital) pseudo-rep\-re\-sen\-ta\-tion \(\avg\lambda\in\Psr_1(\varGamma;E)\) that one obtains by averaging \(\lambda\) against \(\nu\) by means of the formula \eqref{eqn:12B.12.4}. We contend that the pseudo-rep\-re\-sen\-ta\-tion arising in this way is itself nearly multiplicative. Indeed, let \(U\subset M\) be any invariant open subset satisfying \eqref{eqn:12B.14.1} for some choice of metrics on \(E\mathbin|U\). Let us write \(b\), \(c\), \(\avg b\) and \(\avg c\) respectively as short for \(b(\lambda\mathbin|U)\), \(c(\lambda\mathbin|U)\), \(b(\avg\lambda\mathbin|U)\) and \(c(\avg\lambda\mathbin|U)\). By unitality, \(b\geq 1\), so by \eqref{eqn:12B.14.1}, \(c\leq 1/9<1\). Then by \eqref{eqn:12B.12.10a}, \(\avg b\leq b/(1-c)\leq\frac 98b\), so \(\avg b^{-2}\geq\bigl(\frac 89\bigr)^2b^{-2}\geq\frac 12b^{-2}\). Also, by \eqref{eqn:12B.12.10b} and \eqref{eqn:12B.14.1}, \(\avg c\leq 2b^2(1-c)^{-2}c^2\leq 2b^2\bigl(1-\frac 19\bigr)^{-2}\bigl(\frac 19\bigr)^2b^{-4}=2\bigl(\frac 98\bigr)^2\bigl(\frac 19\bigr)^2b^{-2}\leq\frac 19\frac 12b^{-2}\). Hence \(\avg c\leq\frac 19\avg b^{-2}\).

From the above remarks it follows that \(\lambda\) gives rise to a sequence \(\within*{\avg[i]\lambda}_{i=0}^\infty\) of \emph{averaging iterates} \(\avg[i]\lambda\in\Psr_1(\varGamma;E)\), constructed recursively by setting \(\avg[0]\lambda:=\lambda\emphpunct, \avg[i+1]\lambda:=\avg(\avg[i]\lambda)\). Making use of our computations from the previous section, we proceed to show that as a sequence of smooth cross-sec\-tions of the vector bundle \(L(s^*E,t^*E)\) this is Cauchy within \(\Gamma^0\bigl(\varGamma;L(s^*E,t^*E)\bigr)\), the Fr\'e\-chet space of cross-sec\-tions of class \(C^0\) of \(L(s^*E,t^*E)\), and thus converges within the same space to a (unique) pseudo-rep\-re\-sen\-ta\-tion, say, \(\avg[\infty]\lambda\), which a priori will only be of class \(C^0\) (i.e.~continuous).

The idea is to apply Lemma~\ref{cor:12B.13.5} from Appendix~\ref{sec:16a.A} to the open cover of \(\varGamma\) consisting of all those relatively compact open subsets \(\varOmega\subset\varGamma\) for which invariant open subsets \(U\subset M\) which satisfy \eqref{eqn:12B.14.1} for some choice of metrics on \(E\mathbin|U\) exist so that \(\overline\varOmega\subset s^{-1}(U)\cap t^{-1}(U)\). We contend that for any such \(\varOmega\) the sequence of restrictions \(\{\res^\varGamma_{\overline\varOmega}(\avg[i]\lambda)\}\) is Cauchy within \(\Gamma^0\bigl(\overline\varOmega;L(s^*E,t^*E)\bigr)\) relative to the \emph{\(C^0\)\mdash norm topology} (Appendix~\ref{sec:16a.A}). Given \(\varOmega\), let us fix \(U\) together with a metric on \(E\mathbin|U\) as indicated. The \(C^0\)\mdash norm topology on \(\Gamma^0\bigl(\overline\varOmega;L(s^*E,t^*E)\bigr)\) is then generated by the following norm\textemdash recall our notations \eqref{eqn:16a.11}:
\begin{equation}
\label{eqn:16a.38}
	p_\varOmega^{(0)}(\zeta):=\sup_{g\in\varOmega}\within*\|\zeta(g)\|_{sg,tg}.
\end{equation}
Since \(U\) is {\em invariant}, upon restriction, the selected normalized Haar system \(\nu\) on \(\varGamma\tto M\) induces a similar system, say, \(\nu\mathbin|U\), on \(\varGamma\mathbin|U\tto U\), relative to which the following equation holds for all invertible pseudo-rep\-re\-sen\-ta\-tions \(\zeta\in\Psr_\div(\varGamma;E)\).%
\begin{equation}
\label{eqn:16a.39}
	\avg\zeta\mathbin|U=\mathinner{\avg(\zeta\mathbin|U)}
\end{equation}
We may hence suppose that \(U=M\), without loss of generality. Now, either \(b(\lambda)=\infty\), in which case \(c(\lambda)=0\) and \(\lambda\) is \emph{multiplicative} (a representation) so \(\avg[i]\lambda=\lambda\) for all \(i\) and our sequence is constant, or \(1\leq b(\lambda)<\infty\). In the latter case, by our estimates \eqref{eqn:12B.12.10} and by the near multiplicativity condition \eqref{eqn:12B.14.1}, the two sequences of non-neg\-a\-tive real numbers \(b_i:=b(\avg[i]\lambda)\emphpunct, c_i:=c(\avg[i]\lambda)\) must satisfy the hypotheses of Lemma~\ref{lem:12B.12.8}. Then, by the identity \eqref{eqn:12B.12.5a} and by the estimates \eqref{eqn:12B.12.9b} and \eqref{eqn:12B.12.12}, for some number \(0\leq\varepsilon\leq 2/3\) we must have
\begin{align*}
 p_\varOmega^{(0)}\bigl(\res^\varGamma_{\overline\varOmega}(\avg[i+1]\lambda)-\res^\varGamma_{\overline\varOmega}(\avg[i]\lambda)\bigr)
	&\leq\left(\frac{b(\avg[i]\lambda)}{1-c(\avg[i]\lambda)}\right)c(\avg[i]\lambda)
\\	&\leq\sqrt 3/6\cdot b(\lambda)^{-1}\cdot\varepsilon^{2^i}\leq\varepsilon^{2^i}
\end{align*}
for all \(i\). This clearly implies that our sequence is Cauchy.

Next, let us confirm that \(\avg[\infty]\lambda\) is a {\em representation}. Since \(C^0\)\mdash convergence implies pointwise convergence, for every arrow \(g\) the sequence of linear maps \(\{\avg[i]\lambda(g)\}\) converges to \(\avg[\infty]\lambda(g)\) in the fi\-nite-di\-men\-sion\-al vector space \(L(E_{sg},E_{tg})\). Hence \(\avg[\infty]\lambda_{1x}=\lim\avg[i]\lambda_{1x}=\id\) for all \(x\) in \(M\), since every \(\avg[i]\lambda\) is a unital pseudo-rep\-re\-sen\-ta\-tion. Further, given any composable pair of arrows \((g',g)\), there exists \(U\) as before such that \(g',g\in s^{-1}(U)\cap t^{-1}(U)\); for any metric on \(E\mathbin|U\), we then have the following inequality (omitting norm subscripts),%
\begin{align*}
 \within*\|\avg[\infty]\lambda_{g'g}-\avg[\infty]\lambda_{g'}\avg[\infty]\lambda_g\|
	&\leq\within*\|\avg[\infty]\lambda_{g'g}-\avg[i]\lambda_{g'g}\|+\within*\|\avg[i]\lambda_{g'g}-\avg[i]\lambda_{g'}\avg[i]\lambda_g\|%
\\	&\justify+\within*\|\avg[i]\lambda_{g'}-\avg[\infty]\lambda_{g'}\|\mathinner{\within*\|\avg[i]\lambda_g\|}+\mathinner{\within*\|\avg[\infty]\lambda_{g'}\|}\within*\|\avg[i]\lambda_g-\avg[\infty]\lambda_g\|
\end{align*}
whose right-hand side becomes arbitrarily small as \(i\) grows sufficiently large, in view of considerations from the preceding paragraph.

There remains to be seen whether \(\avg[\infty]\lambda\) is actually a {\em smooth}\/ cross-sec\-tion of the vector bundle \(L(s^*E,t^*E)\). It will be enough to show that the sequence of averaging iterates \(\{\avg[i]\lambda\}\) is Cauchy within the Fr\'e\-chet space \(\Gamma^\infty\bigl(\varGamma;L(s^*E,t^*E)\bigr)\). The idea is the same as before: apply Lemma~\ref{cor:12B.13.5} to a suitable cover of \(\varGamma\) by relatively compact open subsets. This time, however, some extra care will be needed in the choice of the cover.

\subsection*{Sketch of the argument for uniform convergence up to order $r$}

We are going to define a suitable open cover of \(\varGamma\) consisting of relatively compact open subsets \(\varOmega\subset\varGamma\) selected from among those considered previously. For each \(\varOmega\) in the cover and for each non-neg\-a\-tive integer \(r=0,1,2,\dotsc\) we are going to define suitable ``standard\/ \(C^r\)\mdash norms'' (cf.~Appendix~\ref{sec:16a.A})%
\begin{subequations}
\label{eqn:12B.14.3}
\begin{align}
	p^{(r)}_\varOmega
\quad	&\text{on}\quad	\Gamma^r\bigl(\overline\varOmega;L(s^*E,t^*E)\bigr),
\label{eqn:12B.14.3a}
\\	\bar p^{(r)}_\varOmega
\quad	&\text{on}\quad	\Gamma^r\bigl(\overline\varOmega;L(t^*E,s^*E)\bigr),
\label{eqn:12B.14.3b}
\\	q^{(r)}_\varOmega
\quad	&\text{on}\quad	\Gamma^r\bigl(\overline\varOmega_2;L(s_2^*E,t_2^*E)\bigr),
\label{eqn:12B.14.3c}
\end{align}
\end{subequations}
where \(\varOmega_2:=\varOmega\ftimes st\varOmega\) and \(s_2,t_2:\varGamma_2:=\varGamma\ftimes st\varGamma\to M\) denote the two maps \((g',g)\mapsto sg\mathord, \mapsto tg'\), so that \(p^{(r)}_\varOmega\leq p^{(r+1)}_\varOmega\), etc.~for all \(r\). For \(r=0\), the norm \eqref{eqn:12B.14.3a} is going to be defined by \eqref{eqn:16a.38}, the other two are going to be defined analogously. Then, letting \(m,\pr_1,\pr_2:\varGamma_2\to\varGamma\) respectively stand for law of arrow composition, first, and second projection, henceforth regarding \(\varOmega\) as fixed, for any (where appropriate, invertible) pseudo-rep\-re\-sen\-ta\-tion \(\zeta\in\Gamma^\infty\bigl(\varGamma;L(s^*E,t^*E)\bigr)\) we set%
\begin{align*}
        b^{(r)}(\zeta)&:=p^{(r)}_\varOmega\bigl(\res^\varGamma_{\overline\varOmega}(\zeta)\bigr),
\\ \bar b^{(r)}(\zeta)&:=\bar p^{(r)}_\varOmega\bigl(\res^\varGamma_{\overline\varOmega}(\zeta)^{-1}\bigr),
\\      c^{(r)}(\zeta)&:=q^{(r)}_\varOmega\bigl(\res^{\varGamma_2}_{\overline\varOmega_2}(m^*\zeta-\pr_1^*\zeta\circ\pr_2^*\zeta)\bigr)
\end{align*}
(by abuse of notation, the latter expression indicates the difference between the following two composite vector bundle morphisms,%
\begin{gather*}
	s_2^*E\cong\pr_2^*s^*E\cong m^*s^*E\xto{m^*\zeta}m^*t^*E\cong\pr_1^*t^*E\cong t_2^*E\makebox[.pt][l]{,}
\\	s_2^*E\cong\pr_2^*s^*E\xto{\pr_2^*\zeta}\pr_2^*t^*E\cong\pr_1^*s^*E\xto{\pr_1^*\zeta}\pr_1^*t^*E\cong t_2^*E\makebox[.pt][l]{).}
\end{gather*}

Our definitions will be such that the three inequalities below hold for every invertible pseudo-rep\-re\-sen\-ta\-tion \(\zeta\in\Gamma^\infty\bigl(\varGamma;\Lis(s^*E,t^*E)\bigr)\) which satisfies the condition \(c^{(0)}(\zeta)<1\);%
\begin{subequations}
\label{eqn:16a.41}
\begin{gather}
	\bar b^{(0)}(\zeta)\leq b^{(0)}(\zeta)\big/\bigl(1-c^{(0)}(\zeta)\bigr)
\label{eqn:12B.14.17}
\\	b^{(0)}(\avg\zeta-\zeta)\leq\mathinner{\bigl[b^{(0)}(\zeta)\big/\bigl(1-c^{(0)}(\zeta)\bigr)\bigr]}c^{(0)}(\zeta)
\label{eqn:16a.41b}
\\	c^{(0)}(\avg\zeta)\leq 2\mathinner{\bigl[b^{(0)}(\zeta)\big/\bigl(1-c^{(0)}(\zeta)\bigr)\bigr]}^2\bigl(c^{(0)}(\zeta)\bigr)^2
\label{eqn:16a.41c}
\end{gather}
\end{subequations}
moreover, there will exist positive constants \(A^{(r)}\), \(B^{(r)}\) and \(C^{(r)}\) such that the further three inequalities below hold for all \(\zeta\in\Gamma^\infty\bigl(\varGamma;\Lis(s^*E,t^*E)\bigr)\).%
\begin{subequations}
\label{eqn:16a.42}
\begin{gather}
	\bar b^{(r+1)}(\zeta)\leq\bigl(\bar b^{(r)}(\zeta)\bigr)^2b^{(r+1)}(\zeta)A^{(r)}
\label{eqn:16a.42a}
\\	b^{(r+1)}(\avg\zeta-\zeta)\leq\mathinner{[\bar b^{(r)}(\zeta)b^{(r+1)}(\zeta)c^{(r)}(\zeta)+c^{(r+1)}(\zeta)]}\bar b^{(r)}(\zeta)B^{(r)}
\label{eqn:12B.14.10}
\\	c^{(r+1)}(\avg\zeta)\leq\mathinner{[\bar b^{(r)}(\zeta)b^{(r+1)}(\zeta)c^{(r)}(\zeta)+c^{(r+1)}(\zeta)]}\bigl(\bar b^{(r)}(\zeta)\bigr)^2c^{(r)}(\zeta)C^{(r)}
\label{eqn:12B.14.15}
\end{gather}
\end{subequations}

Back to our near representation \(\lambda\), now, let us introduce the following abbreviations for every order of iteration \(i=0,1,2,\dotsc\colon\)%
\begin{align*}
        b^{(r)}_i&:=b^{(r)}(\avg[i]\lambda),
\\ \bar b^{(r)}_i&:=\bar b^{(r)}(\avg[i]\lambda),
\\      c^{(r)}_i&:=c^{(r)}(\avg[i]\lambda).
\end{align*}
Arguing by induction on \(r\) on the basis of the inequalities \eqref{eqn:16a.42}, we are going to show that for some number \(0\leq\varepsilon<1\) (independent of \(i\) and \(r\)) the four statements S1--S4 hereafter must be true for all \(r\).
\begin{description}
\itemsep=0pt%
 \item[S1.]\em The sequence\/ \(\bigl\{b^{(r)}_i\bigr\}\) is bounded.\em
 \item[S2.]\em The sequence\/ \(\bigl\{\bar b^{(r)}_i\bigr\}\) is bounded.\em
 \item[S3.]\em There exists some positive constant\/ \(R^{(r)}\) such that\/ \(c^{(r)}_{i+1}\leq\bigl(c^{(r)}_i\bigr)^2R^{(r)}\) for every\/ \(i\).\em
 \item[S4.]\em There exists some non-neg\-a\-tive integer\/ \(i^{(r)}\) such that\/ \(c^{(r)}_i\leq\varepsilon^{2^{i-i^{(r)}}}\) for all\/ \(i\geq i^{(r)}\).\em
\end{description}
For \(r=0\), we can already verify the truth of these statements. Indeed, the two inequalities \eqref{eqn:16a.41b} and \eqref{eqn:16a.41c} imply at once that the two sequences of non-neg\-a\-tive real numbers \(b_i:=b^{(0)}_i\emphpunct, c_i:=c^{(0)}_i\) satisfy the hypothesis \eqref{eqn:12B.12.11} of Lemma~\ref{lem:12B.12.8}; in addition, \(6\bigl(b^{(0)}_0\bigr)^2c^{(0)}_0\leq\varepsilon:=6\mathinner{b(\lambda\mathbin|U)}^2c(\lambda\mathbin|U)\leq 2/3\) by \eqref{eqn:12B.14.1}, and \(b^{(0)}_0\geq 1\) by unitality, so all the hypotheses of that lemma are satisfied. The same lemma then entails that \(c^{(0)}_i<1\) for all \(i\), and\emphpunct: (S1)~\(b^{(0)}_{i+1}\leq b^{(0)}_i/(1-c^{(0)}_i)\leq\sqrt 3b^{(0)}_0\), by \eqref{eqn:16a.41b} and \eqref{eqn:12B.12.12b}\emphpunct; (S2)~\(\bar b^{(0)}_i\leq b^{(0)}_i/(1-c^{(0)}_i)\leq\sqrt 3b^{(0)}_0\), by \eqref{eqn:12B.14.17} and \eqref{eqn:12B.12.12b}\emphpunct; (S3)~\(c^{(0)}_{i+1}\leq 2\bigl[b^{(0)}_i\big/\bigl(1-c^{(0)}_i\bigr)\bigr]^2\bigl(c^{(0)}_i\bigr)^2\leq\bigl(c^{(0)}_i\bigr)^2 6\bigl(b^{(0)}_0\bigr)^2\), by \eqref{eqn:16a.41c} and \eqref{eqn:12B.12.12b}\emphpunct; (S4)~\(c^{(0)}_i\leq\frac 16\bigl(b^{(0)}_0\bigr)^2\varepsilon^{2^i}<\varepsilon^{2^i}\), by \eqref{eqn:12B.12.12a} because \(b^{(0)}_0\geq 1\).

Now, for an arbitrary order of derivation \(r\), the validity of the above statements S1~to S4 enables one to conclude that the sequence of restrictions \(\{\res^\varGamma_{\overline\varOmega}(\avg[i]\lambda)\}\) is Cauchy within \(\Gamma^r\bigl(\overline\varOmega;L(s^*E,t^*E)\bigr)\) relative to the \emph{\(C^r\)\mdash norm topology} (compare Appendix~\ref{sec:16a.A}). (From this, the \(C^\infty\)\mdash convergence of the sequence \(\{\avg[i]\lambda\}\) drops out immediately, as before by virtue of the lemma~\ref{cor:12B.13.5}.) Indeed, on the basis of the identity \eqref{eqn:12B.12.5a} and of the inequality \eqref{eqn:12B.14.10}, for every order of iteration \(i\) we have the following estimate,
\begin{align*}
 p^{(r)}_\varOmega\bigl(\res^\varGamma_{\overline\varOmega}(\avg[i+1]\lambda)-\res^\varGamma_{\overline\varOmega}(\avg[i]\lambda)\bigr)
	&\leq p^{(r+1)}_\varOmega\bigl(\res^\varGamma_{\overline\varOmega}(\avg[i+1]\lambda)-\res^\varGamma_{\overline\varOmega}(\avg[i]\lambda)\bigr)
\\	&\leq\mathinner{(\bar b^{(r)}_ib^{(r+1)}_ic^{(r)}_i+c^{(r+1)}_i)}\bar b^{(r)}_iB^{(r)}
\\	&\leq\sup\mathinner{\bigl\{(\bar b^{(r)}_ib^{(r+1)}_i+1)\bar b^{(r)}_i\bigr\}}B^{(r)} c^{(r+1)}_i
\end{align*}
in the derivation of which we use the fact that \(c^{(r)}_i\leq c^{(r+1)}_i\) (a consequence of the inequality \(p^{(r)}_\varOmega\leq p^{(r+1)}_\varOmega\)).

\subsection*{Inductive step (from $r$ to $r\mathinner+1$)}

We begin to fill in the details of our argument by showing that, once the inequalities \eqref{eqn:16a.42} are taken for granted, the validity for any given value of \(r\) of the statements S1--S4 implies their validity also for the next higher value of \(r\). By \eqref{eqn:16a.42}, there must be positive constants \(B^{(r)}\)~and \(C^{(r)}\) (independent of \(i\)) such that the two inequalities below hold for all \(i\).
\begin{align*}
	b^{(r+1)}_{i+1}&\leq b^{(r+1)}_i+\mathinner{(\bar b^{(r)}_ib^{(r+1)}_ic^{(r)}_i+c^{(r+1)}_i)}\bar b^{(r)}_iB^{(r)}
\\	c^{(r+1)}_{i+1}&\leq\mathinner{(\bar b^{(r)}_ib^{(r+1)}_ic^{(r)}_i+c^{(r+1)}_i)}\bigl(\bar b^{(r)}_i\bigr)^2c^{(r)}_iC^{(r)}
\end{align*}
Our inductive hypothesis S2 then entails the existence of a positive constant \(L^{(r)}\) such that the following two inequalities are satisfied for all \(i\).
\begin{align*}
	b^{(r+1)}_{i+1}&\leq b^{(r+1)}_i+\mathinner{(b^{(r+1)}_ic^{(r)}_i+c^{(r+1)}_i)}L^{(r)}
\\	c^{(r+1)}_{i+1}&\leq\mathinner{(b^{(r+1)}_ic^{(r)}_i+c^{(r+1)}_i)}c^{(r)}_iL^{(r)}
\end{align*}
In order to complete the inductive step, we need to ``solve'' this recursive system.

\begin{lem}\label{lem:12B.14.4} Let\/ \(\{c_0,c_1,c_2,\dotsc\}\), \(\{b'_0,b'_1,b'_2,\dotsc\}\), and\/ \(\{c'_0,c'_1,c'_2,\dotsc\}\) be sequences of non-neg\-a\-tive real numbers. Let\/ \(\varepsilon,L,R\) be positive real constants, with\/ \(\varepsilon<1\). Suppose that%
\begin{subequations}
\label{eqn:12B.14.20}
\begin{gather}
	b'_{i+1}\leq b'_i+\mathinner{(b'_ic_i+c'_i)}L,
\label{eqn:12B.14.20a}
\\	c'_{i+1}\leq\mathinner{(b'_ic_i+c'_i)}c_iL,
\label{eqn:12B.14.20b}
\\	c_{i+1}\leq c_i^2R,
\label{eqn:12B.14.20c}
\end{gather}
\end{subequations}
and that\/ \(c_i\leq c'_i\) for all\/ \(i\). Further, assume there exists\/ \(I\) for which\/ \(c_i\leq\varepsilon^{2^{i-I}}\) when\/ \(i\geq I\). Then, the following three statements hold\emphpunct: {\upshape(a)}~The sequence\/ \(\{b'_i\}\) is bounded. {\upshape(b)}~There exists some constant\/ \(R'>0\) such that\/ \(c'_{i+1}\leq(c'_i)^2R'\) for all\/ \(i\). {\upshape(c)}~There exists\/ \(I'\geq I\) such that\/ \(c'_i\leq\varepsilon^{2^{i-I'}}\) for\/ \(i\geq I'\). \end{lem}

\begin{proof} At the expense of re-in\-dex\-ing our sequences, it will be no loss of generality to assume that \(I=0\). Under such assumption, for every \(i\) we will have \(c_i\leq\varepsilon^{2^i}\leq\varepsilon<1\), and therefore, \(c_i^2\leq c_i\). Let us put \(a'_i:=b'_ic_i+c'_i\). Then%
\begin{align*}
 a'_{i+1}&=b'_{i+1}c_{i+1}+c'_{i+1}
\\       &\leq b'_ic_i^2R+a'_ic_i^2LR+c'_{i+1}
&	&\text{by \eqref{eqn:12B.14.20a} and \eqref{eqn:12B.14.20c},}
\\       &\leq b'_ic_i^2R+a'_ic_iLR+c'_{i+1}
&	&\text{because $c_i^2\leq c_i$,}
\\       &\leq b'_ic_i^2R+(LR+L)a'_ic_i
&	&\text{by \eqref{eqn:12B.14.20b},}
\\       &\leq(L+LR)a'_ic_i+b'_ic_i^2R+c'_ic_iR
&	&\text{a~fortiori,}
\\       &=(L+LR)a'_ic_i+(b'_ic_i+c'_i)c_iR
\\       &=(L+LR+R)a'_ic_i
\end{align*}
and thus, setting \(K:=L+LR+R\),
\begin{equation*}
	a'_{i+1}\leq Ka'_ic_i,
\end{equation*}
whence \(a'_1\leq Ka'_0c_0\emphpunct, a'_2\leq KKa'_0c_0c_1\emphpunct, a'_3\leq KK^2a'_0c_0c_1c_2\) and, in general,
\begin{equation}
\label{eqn:12B.14.22}
\textstyle	a'_i\leq a'_0K^i\prod\limits_{n=0}^{i-1}c_n.
\end{equation}
Since \(1+2+\dotsb+2^{i-1}=2^i-1\), it follows from \eqref{eqn:12B.14.20a} in combination with \eqref{eqn:12B.14.22} and with the hypothesis \(c_n\leq\varepsilon^{2^n}\) that
\begin{align*}
 b'_{i+1}\leq b'_i+La'_i&\leq b'_i+La'_0K^i\varepsilon^{1+2+\dotsb+2^{i-1}}
\\                      &=b'_i+La'_0\varepsilon^{-1}K^i\varepsilon^{2^i}
\end{align*}
and, therefore, by induction,
\begin{equation*}
\textstyle	b'_i\leq b'_0+La'_0\varepsilon^{-1}\sum\limits_{n=0}^{i-1}K^n\varepsilon^{2^n}.
\end{equation*}
The last inequality implies that the sequence \(\{b'_i\}\) has to be bounded, which was our first claim (a). Using this fact in combination with the hypothesis \(c_i\leq c'_i\) and with \eqref{eqn:12B.14.20b}, we are then able to establish our second claim (b) as well:
\begin{equation*}
	c'_{i+1}\leq\mathinner{(b'_ic_i+c'_i)}c_iL\leq\mathinner{(b'_ic'_i+c'_i)}c'_iL=(c'_i)^2(b'_i+1)L.
\end{equation*}
As to our third claim (c), we have
\begin{align*}
 c'_{i+1}&\leq\biggl(a'_0K^i\textstyle\prod\limits_{n=0}^{i-1}c_n\biggr)c_iL
&	&\text{by \eqref{eqn:12B.14.20b} and \eqref{eqn:12B.14.22},}
\\       &\leq a'_0K^i\varepsilon^{1+2+\dotsb+2^{i-1}}\varepsilon^{2^i}L
&	&\text{because $c_n\leq\varepsilon^{2^n}$,}
\\       &=(La'_0\varepsilon^{-1}K^i\varepsilon^{2^i})\varepsilon^{2^i},
\end{align*}
the parenthesized factor being \(<1\) for \(i\) sufficiently large because \(\lim K^i\varepsilon^{2^i}=0\). \end{proof}

By virtue of the inductive hypotheses S3 and S4 and of the inequality \(q^{(r)}_\varOmega\leq q^{(r+1)}_\varOmega\), the three sequences of non-neg\-a\-tive real numbers \(c_i:=c^{(r)}_i\emphpunct, b'_i:=b^{(r+1)}_i\emphpunct, c'_i:=c^{(r+1)}_i\) satisfy the assumptions of Lemma~\ref{lem:12B.14.4} with \(L:=L^{(r)}\emphpunct, R:=R^{(r)}\). The truth of the inductive claims S1, S3 and S4 (for the next higher value of \(r\)) is then precisely the content of the lemma. As to the remaining inductive claim S2, this follows from the estimate \eqref{eqn:16a.42a} in conjunction with the already proven inductive claim S1 and the inductive hypothesis S2.

\subsection*{Choice of the relatively compact open cover $\{\varOmega\}$ and of the standard $C^r$\mdash norms $p_\varOmega^{(r)}$, $\bar p_\varOmega^{(r)}$ and $q_\varOmega^{(r)}$ \eqref{eqn:12B.14.3}}

Whenever the near multiplicativity inequality \eqref{eqn:12B.14.1} is satisfied on an invariant open subset \(U\subset M\) for some choice of metrics on \(E\mathbin|U\), the same inequality must be satisfied on any open (not necessarily invariant) subset of \(M\) whose closure lies within \(U\) for some choice of a {\em globally defined}\/ metric on \(E\).

Suppose that \(U\subset M\) is a relatively compact open subset on which \eqref{eqn:12B.14.1} holds for some choice of metrics on \(E\); furthermore, suppose that \(U\) is \emph{adjusted} to the given normalized Haar system \(\nu\) (Appendix~\ref{sec:16a.B}). The property of being \(\nu\)\mdash adjusted entails that \(\nu\) restricts on \(\varGamma\mathbin|U\tto U\) to a normalized Haar system \(\nu\mathbin|U\) such that the same identity as in the invariant case \eqref{eqn:16a.39} is satisfied; in other words, {\em averaging commutes with restriction}\/ over \(U\) when \(U\) is \(\nu\)\mdash adjusted.

For \(U\) as in the preceding paragraph, let \(\varOmega:=s^{-1}(U)\cap t^{-1}(U)\). The open sets \(\varOmega\) thus obtained form an open cover of \(\varGamma\). Since the closure \(\overline U\) of \(U\) within \(M\) is compact, so will be the closure \(\overline\varOmega\) of \(\varOmega\) within \(\varGamma\), because of properness. Moreover, \(\varOmega_2:=\varOmega\ftimes st\varOmega\) will be a relatively compact open subset of \(\varGamma_2:=\varGamma\ftimes st\varGamma\emphpunct, \varOmega_3:=\varOmega\ftimes st\varOmega\ftimes st\varOmega\) will be a similar subset of \(\varGamma_3:=\varGamma\ftimes st\varGamma\ftimes st\varGamma\), and so forth.

Considering \(U\) as fixed now, endow \(E\) with some vec\-tor-bun\-dle metric so that \eqref{eqn:12B.14.1} is satisfied. Regard \(L(s^*E,t^*E)\) as a normed vector bundle (see Appendix~\ref{sec:16a.A}) by endowing it with the continuous vec\-tor-bun\-dle norm defined by \eqref{eqn:16a.11}; similarly for \(L(t^*E,s^*E)\) and \(L(s_2^*E,t_2^*E)\). In addition, choose any three locally finite trivializing vec\-tor-bun\-dle atlases, say, \(\matheusm A\) for \(L(s^*E,t^*E)\), \(\tilde{\matheusm A}\) for \(L(t^*E,s^*E)\), and \(\matheusm B\) for \(L(s_2^*E,t_2^*E)\). Then, in the notations of Appendix~\ref{sec:16a.A} \eqref{eqn:12B.13.3}, make the following definitions for every natural number \(r\).
\begin{align*}
        p_\varOmega^{(r)}&:=\within\|\void\|_{C^r\overline\varOmega;L(s^*E,t^*E),\matheusm A}
\\ \bar p_\varOmega^{(r)}&:=\within\|\void\|_{C^r\overline\varOmega;L(t^*E,s^*E),\tilde{\matheusm A}}
\\      q_\varOmega^{(r)}&:=\within\|\void\|_{C^r\overline\varOmega_2;L(s_2^*E,t_2^*E),\matheusm B}
\end{align*}

We leave it as an exercise for the reader to verify that with these definitions the three inequalities \eqref{eqn:16a.41} are in fact satisfied. Further, the estimate \eqref{eqn:16a.42a} is a general consequence of Lemma~\ref{lem:12B.13.8} \eqref{eqn:12B.13.9}. So, to finish the proof of Theorem~\ref{thm:12B.14.2}, we only need to establish the other two estimates, \eqref{eqn:12B.14.10} and \eqref{eqn:12B.14.15}.

\subsection*{Proof of the estimate \eqref{eqn:12B.14.10}}

Let \(\d\nu\) denote the Haar integration functional~\eqref{eqn:12B.10.6b}
\begin{equation*}
	\Gamma^\infty\bigl(\varGamma_2;\pr_1^*L(s^*E,t^*E)\bigr)\longto\Gamma^\infty\bigl(\varGamma;L(s^*E,t^*E)\bigr)\emphpunct, \vartheta\mapsto\within*<\vartheta,\d\nu>:=\integral*\vartheta d\nu
\end{equation*}
depending on parameters in \(\varGamma\xto sM\). Also recall from \S\ref{sec:16a.4} that an invertible pseudo-rep\-re\-sen\-ta\-tion \(\zeta\in\Psr_\div(\varGamma;E)\) determines a cross-sec\-tion \eqref{eqn:16a.32}%
\begin{equation*}
	\Delta^\zeta\in\Gamma^\infty\bigl(\varGamma_\div;L(s_\div^*E,t_\div^*E)\bigr).
\end{equation*}
Let \(a\) denote the diffeomorphism \(\varGamma_2\approxto\varGamma_\div\emphpunct, (g,k)\mapsto(gk,k)\). Regarding \(a^*\Delta^\zeta\) as a cross-sec\-tion of the vector bundle \(\pr_1^*L(s^*E,t^*E)\) over \(\varGamma_2\)\textemdash and likewise regarding similar expressions occurring hereafter as cross-sec\-tions of the appropriate vector bundles\textemdash by the identity \eqref{eqn:12B.12.5a} we have, making use of the indicated lemmas from Appendix~\ref{sec:16a.A} and \ref{sec:16a.B},
\begin{flalign*}
&&\qquad	\within*\|\res^\varGamma_{\overline\varOmega}(\avg\zeta-\zeta)\|_{C^{r+1}}
		&=\within*\|\res^\varGamma_{\overline\varOmega}\within*<a^*\Delta^\zeta,\d\nu>\|_{C^{r+1}}
\\%
\makebox[.pt][l]{[by Lemma~\ref{lem:12B.20.1}: ]}
&&		&\leq\within*\|\res^{\varGamma_2}_{\overline\varOmega_2}(a^*\Delta^\zeta)\|_{C^{r+1}}
\tag{\dag}%
\\%
\makebox[.pt][l]{[by Lemma~\ref{lem:12B.13.9}: ]}
&&		&=\within*\|\res^{\varGamma_2}_{\overline\varOmega_2}(m^*\zeta-\pr_1^*\zeta\circ\pr_2^*\zeta)\circ\res^{\varGamma_2}_{\overline\varOmega_2}(\pr_2^*\zeta^{-1})\|_{C^{r+1}}
\\%
\makebox[.pt][l]{[by Lemma~\ref{lem:12B.13.7} \eqref{eqn:12B.13.5b}: ]}
&&		&\leq\within*\|\res^{\varGamma_2}_{\overline\varOmega_2}(m^*\zeta-\pr_1^*\zeta\circ\pr_2^*\zeta)\|_{C^r}\within*\|\res^{\varGamma_2}_{\overline\varOmega_2}(\pr_2^*\zeta^{-1})\|_{C^{r+1}}\\*	&&	&\justify+\within*\|\res^{\varGamma_2}_{\overline\varOmega_2}(m^*\zeta-\pr_1^*\zeta\circ\pr_2^*\zeta)\|_{C^{r+1}}\within*\|\res^{\varGamma_2}_{\overline\varOmega_2}(\pr_2^*\zeta^{-1})\|_{C^r}
\\%
\makebox[.pt][l]{[by Lemma~\ref{lem:12B.13.10}: ]}
&&		&\leq\within\|\mathellipsis\|_{C^r}\within*\|\res^\varGamma_{\overline\varOmega}(\zeta^{-1})\|_{C^{r+1}}+\within\|\mathellipsis\|_{C^{r+1}}\within*\|\res^\varGamma_{\overline\varOmega}(\zeta^{-1})\|_{C^r}
\\%
\makebox[.pt][l]{[by Lemma~\ref{lem:12B.13.8} \eqref{eqn:12B.13.9}: ]}
&&		&\leq\within\|\mathellipsis\|_{C^r}\within*\|\res^\varGamma_{\overline\varOmega}(\zeta^{-1})\|_{C^r}^2\within*\|\res^\varGamma_{\overline\varOmega}(\zeta)\|_{C^{r+1}}+\within\|\mathellipsis\|_{C^{r+1}}\within*\|\res^\varGamma_{\overline\varOmega}(\zeta^{-1})\|_{C^r}.	&&
\end{flalign*}
Plugging in the norms \eqref{eqn:12B.14.3}, this translates into the desired estimate \eqref{eqn:12B.14.10}.

\subsection*{Proof of the estimate \eqref{eqn:12B.14.15}}

Let \(m_{23},\pr_{12},\pr_{23}:\varGamma_3\to\varGamma_2\) denote the maps which to \((g',g,k)\in\varGamma_3\) associate, respectively, \((g',gk),(g',g),(g,k)\in\varGamma_2\). Let \(\d_2\nu\) stand for the Haar integration functional~\eqref{eqn:12B.10.6b}
\begin{equation*}
	\Gamma^\infty\bigl(\varGamma_3;\pr_{12}^*L(s_2^*E,t_2^*E)\bigr)\longto\Gamma^\infty\bigl(\varGamma_2;L(s_2^*E,t_2^*E)\bigr)\emphpunct, \vartheta\mapsto\within*<\vartheta,\d_2\nu>:=\integral*\vartheta d\nu
\end{equation*}
depending on parameters in \(\varGamma_2\xto{s_2}M\). By the identity \eqref{eqn:12B.12.5b}, since by virtue of the left invariance of our Haar system \(\nu\) the double integral therein occurring can be rewritten as
\begin{align*}
 \integral*\left({\integral*\Delta(g'gk,gk)dk}\right)\circ\Delta(gl,l)dl
	&=\integral*\left({\integral*\Delta(g'k',k')dk'}\right)\circ\Delta(gl,l)dl
\\	&=\within({\integral*\Delta(g'k',k')dk'})\circ\within({\integral*\Delta(gk,k)dk}),
\end{align*}
we have (making tacit use of Lemma~\ref{lem:12B.13.9} at the appropriate places)
\begin{flalign*}
&&	\within*\|\res^{\varGamma_2}_{\overline\varOmega_2}(m^*\avg\zeta-\pr_1^*\avg\zeta\circ\pr_2^*\avg\zeta)\|_{C^{r+1}}
		&\leq\within*\|\res^{\varGamma_2}_{\overline\varOmega_2}\within*<m_{23}^*a^*\Delta^\zeta\circ\pr_{23}^*a^*\Delta^\zeta,\d_2\nu>\|_{C^{r+1}}\\*	&&	&\justify+\within*\|\res^{\varGamma_2}_{\overline\varOmega_2}(\pr_1^*\within*<a^*\Delta^\zeta,\d\nu>)\circ\res^{\varGamma_2}_{\overline\varOmega_2}(\pr_2^*\within*<a^*\Delta^\zeta,\d\nu>)\|_{C^{r+1}}	&&
\\%
\makebox[.pt][l]{[by Lemma~\ref{lem:12B.20.1}: ]}
&&		&\leq\within*\|\res^{\varGamma_3}_{\overline\varOmega_3}(m_{23}^*a^*\Delta^\zeta)\circ\res^{\varGamma_3}_{\overline\varOmega_3}(\pr_{23}^*a^*\Delta^\zeta)\|_{C^{r+1}}\\*	&&	&\justify+\within*\|\res^{\varGamma_2}_{\overline\varOmega_2}(\pr_1^*\within*<\mathellipsis>)\circ\res^{\varGamma_2}_{\overline\varOmega_2}(\pr_2^*\within*<\mathellipsis>)\|_{C^{r+1}}
\\%
\makebox[.em][l]{[by Lemma~\ref{lem:12B.13.7} \eqref{eqn:12B.13.5b}: ]}
&&		&\leq\within*\|\res^{\varGamma_3}_{\overline\varOmega_3}(m_{23}^*a^*\Delta^\zeta)\|_{C^r}\within*\|\res^{\varGamma_3}_{\overline\varOmega_3}(\pr_{23}^*a^*\Delta^\zeta)\|_{C^{r+1}}\\*	&&	&\hskip 3pc\justify+\within*\|\res^{\varGamma_3}_{\overline\varOmega_3}(m_{23}^*a^*\Delta^\zeta)\|_{C^{r+1}}\within*\|\res^{\varGamma_3}_{\overline\varOmega_3}(\pr_{23}^*a^*\Delta^\zeta)\|_{C^r}\\*	&&	&\justify+\within*\|\res^{\varGamma_2}_{\overline\varOmega_2}(\pr_1^*\within<\mathellipsis>)\|_{C^r}\within*\|\res^{\varGamma_2}_{\overline\varOmega_2}(\pr_2^*\within<\mathellipsis>)\|_{C^{r+1}}\\*	&&	&\hskip 3pc\justify+\within*\|\res^{\varGamma_2}_{\overline\varOmega_2}(\pr_1^*\within<\mathellipsis>)\|_{C^{r+1}}\within*\|\res^{\varGamma_2}_{\overline\varOmega_2}(\pr_2^*\within<\mathellipsis>)\|_{C^r}
\\%
\makebox[.pt][l]{[by Lemma~\ref{lem:12B.13.10}: ]}
&&		&\leq\within*\|\res^{\varGamma_2}_{\overline\varOmega_2}(a^*\Delta^\zeta)\|_{C^r}\within*\|\res^{\varGamma_2}_{\overline\varOmega_2}(a^*\Delta^\zeta)\|_{C^{r+1}}\\*	&&	&\justify+\within*\|\res^\varGamma_{\overline\varOmega}\within*<a^*\Delta^\zeta,\d\nu>\|_{C^r}\within*\|\res^\varGamma_{\overline\varOmega}\within*<a^*\Delta^\zeta,\d\nu>\|_{C^{r+1}}
\\%
\makebox[.pt][l]{[by Lemma~\ref{lem:12B.20.1}: ]}
&&		&\leq\within*\|\res^{\varGamma_2}_{\overline\varOmega_2}(a^*\Delta^\zeta)\|_{C^r}\within*\|\res^{\varGamma_2}_{\overline\varOmega_2}(a^*\Delta^\zeta)\|_{C^{r+1}}
\\%
\makebox[.pt][l]{[by Lemma~\ref{lem:12B.13.7} \eqref{eqn:12B.13.5}: ]}
&&		&\leq\within*\|\res^{\varGamma_2}_{\overline\varOmega_2}(m^*\zeta-\pr_1^*\zeta\circ\pr_2^*\zeta)\|_{C^r}\within*\|\res^{\varGamma_2}_{\overline\varOmega_2}(\pr_2^*\zeta^{-1})\|_{C^r}\\*	&&	&\hskip 12pc\justify\times\within*\|\res^{\varGamma_2}_{\overline\varOmega_2}(a^*\Delta^\zeta)\|_{C^{r+1}}
\\%
\makebox[.pt][l]{[by Lemma~\ref{lem:12B.13.10}: ]}
&&		&\leq\within*\|\res^{\varGamma_2}_{\overline\varOmega_2}(m^*\zeta-\pr_1^*\zeta\circ\pr_2^*\zeta)\|_{C^r}\within*\|\res^\varGamma_{\overline\varOmega}(\zeta^{-1})\|_{C^r}\\*	&&	&\hskip 12pc\justify\times\within*\|\res^{\varGamma_2}_{\overline\varOmega_2}(a^*\Delta^\zeta)\|_{C^{r+1}}.
\end{flalign*}
The latter factor can be estimated as in the preceding subsection after the mark~(\dag). The desired estimate \eqref{eqn:12B.14.15} then drops out by plugging in the norms \eqref{eqn:12B.14.3}.

% 31049 Dec 18 12:02

\section{Fast convergence theorem~B (connections)}\label{sec:16a.6}

We proceed to prove Theorem~\ref{thm:12B.15.1}. As in the previous section, we are given a proper Lie groupoid \(\varGamma\tto M\) along with a normalized Haar system \(\nu\) on \(\varGamma\tto M\). We are also given a \emph{nearly effective} connection (cf.~Definition~\ref{defn:12B.14.1}) \(\varPsi\in\Conn_1(\varGamma)\); we let \(\lambda:=\lambda^\varPsi\) stand for the effect of \(\varPsi\), which by definition, is a near representation. These data will remain fixed throughout the sequel.

The first thing to see is that the sequence \(\{\avg[i]\eta:=\eta^{\avg[i]\varPsi}\emphpunct| i=0,1,2,\dotsc\}\) is Cauchy within \(\Gamma^\infty\bigl(\varGamma;L(s^*TM,T\varGamma)\bigr)\), the Fr\'e\-chet space of cross-sec\-tions of class \(C^\infty\) of the vector bundle \(L(s^*TM,T\varGamma)\), and hence convergent therein to a (unique) cross-sec\-tion, say, \(\avg[\infty]\eta\), of class \(C^\infty\). Since \(C^\infty\)\mdash convergence implies pointwise convergence\textemdash so that \(\avg[\infty]\eta_g=\lim\avg[i]\eta_g\) for all \(g\in\varGamma\)\textemdash this has to be the horizontal lift for a (unique) connection, say, \(\avg[\infty]\varPsi\) on \(\varGamma\tto M\), necessarily unital; we contend that \(\avg[\infty]\varPsi\) is multiplicative.

Proving our first claim is tantamount to showing that the the sequence of differences \(\avg[i]\eta-\avg[0]\eta\) is Cauchy within the Fr\'e\-chet (i.e., closed) subspace \(\Gamma^\infty\bigl(\varGamma;L(s^*TM,\ker{\d s})\bigr)\subset\Gamma^\infty\bigl(\varGamma;L(s^*TM,T\varGamma)\bigr)\), or alternatively by virtue of Lemma~\ref{lem:12B.13.9} on p.~\pageref{lem:12B.13.9}, that the sequence of \emph{\(\varPsi\)\mdash vertical components\/ \eqref{eqn:12B.11.12} \(\avg[i]X:=X^{\avg[i]\varPsi:\varPsi}\)} is Cauchy within \(\Gamma^\infty\bigl(\varGamma;L(s^*TM,t^*\mathfrak g)\bigr)\). Our method will be the same as before: apply Lemma~\ref{cor:12B.13.5} (Appendix~\ref{sec:16a.A}), for each finite order of derivation \(r\), to a suitable relatively compact open cover \(\{\varOmega\}\) of the manifold \(\varGamma\), namely, the one associated with the near representation \(\lambda\sidetext(=\text{ the effect of }\varPsi)\) in the way described in the preceding section.

Recall from \S\ref{sec:16a.3} that for every non-de\-gen\-er\-ate connection \(\varPhi\in\Conn_\div(\varGamma)\) there is the \emph{\(\varPhi\)\mdash vertical component} of the \emph{division cocycle} \eqref{eqn:12B.11.5} \(\Delta^\varPhi:=q_\div^*\omega\circ q_\div^*\pi^\varPhi\circ\delta^\varPhi\), which by analogy with \ref{npar:12B.12.1} \eqref{eqn:12B.12.3}, here we want to consider as a cross-sec\-tion \(\Delta^\varPhi\in\Gamma^\infty\bigl(\varGamma_\div;L(s_\div^*TM,t_\div^*\mathfrak g)\bigr)\). Similar to before, let us put \(\varOmega_\div:=\varOmega\ftimes ss\varOmega\). (By properness, this is a relatively compact open subset of \(\varGamma_\div\).) Let us fix arbitrary standard \(C^r\)\mdash norms, say,%
\begin{subequations}
\label{eqn:16a.45}
\begin{align}
 p^{[r]}_\varOmega\quad
	&\text{on}\quad	\Gamma^r\bigl(\overline\varOmega;L(s^*TM,t^*\mathfrak g)\bigr),
\label{eqn:16a.45a}
\\ q^{[r]}_\varOmega\quad
	&\text{on}\quad	\Gamma^r\bigl(\overline\varOmega_\div;L(s_\div^*TM,t_\div^*\mathfrak g)\bigr),
\label{eqn:16a.45b}
\end{align}
\end{subequations}
and then for any connections \(H,\varPhi\in\Conn(\varGamma)\) (where appropriate, non-de\-gen\-er\-ate) set
\begin{align*}
 b^{[r]}(H\mathinner:\varPhi)&:=p^{[r]}_\varOmega\bigl(\res^\varGamma_{\overline\varOmega}(X^{H:\varPhi})\bigr),
\\           c^{[r]}(\varPhi)&:=q^{[r]}_\varOmega\bigl(\res^{\varGamma_\div}_{\overline\varOmega_\div}(\Delta^\varPhi)\bigr).
\end{align*}

We are going to show that there exist positive constants \(B^{[r]}\) and \(C^{[r]}\) such that for any non-de\-gen\-er\-ate connection \(\varPhi\in\Conn_\div(\varGamma)\)%
\begin{gather}
	b^{[r]}(\avg\varPhi\mathinner:\varPhi)\leq c^{[r]}(\varPhi)B^{[r]}
\label{eqn:12B.15.11}
\\\intertext{%
and, in the notations of \S\ref{sec:16a.5}, providing that \(\avg\varPhi\) itself is non-de\-gen\-er\-ate,%
}	c^{[r]}(\avg\varPhi)\leq\bar b^{(r)}(\lambda^\varPhi)\bar b^{(r)}\bigl(\avg(\lambda^\varPhi)\bigr)c^{(r)}(\lambda^\varPhi)c^{[r]}(\varPhi)C^{[r]}.
\label{eqn:12B.15.10}
\end{gather}

Let us suppose for a moment that we have already established these estimates. Let us set%
\begin{equation*}
	c^{[r]}_i:=c^{[r]}(\avg[i]\varPsi).
\end{equation*}
From \eqref{eqn:12B.15.10}, we immediately deduce that for every \(i\)
\begin{equation*}
	c^{[r]}_{i+1}\leq\bar b^{(r)}_i\bar b^{(r)}_{i+1}c^{(r)}_ic^{[r]}_iC^{[r]}.
\end{equation*}
By Statement~S2 in the proof of Theorem~\ref{thm:12B.14.2}~(\S\ref{sec:16a.5}), the sequence \(\{\bar b^{(r)}_i\}\) has to be bounded, so there must be some positive constant \(K\) such that the inequality below holds for all \(i\),
\begin{equation*}
	c^{[r]}_{i+1}\leq c^{(r)}_ic^{[r]}_iK
\end{equation*}
whence \(c^{[r]}_1\leq c^{(r)}_0c^{[r]}_0K\emphpunct, c^{[r]}_2\leq c^{(r)}_1c^{(r)}_0c^{[r]}_0KK\emphpunct, c^{[r]}_3\leq c^{(r)}_2c^{(r)}_1c^{(r)}_0c^{[r]}_0K^2K\) and, in general, for \(i\geq i^{(r)}\), where \(i^{(r)}\) is as in Statement~S4 in the proof of Theorem~\ref{thm:12B.14.2},
\begin{align*}
 c^{[r]}_i
	&\leq\biggl(\textstyle\prod\limits_{n=0}^{i-1}c^{(r)}_n\biggr)c^{[r]}_0K^i
\\	&=\biggl(K^{i-i^{(r)}}\textstyle\prod\limits_{n=i^{(r)}}^{i-1}c^{(r)}_n\biggr)c^{[r]}_0K^{i^{(r)}}\textstyle\prod\limits_{n=0}^{i^{(r)}-1}c^{(r)}_n
\\	&\leq(K^{i-i^{(r)}}\varepsilon^{1+2+\dotsb+2^{i-i^{(r)}-1}})c^{[r]}_0K^{i^{(r)}}\textstyle\prod\limits_{n=0}^{i^{(r)}-1}c^{(r)}_n
\\	&=(K^{i-i^{(r)}}\varepsilon^{2^{i-i^{(r)}}})\varepsilon^{-1}c^{[r]}_0K^{i^{(r)}}\textstyle\prod\limits_{n=0}^{i^{(r)}-1}c^{(r)}_n.
\tag{\dag}
\end{align*}
Now, in view of the obvious identity \(X^{H:\varPsi}-X^{\varPhi:\varPsi}=X^{H:\varPhi}\), valid for any \(H,\varPhi\in\Conn(\varGamma)\), thanks to \eqref{eqn:12B.15.11} we see that for \(i\geq i^{(r)}\)%
\begin{align*}
 p^{[r]}_\varOmega\bigl(\res^\varGamma_{\overline\varOmega}(\avg[i+1]X)-\res^\varGamma_{\overline\varOmega}(\avg[i]X)\bigr)
	&=p^{[r]}_\varOmega\bigl(\res^\varGamma_{\overline\varOmega}(X^{\avg[i+1]\varPsi:\avg[i]\varPsi})\bigr)
\\	&=b^{[r]}(\avg[i+1]\varPsi\mathinner:\avg[i]\varPsi)
\\	&\leq c^{[r]}_iB^{[r]}\leq(K^{i-i^{(r)}}\varepsilon^{2^{i-i^{(r)}}})\cdot\text{constant.}
\end{align*}
From this, it follows immediately that the sequence of restrictions \(\{\res^\varGamma_{\overline\varOmega}(\avg[i]X)\}\) is Cauchy, as claimed. As \(\varOmega\) is let vary, from the estimates (\dag) it also follows that the sequence \(\{\Delta^{\avg[i]\varPsi}\}\) converges within \(\Gamma^\infty\bigl(\varGamma_\div;L(s_\div^*TM,t_\div^*\mathfrak g)\bigr)\) to zero; this implies multiplicativity of \(\avg[\infty]\varPsi\):
\begin{align*}
 &\omega_{gh^{-1}}\circ\bigl((\eta^{\avg[\infty]\varPsi}_g\mathpunct\div\eta^{\avg[\infty]\varPsi}_h)-\eta^{\avg[\infty]\varPsi}_{gh^{-1}}\lambda^{\avg[\infty]\varPsi}_h\bigr)
\\*	&\hskip 3em=\omega_{gh^{-1}}\circ\bigl(T_{(g,h)}q_\div\mathinner\circ(\avg[\infty]\eta_g,\avg[\infty]\eta_h)-\avg[\infty]\eta_{gh^{-1}}\mathinner\circ T_ht\mathinner\circ\avg[\infty]\eta_h\bigr)
\\	&\hskip 3em=\lim{}\bigl\{\omega_{gh^{-1}}\circ\bigl(T_{(g,h)}q_\div\mathinner\circ(\avg[i]\eta_g,\avg[i]\eta_h)-\avg[i]\eta_{gh^{-1}}\mathinner\circ T_ht\mathinner\circ\avg[i]\eta_h\bigr)\bigr\}
\\	&\hskip 3em=\lim{}\bigl\{\omega_{gh^{-1}}\circ\bigl((\avg[i]\eta_g\mathinner\div\avg[i]\eta_h)-\avg[i]\eta_{gh^{-1}}\avg[i]\lambda_h\bigr)\bigr\}
\\	&\hskip 3em=\lim{}\bigl\{\omega_{gh^{-1}}\circ\bigl(\delta^{\avg[i]\varPsi}(g,h)-\eta^{\avg[i]\varPsi}_{gh^{-1}}\bigr)\circ\avg[i]\lambda_h\bigr\}
\\	&\hskip 3em=\bigl\{\lim\Delta^{\avg[i]\varPsi}(g,h)\bigr\}\circ\avg[\infty]\lambda_h=0,
\end{align*}
which proves that \(\avg[\infty]\varPsi\) satisfies the multiplicativity condition \eqref{eqn:12B.11.4} and hence (being unital, as we know) is multiplicative.

\begin{stmt}\label{stmt:16a.6.1} The equations below\textemdash which are the analogs for connections of\/ \eqref{eqn:12B.12.5b}, \eqref{eqn:12B.12.5a}\textemdash hold:%
\begin{subequations}
\label{eqn:12B.15.5}
\begin{align}
\begin{split}
 \Delta^{\avg\varPhi}(g,h)\avg\lambda^\varPhi_h
	&=\integral*\Delta^\varPhi(gk,hk)\circ\Delta^{\lambda^\varPhi}(hk,k)dk
\\	&\justify-\integral*\Delta^\varPhi(gk,hk)dk\circ\integral*\Delta^{\lambda^\varPhi}(hk,k)dk
\end{split}
\label{eqn:12B.15.5b}
\\ X^{\avg\varPhi:\varPhi}_g
	&=\integral*\Delta^\varPhi(gk,k)dk
\label{eqn:12B.15.5a}
\end{align}
\end{subequations} \end{stmt}

\begin{proof} For every non-de\-gen\-er\-ate connection \(H\in\Conn_\div(\varGamma)\) the difference \(\delta^H(g,h)-\eta^H_{gh^{-1}}\) is a linear map of \(T_{th}M\) into \(\ker T_{gh^{-1}}s\), and
\begin{equation*}
	\Delta^H(g,h)=\omega_{gh^{-1}}\circ\bigl(\delta^H(g,h)-\eta^H_{gh^{-1}}\bigr),
\end{equation*}
so for any other non-de\-gen\-er\-ate groupoid connection \(\varPhi\),
\begin{align*}
 \Delta^H(g,h)\lambda^H_h
	&=\omega_{gh^{-1}}\circ\bigl(\delta^H(g,h)-\eta^H_{gh^{-1}}\bigr)\circ\lambda^H_h
\\	&=\omega_{gh^{-1}}\circ\pi^\varPhi_{gh^{-1}}\circ\bigl(\delta^H(g,h)\lambda^H_h-\eta^H_{gh^{-1}}\lambda^H_h\bigr)
\\	&=\pr_1\circ\sigma^\varPhi_{gh^{-1}}\circ\bigl((\eta^H_g\div\eta^H_h)-\eta^H_{gh^{-1}}\lambda^H_h\bigr)
&&\text{by definition \eqref{eqn:12B.11.8},}
\\	&=\dot q^\varPhi_\updownarrow(g,h)(X^H_g,X^H_h)+\Delta^\varPhi(g,h)\lambda^\varPhi_h-X^H_{gh^{-1}}\lambda^H_h
&&\text{by \eqref{eqn:12B.11.11} and \eqref{eqn:12B.11.16}.}
\end{align*}
Under the provision that \(\avg\varPhi\) be itself non-de\-gen\-er\-ate, making \(H:=\avg\varPhi\), setting \(\avg X^\varPhi:=X^{\avg\varPhi:\varPhi}\), and referring back to our computations in the proof of Proposition~\ref{prop:12B.11.7}~(\S\ref{sec:16a.3}),
\begin{align*}
 \Delta^{\avg\varPhi}(g,h)\avg\lambda^\varPhi_h
	&=\dot q^\varPhi_\updownarrow(g,h)(\avg X^\varPhi_g,\avg X^\varPhi_h)-\avg X^\varPhi_{gh^{-1}}\avg\lambda^\varPhi_h+\Delta^\varPhi(g,h)\lambda^\varPhi_h
\\	&=\iintegral*\Delta^\varPhi(gk,hk)\circ\dot s^\varPhi_\updownarrow(h)\circ\bigl(\Delta^\varPhi(hk,k)-\Delta^\varPhi(hk',k')\bigr)dkdk'
\\	&=\integral*\Delta^\varPhi(gk,hk)\circ\Delta^{\lambda^\varPhi}(hk,k)dk-\iintegral*\Delta^\varPhi(gk,hk)\circ\Delta^{\lambda^\varPhi}(hk',k')dkdk',
\end{align*}
which evidently gives \eqref{eqn:12B.15.5b}. The other identity, \eqref{eqn:12B.15.5a}, is equally clear:
\begin{align*}
 X^{\avg\varPhi:\varPhi}_g=\omega_g\circ(\avg\eta^\varPhi_g-\eta^\varPhi_g)
	&=\omega_g\circ\integral*\bigl(\delta^\varPhi(gk,k)-\eta^\varPhi_g\bigr)dk
\\	&=\integral*\omega_g\circ\bigl(\delta^\varPhi(gk,k)-\eta^\varPhi_g\bigr)dk=\integral*\Delta^\varPhi(gk,k)dk.	\qedhere%
\end{align*} \end{proof}

\begin{npar}\label{npar:16a.6.2} {\it Proof of the estimate\/ \eqref{eqn:12B.15.11}}\emphpunct: Letting \(\d\nu\) indicate the integration functional \eqref{eqn:12B.10.6b}
\begin{equation*}
	\Gamma^\infty\bigl(\varGamma_2;\pr_1^*L(s^*TM,t^*\mathfrak g)\bigr)\longto\Gamma^\infty\bigl(\varGamma;L(s^*TM,t^*\mathfrak g)\bigr)\emphpunct, \vartheta\mapsto\within*<\vartheta,\d\nu>:=\integral*\vartheta d\nu
\end{equation*}
depending on parameters in \(\varGamma\xto sM\), in view of the identity \eqref{eqn:12B.15.5a} and of Lemma~\ref{lem:12B.20.1}, we have
\begin{align*}
 \within*\|\res^\varGamma_{\overline\varOmega}(X^{\avg\varPhi:\varPhi})\|_{C^r}
	&=\within*\|\res^\varGamma_{\overline\varOmega}\within*<a^*\Delta^\varPhi,\d\nu>\|_{C^r}
\\	&\leq\within*\|\res^{\varGamma_\div}_{\overline\varOmega_\div}(\Delta^\varPhi)\|_{C^r},
\end{align*}
whence the desired estimate \eqref{eqn:12B.15.11} by plugging in the specified standard \(C^r\)\mdash norms \eqref{eqn:16a.45}. \end{npar}

\begin{npar}\label{npar:16a.6.3} {\it Proof of the estimate \eqref{eqn:12B.15.10}}\emphpunct: Let \(\pr_\div:\varGamma_\div\to\varGamma\) indicate the map \((g,h)\mapsto h\); further, let \(r_\div:\varGamma_\div\to M\) indicate the map \((g,h)\mapsto sg=sh\), and then \(m_\div:\varGamma_\div\ftimes{r_\div}t\varGamma\to\varGamma_\div\) the map \((g,h;k)\mapsto(gk,hk)\). Also, let \(\d_\div\nu\) denote the appropriate Haar integration functional \eqref{eqn:12B.10.6b} depending on parameters in \(\varGamma_\div\xto{r_\div}M\) associated with the vector bundle \(L(r_\div^*TM,t_\div^*\mathfrak g)\), resp., \(L(s_\div^*TM,t_\div^*\mathfrak g)\). Finally, let \(\d\nu\) have the same meaning as in the previous section. By virtue of \eqref{eqn:12B.15.5b} and of the appropriate lemmas from the appendices \ref{sec:16a.A}~and \ref{sec:16a.B}, we have
\begin{align*}
 \within*\|\res^{\varGamma_\div}_{\overline\varOmega_\div}(\Delta^{\avg\varPhi})\|_{C^r}
	&=\within*\|\res^{\varGamma_\div}_{\overline\varOmega_\div}(\Delta^{\avg\varPhi}\mathpunct\circ\pr_\div^*\avg\lambda^\varPhi)\circ\res^{\varGamma_\div}_{\overline\varOmega_\div}(\pr_\div^*\avg\lambda^\varPhi)^{-1}\|_{C^r}
\\		&\leq\within*\|\res^{\varGamma_\div}_{\overline\varOmega_\div}(\Delta^{\avg\varPhi}\mathpunct\circ\pr_\div^*\avg\lambda^\varPhi)\|_{C^r}\within*\|\res^{\varGamma_\div}_{\overline\varOmega_\div}\bigl(\pr_\div^*(\avg\lambda^\varPhi)^{-1}\bigr)\|_{C^r}
\\	&\leq\|\res^{\varGamma_\div}_{\overline\varOmega_\div}\within*<m_\div^*\Delta^\varPhi\mathpunct\circ(\pr_\div\mathpunct\times\id)^*a^*\Delta^{\lambda^\varPhi},\d_\div\nu>
\\	&\phantom{\|}\justify-\res^{\varGamma_\div}_{\overline\varOmega_\div}(\within*<m_\div^*\Delta^\varPhi,\d_\div\nu>\circ\pr_\div^*\within*<a^*\Delta^{\lambda^\varPhi},\d\nu>)\|^{}_{C^r}\within*\|\res^\varGamma_{\overline\varOmega}(\avg\lambda^\varPhi)^{-1}\|_{C^r}
\\	&\leq\Bigl(\within*\|\res^{\varGamma_\div\ftimes{r_\div}t\varGamma}_{\overline{\varOmega_\div\ftimes{r_\div}t\varOmega}}\bigl(m_\div^*\Delta^\varPhi\mathpunct\circ(\pr_\div\mathpunct\times\id)^*a^*\Delta^{\lambda^\varPhi}\bigr)\|_{C^r}
\\	&\phantom{\Bigl(}\justify+\within*\|\res^{\varGamma_\div}_{\overline\varOmega_\div}\within*<m_\div^*\Delta^\varPhi,\d_\div\nu>\|_{C^r}\within*\|\res^{\varGamma_\div}_{\overline\varOmega_\div}(\pr_\div^*\within*<a^*\Delta^{\lambda^\varPhi},\d\nu>)\|_{C^r}\Bigr)\within*\|\res^\varGamma_{\overline\varOmega}(\avg\lambda^\varPhi)^{-1}\|_{C^r}
\\	&\leq\Bigl(\within*\|\res^{\varGamma_\div\ftimes{r_\div}t\varGamma}_{\overline{\varOmega_\div\ftimes{r_\div}t\varOmega}}(m_\div^*\Delta^\varPhi)\|_{C^r}\within*\|\res^{\varGamma_\div\ftimes{r_\div}t\varGamma}_{\overline{\varOmega_\div\ftimes{r_\div}t\varOmega}}\bigl((\pr_\div\mathpunct\times\id)^*a^*\Delta^{\lambda^\varPhi}\bigr)\|_{C^r}
\\	&\phantom{\Bigl(}\justify+\within*\|\res^{\varGamma_\div\ftimes{r_\div}t\varGamma}_{\overline{\varOmega_\div\ftimes{r_\div}t\varOmega}}(m_\div^*\Delta^\varPhi)\|_{C^r}\within*\|\res^\varGamma_{\overline\varOmega}\within*<a^*\Delta^{\lambda^\varPhi},\d\nu>\|_{C^r}\Bigr)\within*\|\res^\varGamma_{\overline\varOmega}(\avg\lambda^\varPhi)^{-1}\|_{C^r}
\\	&\leq\within*\|\res^{\varGamma_\div}_{\overline\varOmega_\div}(\Delta^\varPhi)\|_{C^r}\within*\|\res^{\varGamma_2}_{\overline\varOmega_2}(a^*\Delta^{\lambda^\varPhi})\|_{C^r}\within*\|\res^\varGamma_{\overline\varOmega}(\avg\lambda^\varPhi)^{-1}\|_{C^r}
\\	&\leq\within*\|\res^{\varGamma_\div}_{\overline\varOmega_\div}(\Delta^\varPhi)\|_{C^r}\within*\|\res^{\varGamma_2}_{\overline\varOmega_2}(m^*\lambda^\varPhi-\pr_1^*\lambda^\varPhi\circ\pr_2^*\lambda^\varPhi)\|_{C^r}
\\	&\hskip 12em\justify[+]\times\within*\|\res^\varGamma_{\overline\varOmega}(\lambda^\varPhi)^{-1}\|_{C^r}\within*\|\res^\varGamma_{\overline\varOmega}(\avg\lambda^\varPhi)^{-1}\|_{C^r},
\end{align*}
whence after plugging in our standard \(C^r\)\mdash norms \eqref{eqn:16a.45}, the desired estimate \eqref{eqn:12B.15.10}. \end{npar}
% 15756 Dec 21 10:26

\section{Reduction to the regular case}\label{sec:16a.7}

We are finally in a position to prove the principal result of this paper, Theorem~\ref{prop:14A.5.2}. We start by establishing a special case of the theorem, Corollary~\ref{cor:14A.4.3}, which is substantially a direct consequence of the fast convergence result proved in the last section. By a fairly classical argument (cf.~\cite[Proof of Theorem~7.1]{Wein}), this corollary can be seen to imply the local linearizability of proper Lie groupoids around fix-points. Hence, as a byproduct of our fast convergence results, we obtain yet another proof of the linearization theorem, alternative to Zung's \cite{Zung} and perhaps closer in spirit to Wein\-stein's original idea of averaging the action of a certain ``near group'' of bisections recursively (see the introduction to \cite{Wein00}). With hindsight, we could say that our averaging operator \eqref{eqn:16a.8} is the missing ingredient which was needed to generalize Wein\-stein's arguments from bundles of compact Lie groups to arbitrary proper Lie groupoids. In any case, we will not concern ourselves with linearization problems here. The applications of Theorem~\ref{prop:14A.5.2}\textemdash which, we stress, is a global result\textemdash go well beyond Corollary~\ref{cor:14A.4.3}.

We begin with a general remark. Let \(\lambda\in\Psr(\varGamma;E)\) be an arbitrary pseudo-rep\-re\-sen\-ta\-tion of our proper Lie groupoid \(\varGamma\tto M\) on some vector bundle \(E\) over \(M\), and let \(S\) be an invariant subset of \(M\) along which \(\lambda\) is multiplicative, meaning that \(\lambda(1_x)=\id\) for all base points \(x\in S\) and that \(\lambda(g_1g_2)=\lambda(g_1)\lambda(g_2)\) for any composable arrows \(g_1,g_2\in\varGamma(S,S)\). Let us endow \(E\) with some vec\-tor-bun\-dle metric. By averaging the metric with respect to \(\lambda\), we can always arrange for the situation when \(\lambda(g)\) is an isometry of \(E_{sg}\) onto \(E_{tg}\) for every arrow \(g\in\varGamma(S,S)\).

\begin{lem}\label{lem:14A.4.1} Suppose that\/ $U$ is an open subset of\/ $M$ such that\/ $\varGamma U=M$ and such that the intersection\/ $\overline U\cap\varGamma K$ is compact for any compact set\/ $K\subset U$. Then within\/ $U$ one can find a relatively invariant open neighborhood\/ $V=U\cap\varGamma V$ of the intersection\/ $U\cap\overline S$ with the property that for all arrows\/ $g$ and composable pairs of arrows\/ $g_1,g_2$ in $\varGamma(V,V)$ the following estimates, involving the operator norms\/ \eqref{eqn:16a.11}, hold.%
\begin{subequations}
\label{eqn:14A.4.1}
\begin{gather}
	\label{eqn:14A.4.1a}	\within*\|\lambda(g)\|<\sqrt 2
\\	\label{eqn:14A.4.1b}	\within*\|\lambda(g_1g_2)-\lambda(g_1)\circ\lambda(g_2)\|<\frac 1{18}
\end{gather}
\end{subequations} \end{lem}

\begin{proof} The union \(\bigcup V_i\) of any family of relatively invariant open subsets \(V_i\subset U\) over which the estimates \eqref{eqn:14A.4.1} hold must itself be one such subset of \(U\). Thus, if we set \(C=\overline S\), it will be enough to show that each point \(x\in C\cap U\) admits some relatively invariant open neighborhood over which the desired estimates hold. Consider the open neighborhoods
\begin{alignat*}{5}
 \varOmega&:=\{g\in\varGamma:\within*\|\lambda_g\|<\sqrt 2\}
&\mskip 7mu&\text{of}&\mskip 7mu
 \varSigma&:=s^{-1}(C)
&\mskip 7mu&\text{within}&\mskip 7mu
&\varGamma\text,
\\ \varOmega_2&:=\biggl\{(g_1,g_2)\in\varGamma_2:\within*\|\lambda_{g_1g_2}-\lambda_{g_1}\lambda_{g_2}\|<\frac 1{18}\biggr\}
&\mskip 7mu&\text{of}&\mskip 7mu
 \varSigma_2&:=\varSigma\ftimes st\varSigma
&\mskip 7mu&\text{within}&\mskip 7mu
&\varGamma_2\text.
\end{alignat*}
Pick a sequence \(B_n\ni x\) of relatively compact open balls with \(B_n\supset\overline B_{n+1}\) and \(\bigcap B_n=\{x\}\). Put \(V_n:=\varGamma B_n\cap U\emphpunct, K_n:=\varGamma\overline B_n\cap\overline U\), and \(A_n:=s^{-1}(K_n)\cap t^{-1}(K_n)\). By properness, for \(n\) large we will have \(A_n\subset\varOmega\) and \(A_n\ftimes stA_n\subset\varOmega_2\). Since \(\varGamma(V_n,V_n)\subset A_n\) and \(\varGamma(V_n,V_n)\ftimes st\varGamma(V_n,V_n)\subset A_n\ftimes stA_n\), the neighborhood \(V_n\ni x\) will have the required properties. \end{proof}

Fix an arbitrary sequence $U_0,U_1,U_2,\dotsc$ of open sets as in the lemma with $U_0\subset U_1\subset U_2\subset\dotsb$ and $\bigcup_{n=0}^\infty U_n=M$. For each $n$ let $V_n$ be a relatively invariant open neighborhood of the intersection $U_n\cap\overline S$ within $U_n$ such that the estimates \eqref{eqn:14A.4.1} hold for all arrows $g$ and composable pairs of arrows $g_1,g_2$ in $\varGamma(V_n,V_n)$, whose existence is guaranteed by the lemma. Observe that when $m<n$,
\begin{align*}
 V_m\cap\varGamma V_n&=(V_m\cap U_m)\cap\varGamma V_n
\\	&\subset V_m\cap(U_n\cap\varGamma V_n)
\\	&=V_m\cap V_n.
\end{align*}
Setting $V=\bigcup_{n=0}^\infty V_n$, if $g_1,g_2\in\varGamma(V,V)$ are composable then for some $n$ both $g_1$ and $g_2$ must lie within $\varGamma(V_n,V_n)$, so by \eqref{eqn:14A.4.1} the following estimate must be true,
\begin{equation*}
	\within*\|\lambda_{g_1g_2}-\lambda_{g_1}\lambda_{g_2}\|<\frac 19\mathinner{\Bigl(\sup_{g\in\varGamma(V,V)}\within*\|\lambda_g\|\Bigr)}^{-2}
\end{equation*}
which tells one precisely that the restriction to \(\varGamma\mathbin|V\tto V\) of \(\lambda\) is a near representation (at least when $\lambda$ is unital). We have therefore proved the first half of the following result:

\begin{prop}\label{prop:14A.4.2} Let\/ \(\varGamma\tto M\) be any Lie groupoid which is proper. Let\/ \(H\) be an arbitrary unital connection on it and let\/ \(S\subset M\) be an invariant subset along which\/ \(H\) is effective. There exist open neighborhoods\/ $V$ of\/ $\overline S$ within\/ $M$ for which\/ $H\mathbin|V$ on\/ $\varGamma\mathbin|V\tto V$ is nearly effective; when\/ $\varGamma\tto M$ is source proper, $V$ can be taken invariant. \end{prop}

\noindent The remaining half (concerning the source-prop\-er situation) can be proved by observing that every neighborhood $V$ of an orbit $\varGamma x$ must contain some invariant open tube around $x$ and then taking the union of all such tubes as $x$ ranges across $\overline S$.

As an immediate application of Proposition~\ref{prop:14A.4.2} (in combination with our fast convergence theorem~B\textemdash Theorem~\ref{thm:12B.15.1}) we get the following special case of Theorem~\ref{prop:14A.5.2}:

\begin{cor}\label{cor:14A.4.3} Let\/ $\varGamma\tto M$ be an arbitrary proper Lie groupoid. Over some open neighborhood\/ $V$ of its semi-fixed regular locus\/ \(M_0\) the groupoid admits multiplicative connections; when\/ $\varGamma\tto M$ is source proper, such a neighborhood\/ $V$ can be taken invariant. \end{cor}

\subsection*{Proof of Theorem~\ref{prop:14A.5.2}}

Let us now turn to the demonstration of the general theorem. For clarity, we are going to subdivide our argument into steps.

\vskip\topsep\noindent{\it (Step~0; Preliminaries.)} To begin with, let us point out that it is not restrictive to argue under the extra hypothesis that \(C\cup Z\) be closed within \(M\). This can always be achieved by taking \(M\) smaller around \(U\cup Z\) and then replacing \(C\) with a slightly larger invariant closed set, as follows. Since \(Z\subset M\) is a differentiable submanifold, we can find \(A\supset Z\) open such that \(A\cap\overline Z\subset Z\). At the expense of substituting \(M\) with \(A\cup U\), we may assume that \(\overline Z\smallsetminus Z\subset U\). If we now further replace \(C\) with \(C\cup(\overline Z\smallsetminus Z)\), we accomplish the desired situation.

For any choice of a Riemannian metric on \(M\), there is a normal vector bundle \(NZ\) over \(Z\) defined at every \(z\in Z\) by setting \(N_zZ:=(T_zZ)^\bot\), along with a canonical vec\-tor-bun\-dle decomposition \[%
	TM\mathbin|Z\cong TZ\oplus NZ.
\] For every \(x\in M\), let \(\varLambda_x:=T_x(\varGamma x)\). Since \(Z\) is invariant, the infinitesimal effect (see \cite[\S 1]{Tre6}) of each arrow \(g\in\varGamma(Z,Z)\) will carry \(T_{sg}Z\mathbin/\varLambda_{sg}\subset T_{sg}M\mathbin/\varLambda_{sg}\) into \(T_{tg}Z\mathbin/\varLambda_{tg}\subset T_{tg}M\mathbin/\varLambda_{tg}\) and therefore descend, along the quotient projections \(T_zM\mathbin/\varLambda_z\to T_zM\mathbin/T_zZ\cong N_zZ\), to a well-de\-fined linear map, hereafter denoted $\nu_Z(g)$, of $N_{sg}Z$ into $N_{tg}Z$; with this, we obtain a well-de\-fined smooth representation \(\nu_Z\) of \(\varGamma\mathbin|Z\tto Z\) on \(NZ\).

\vskip\topsep\noindent{\it (Step~1.)} We proceed to prove that at the expense of taking \(M\) smaller around \(C\cup Z\) and \(U\) smaller around \(C\), a Riemannian metric can be found on \(M\) which, on \(U\), is invariant under \(\lambda^\varPhi\) and, along \(Z\), is invariant under \(\lambda^\varTheta\). Let us choose an arbitrary open set $U'$ with $C\subset U'\subset\overline{U'}\subset U$. Clearly, one can find metrics $\phi_0$ on $M$ which are $\lambda^\varPhi$\mdash invariant on $U'$. Making use of the above orthogonal decomposition with respect to any such metric, we may promote $\lambda^\varTheta:\varGamma\mathbin|Z\to\GL(TZ)$ to a representation
\begin{equation*}
\begin{pmatrix}
 \lambda^\varTheta & 0
\\               0 & \nu_Z
\end{pmatrix}:\varGamma\mathbin|Z\longto\GL(TM\mathbin|Z),
\end{equation*}
which, by the invariance under $\lambda^\varPhi$ of the selected metric, will agree with $\lambda^\varPhi$ over $Z\cap U'$. Next, we may replace $M$ with a smaller open neighborhood $M'$ of $U'\cup Z$ so that we are able to extend this representation to a unital pseudo-rep\-re\-sen\-ta\-tion of $\varGamma\mathbin|M'$ on $TM'$. By averaging the metric $\phi_0$ along the invariant subset $Z$ with respect to this pseudo-rep\-re\-sen\-ta\-tion, we construct a metric $\phi_1$ on $M'$ whose restriction $\incl_Z^*\phi_1$ to $Z$ is $\lambda^\varTheta$\mdash invariant. By further averaging $\phi_1$ with respect to $\lambda^\varPhi$, we construct a metric $\phi_2$ on $M'\cap U$ which is both $\lambda^\varPhi$\mdash invariant and satisfies \[\incl_{U\cap Z}^*\phi_2=\incl_{U\cap Z}^*\phi_1.\] By using a partition of unity subordinated to the open cover \(M'\cap U\emphpunct*,M'\setminus\overline{U'}\) of \(M'\), we may finally glue $\phi_1$ together with $\phi_2$ in order to obtain a metric on $M'$ which is both $\lambda^\varPhi$\mdash invariant on $U'$ and $\lambda^\varTheta$\mdash invariant along $Z$, as desired.

\vskip\topsep\noindent{\it (Step~2.)} Let $Z'\subset Z$ be any invariant differentiable submanifold which is both \emph{homogeneous}\textemdash in the sense that it consists of orbits all of the same dimension\textemdash and such that $(C\cap Z)\cup Z'$ is closed within $Z$. We proceed to show that there must be some open neighborhood \(V'\subset M\) of $C\cup Z'$ together with some multiplicative connection $\varPsi'$ on $\varGamma\mathbin|V'\tto V'$ such that $\varPsi'$ coincides with $\varPhi$ near $C$ and such that $\varPsi'\mathbin|V'\cap Z=\varTheta\mathbin|Z\cap V'$.

At the expense of shrinking \(M\) around \(C\cup Z\) and \(U\) around \(C\), by following the same order of ideas as in the proof of Lemma~\ref{lem:14A.3.1}, we can construct some unital connection on \(\varGamma\tto M\) which coincides with \(\varPhi\) over \(U\), induces \(\varTheta\) along \(Z\), and is effective along \(Z'\). In detail, let us put $\varSigma:=s^{-1}(Z)$ and $\varSigma':=s^{-1}(Z')$. Let us consider an arbitrary splitting say $\sigma:(s^*TM)\mathbin|\varSigma\to T\varGamma\mathbin|\varSigma$ to the following epimorphism of vector bundles over $\varSigma$,
\begin{equation*}
	T\varGamma\mathbin|\varSigma\xto{\d s\mathbin|\varSigma}(s^*TM)\mathbin|\varSigma
\end{equation*}
with the property that \(\sigma_{1z}=T_z1\) for all \(z\in Z\) and which coincides with \(\eta^\varPhi\) over \(Z\cap U\) (such splittings can always be found, perhaps at the cost of shrinking $U$ around $C$ a bit). The vec\-tor-bun\-dle morphism $\eta^{\varTheta,\sigma}$ of $(s^*TM)\mathbin|\varSigma=(s\mathbin|\varSigma)^*(TM\mathbin|Z)=(s\mathbin|\varSigma)^*TZ\oplus(s\mathbin|\varSigma)^*NZ$ into $T\varGamma\mathbin|\varSigma$ given at all $g\in\varSigma$ by
\begin{equation*}
	\eta^{\varTheta,\sigma}_g:=\eta^\varTheta_g\circ\pr_{T_{sg}Z}+\sigma_g\circ\pr_{N_{sg}Z}:T_{sg}M=T_{sg}Z\oplus N_{sg}Z\to T_g\varGamma
\end{equation*}
evidently satisfies the condition $\d s\circ\eta^{\varTheta,\sigma}=\id$, is unital, coincides with $\eta^\varPhi$ over $Z\cap U$, and restricts to $\eta^\varTheta$ along $Z$. In general, its effect will not be a representation of $\varSigma\tto Z$ on $TM\mathbin|Z$. However, it will be possible to ``correct'' $\sigma$ so that the effect of $\eta^{\varTheta,\sigma}$ becomes a representation {\it along\/ $Z'$}. Namely, let us consider the following di\-rect-sum decomposition of differentiable vector bundles over $Z'$,
\begin{equation*}
	TM\mathbin|Z'=\varLambda'\oplus T'\oplus N'
\end{equation*}
where for every $z'\in Z'$ we set \(\varLambda'_{z'}:=\varLambda_{z'}\emphpunct, T'_{z'}:=(\varLambda'_{z'})^\bot\cap T_{z'}Z\), and \(N'_{z'}:=N_{z'}Z\); we stress that \(\varLambda'\) and \(T'\) must be differentiable vector bundles in virtue of the hypothesis that $Z'$ is homogeneous. Because of our thoughtful choice of metric {\it (Step~1)}, the effect of an arrow $g'\in\varSigma'$ under $\eta^{\varTheta,\sigma}$ expressed in matrix form relative to this decomposition will read
\begin{equation*}
	\left(%
\setlength\arraycolsep{.em}%
\renewcommand*{\arraystretch}{1.3}%
\begin{array}{ccccc}
	\lambda^\varTheta_{\varLambda'}(g')
\mkern8mu	&%
		&		0	&\vline%
					&\mkern8mu	\delta(g')
\\\cline{2-5}
	\makebox[.em]{$0$}
\mkern8mu	&\vline%
		&\mkern8mu	\lambda^\varTheta_{T'}(g')
			\mkern8mu	&\vline%
					&\mkern8mu	\mu(g')
\\\cline{1-4}
	\makebox[.em]{$0$}
\mkern8mu	&\vline%
		&		0	&%
					&\mkern8mu	\nu(g')
\end{array}\right):T_{sg'}M\to T_{tg'}M.
\end{equation*}
Let us pack the linear maps $\delta(g'):N'_{sg'}\to \varLambda'_{tg'}$ together into a morphism $\delta:(s\mathbin|\varSigma')^*N'\to(t\mathbin|\varSigma')^*\varLambda'$. As in the proof of Lemma~\ref{lem:14A.3.1}, we may lift $\delta$ to a $(\ker{\d s})\mathbin|\varSigma'$~valued morphism of vector bundles over $\varSigma'$. Now $\varSigma'\subset\varSigma$ is a differentiable submanifold, and $\delta$ vanishes on the units and over $Z'\cap U$, hence we can extend this lift of $\delta$ to a vec\-tor-bun\-dle morphism $\zeta:(s^*TM)\mathbin|\varSigma\to(\ker{\d s})\mathbin|\varSigma$ vanishing on the units and over $Z\cap U$. Let us replace $\sigma$ with $\sigma-\zeta$. With this new splitting, the matrix for the effect of \(g'\) will look like
\begin{equation*}
	\left(%
\setlength\arraycolsep{.44em}%
\renewcommand*{\arraystretch}{1.3}%
\begin{array}{c|cc}
	\lambda^\varTheta_{\varLambda'}(g')
	&	0
		&	0
\\\hline	0
	&	\lambda^\varTheta_{T'}(g')
		&	\mu(g')
\\	0
	&	0
		&	\nu(g')
\end{array}\right)
\end{equation*}
which defines a {\em representation}\/ of $\varSigma'\tto Z'$ on $TM\mathbin|Z'$. By a standard partition of unity argument, the partial vec\-tor-bun\-dle section $\eta^{\varTheta,\sigma-\zeta}$ of $L(s^*TM,T\varGamma)$ (defined over the submanifold $\varSigma$ of $\varGamma$) may be extended to a unital global section $\eta^H$ satisfying the condition $\d s\circ\eta^H=\id$ and coinciding with $\eta^\varPhi$ over a suitable open sub-do\-main of $U$. Hence the $H$ defined by $\eta^H$ will be a unital connection on $\varGamma\tto M$ with the desired properties.

So, let \(H\) be a unital connection on \(\varGamma\tto M\) which coincides with \(\varPhi\) over \(U\), induces \(\varTheta\) along \(Z\), and is effective along \(Z'\). By Proposition~\ref{prop:14A.4.2}, there must be some open neighborhood $V'$ of $C\cup Z'$ over which $H$ is nearly effective. By confining attention to the restriction of $\varGamma$ over $V'$, we may without loss of generality suppose that $H$ is nearly effective over the whole $M$ in other words that $V'=M$.

If we choose a random normalized Haar system say $\nu$ on $\varGamma\tto M$, the averaging limit connection $\hat H^{(\infty)}$ relative to $\nu$ will be multiplicative and induce $\varTheta$ along $Z$. However, it will not in general restrict to $\varPhi$ over a neighborhood of $C$. To accomplish this, we must select $\nu$ more carefully. Let us pick an arbitrary open set $U'$ such that $C\subset U'\subset\overline{U'}\subset U$ and such that $\varGamma\overline{U'}$ is closed. Let us consider the open set
\begin{equation*}
	W:=U\cup(M\setminus\varGamma\overline{U'}).
\end{equation*}
Clearly, $M=\varGamma W$. If we choose $\nu$ subordinated to $W$, now, the restriction over $U'$ of the averaging limit connection $\hat H^{(\infty)}$ relative to $\nu$ will coincide with $\varPhi\mathbin|U'$, as desired.

\vskip\topsep\noindent{\it (Step~3.)} The above argument constitutes enough of a proof if $Z$ itself is homogeneous, in which case we take $Z':=Z$. For the general case, we need to argue a bit further. Let us set \(C_0:=C\emphpunct, U_0:=U\) and, inductively, $C_{q+1}:=C_q\cup Z_q$. We have that $C_q\cap Z=(C\cap Z)\cup Z_0\cup\dotsb\cup Z_{q-1}$ is a closed subset of $Z$ and therefore that \(C_q\) is a closed subset of \(M\), since by hypothesis so was \(C\cup Z\) {\it (Step~0)}. We proceed by induction on $q$. Suppose we have found some open neighborhood $U_q$ of $C_q$ along with some multiplicative connection $\varPsi_q$ on $\varGamma\mathbin|U_q\tto U_q$ such that $\varPsi_q$ coincides with $\varPhi$ near $C$ and such that $\varPsi_q\mathbin|U_q\cap Z=\varTheta\mathbin|Z\cap U_q$. Then, by {\it Step~2}\/ applied to $C_q$ and $Z_q$, we can find some open neighborhood $U_{q+1}$ of $C_{q+1}$ together with some multiplicative connection $\varPsi_{q+1}$ on $\varGamma\mathbin|U_{q+1}\tto U_{q+1}$ such that $\varPsi_{q+1}$ coincides with $\varPsi_q$ near $C_q$ (hence near $C$) and such that $\varPsi_{q+1}\mathbin|U_{q+1}\cap Z=\varTheta\mathbin|Z\cap U_{q+1}$. Proceeding in this way we eventually (for $q$ big enough) reach some open neighborhood $U_q$ of $C\cup Z$ which has the properties required of $V$ in the statement of the theorem.
% 18269 Dec 24 17:05

\section{More examples}\label{sec:16a.8}

In \S\ref{sec:16a.1}, we have explained how the extension problem~\ref{prob:16a.1.9} for a general proper Lie groupoid \(\varGamma\tto M\) reduces to the case when \(S=\ast\sidetext(\text{pa\-ram\-e\-ter-free case})\) and \(\varGamma\tto M\) is also regular. In the present section, we focus on the latter special case of the problem, which we restate below for convenience:

\begin{prob}\label{prob:16a.8.1} Given any Lie groupoid \(\varGamma\tto M\) that is both proper and regular, any pair of open sets \(U,V\subset M\) such that \(\overline U\subset V\), and any multiplicative connection \(\varPhi\in\Mcon(\varGamma\mathbin|V)\) defined over \(V\), what are the obstructions to extending \(\varPhi\mathbin|U\in\Mcon(\varGamma\mathbin|U)\) to a connection which is both globally defined on \(\varGamma\tto M\) and multiplicative? \end{prob}

From Lemma~\ref{lem:14A.3.1} (or better, the commentary following its proof) in combination with Proposition~\ref{prop:12B.11.7}, one deduces immediately that \(\varPhi\mathbin|U\in\Mcon(\varGamma\mathbin|U)\) can be extended to a multiplicative connection defined on all of \(\varGamma\tto M\) if, and only if, its longitudinal effect \(\lambda^{\varPhi\mathbin|U}_\varLambda\in\Rep(\varGamma\mathbin|U;\varLambda\mathbin|U)\) can be extended to a longitudinal representation of the whole \(\varGamma\tto M\). The argument behind this statement makes essential use of the averaging operator for groupoid connections. In fact, the truth of this statement depends in a substantial way on the hypothesis that \(\varGamma\tto M\) is proper (compare Example~\ref{npar:14A.3.4}). The extension problem~\ref{prob:16a.8.1} is then equivalent to the following (a~priori simpler) problem:

\begin{prob}\label{prob:16a.8.2} Let \(\varGamma\tto M\emphpunct, U,V\subset M\) be as in the statement of Problem~\ref{prob:16a.8.1}. Suppose that \(\alpha\in\Rep(\varGamma\mathbin|V;\varLambda\mathbin|V)\) is a longitudinal representation of \(\varGamma\mathbin|V\tto V\). Under what conditions is it possible to extend \(\alpha\mathbin|U\in\Rep(\varGamma\mathbin|U;\varLambda\mathbin|U)\) to some longitudinal representation of all of \(\varGamma\tto M\)? \end{prob}

We now describe how Problem~\ref{prob:16a.8.2} can be rephrased as a standard problem in equivariant (orbifold) obstruction theory. This shift of perspective will offer us potential clues as to how Problem~\ref{prob:16a.8.2} (and therefore Problem~\ref{prob:16a.1.9}) could be approached in general. The findings of some preliminary investigations that we have conducted in this respect, along with several details that will be omitted in the course of the subsequent discussion, have been collected in \cite{2016b}. The development of a comprehensive obstruction theory for longitudinal representations of proper regular groupoids\textemdash with applications to classical problems in geometry\textemdash is the subject of ongoing research, and will not be discussed further in this paper.

Throughout the sequel, \(\varGamma\tto M\) shall be an arbitrary, yet fixed, proper regular groupoid, unless otherwise specified. Let \(\nef\varGamma\) denote the totally isotropic normal subgroupoid of \(\varGamma\) formed by all its ineffective arrows (recall that an isotropic arrow \(g\in\varGamma^x_x\) is \emph{ineffective} if its intrinsic \emph{infinitesimal effect} on the quotient tangent space \(T_xM\mathbin/\varLambda_x\) is trivial; cf.~\cite[\S 1]{Tre6}). By regularity and properness, this subgroupoid is of necessity smooth and closed. As a consequence,~\cite[\S A]{Tre6} the quotient groupoid
\begin{equation*}
	P=\varGamma/\nef\varGamma\tto M
\end{equation*}
inherits a natural {\em Lie}\/ groupoid structure, making it a smooth, proper, foliation groupoid. The groupoid composition law \(\varGamma\ftimes st\varGamma\to\varGamma\) descends to a Lie-group\-oid action \(\varGamma\ftimes stP\to P\). Let
\begin{equation*}
	\varPi=\varGamma\ltimes P\tto P
\end{equation*}
stand for the Lie groupoid of translations corresponding to this action. The latter groupoid is itself proper and regular. Notice that the \(\varPi\)\mdash orbits coincide with the source fibers \(P^x\) for \(P\tto M\). Let \(\varPi^x=\varPi\mathbin|P^x\tto P^x\) indicate the restriction of \(\varPi\tto P\) over an orbit \(P^x\), and
\begin{equation*}
	\mathscr R^\varGamma_x=\Rep(\varPi^x;TP^x)
\end{equation*}
the space of all tangent representations of \(\varPi^x\tto P^x\). We are going to regard the spaces \(\mathscr R^\varGamma_x\sidetext(x\in M)\) as the fibers of a set-the\-o\-ret\-ic fiber bundle \(\mathscr R^\varGamma\) over \(M\). Right translation by arrows of \(\varGamma\) induces a well-de\-fined groupoid action \(P\times_M\mathscr R^\varGamma\to\mathscr R^\varGamma\). There is a natural \(C^\infty\)\mdash structure on \(\mathscr R^\varGamma\) which makes the fi\-ber-bun\-dle projection \(\mathscr R^\varGamma\to M\) and the groupoid action \(P\times_M\mathscr R^\varGamma\to\mathscr R^\varGamma\) into \(C^\infty\) mappings and for which the \(C^\infty\) sections to \(\mathscr R^\varGamma\to M\) over any open subset \(U\) of \(M\) correspond to the longitudinal representations of the restriction of \(\varPi\tto P\) over \(P^U=\bigcup_{u\in U}P^u\subset P\). Also, for \(U\) invariant, the {\em equivariant}\/ \(C^\infty\) sections to \(\mathscr R^\varGamma\to M\) over \(U\) correspond precisely to the inverse images, along the natural projection of \(\varPi\mathbin|P^U\tto P^U\) onto \(\varGamma\mathbin|U\tto U\), of the elements of \(\Rep(\varGamma\mathbin|U;\varLambda\mathbin|U)\). We therefore see that for \(U,V\subset M\) invariant Problem~\ref{prob:16a.8.2} can be rephrased as follows:

\begin{prob}\label{prob:16a.8.3} Let \(U,V\subset M\) be invariant open sets such that \(\overline U\subset V\). Let \(\rho\in\Gamma^\infty(V;\mathscr R^\varGamma)^P\) be an arbitrary equivariant local \(C^\infty\) section to \(\mathscr R^\varGamma\to M\) defined over \(V\). What are the obstructions to extending \(\rho\mathbin|U\in\Gamma^\infty(U;\mathscr R^\varGamma)^P\) to a global equivariant \(C^\infty\) section of \(\mathscr R^\varGamma\)? \end{prob}

We refer to the common dimension of the orbits of any regular groupoid as its \emph{rank}. Notice that the rank of \(\varGamma\tto M\) is the same as the rank of \(\varPi\tto P\). Under relatively mild assumptions on \(\varGamma\tto M\) (such as for instance source properness) the pullback \(i^*\mathscr R^\varGamma\to S\) of the \(P\tto M\)~equivariant \(C^\infty\) fibration \(\mathscr R^\varGamma\to M\) along an arbitrary completely transversal submanifold \(i:S\to M\) of dimension complementary to the rank of \(\varGamma\tto M\) turns out to be \(i^*P\tto S\)~equivariantly locally trivial. From a differentiable stacks perspective, this means that the ``orbibundle'' \([\mathscr R^\varGamma/P]\to [M/P]\) is locally trivial. Accordingly, Problem~\ref{prob:16a.8.3} may be understood as asking when, for this locally trivial fibration of orbifold stacks, a given partially defined section is extendable to a global section. This is a familiar type of question in obstruction theory. Since the base \([M/P]\) of our orbifold fibration \([\mathscr R^\varGamma/P]\to[M/P]\) is a fi\-nite-di\-men\-sion\-al object, the classical methods of equivariant obstruction theory might be adapted to the study of Problem~\ref{prob:16a.8.3}. We intend to explore this order of ideas in a future paper.

\begin{exmp}\label{exmp:16a.8.4} In the rank-ze\-ro situation, there are no obstructions to solving Problem~\ref{prob:16a.8.3}; indeed, for every \(x\in M\) one has \(\dim P^x=0\) and therefore \(\mathscr R^\varGamma_x=\{\ast\}\). \end{exmp}

\begin{exmp}[see Example~1.9 of \cite{2016b}]\label{exmp:16a.8.5} When the rank is equal to \(1\), it might not always be possible to solve Problem~\ref{prob:16a.8.3}, as our counterexample \ref{npar:14A.6.9} illustrates. However, in the situation of connected source fibers, there are again no obstructions; the short argument goes as follows. Let us fix any \(P\tto M\)~invariant vec\-tor-bun\-dle metric on the longitudinal tangent distribution of \(\varPi\tto P\) (equivalently, any vec\-tor-bun\-dle metric on the longitudinal tangent distribution of \(\varGamma\tto M\)). Because of the rank~=~1 hypothesis and of the source connectedness of \(\varPi\tto P\), for every \(x\in M\) there will be exactly one representation, say, \(\sigma_x\) of \(\varPi^x\tto P^x\) on \(TP^x\) which is orthogonal for the chosen metric. These representations \(\sigma_x\) together must give a global equivariant \(C^\infty\) section \(\sigma\) of \(\mathscr R^\varGamma\). Now, because of source connectedness, for each \(h\in\varPi^V\) there will be a unique number \(c(h)>0\) such that \(\rho(h)=c(h)\sigma(h)\). By properness, this number will only depend on the pair \((sh,th)\in P\times P\). Then, for any choice of a smooth invariant function \(\varphi\) on \(M\) with \(\varphi\mathbin|U=1\) and \(\supp\varphi\subset V\), an equivariant global \(C^\infty\) prolongation \(\tilde\rho\) of \(\rho\mathbin|U\) can be obtained by setting
\begin{equation*}
	\tilde\rho(h)=\exp\bigl(\varphi(th)\log c(h)\bigr)\sigma(h).
\end{equation*} \end{exmp}

\begin{exmp}[see Theorem~6.2 of \cite{2016b}]\label{exmp:16a.8.6} The obstructions become slightly more interesting when the rank is two. Suppose that \(\varGamma\tto M\) is source proper, source connected, and that its rank is two. Further suppose that its source fibers have finite fundamental groups. Then, for any given \(U\), \(V\), and \(\rho\), the extension problem \ref{prob:16a.8.3} admits solutions if, and only if, for every base point \(x\) the four conditions below, concerning the long exact sequence of homotopy groups associated with the pointed fibration \((\varGamma^x_x,1_x)\xto\subset(\varGamma^x,1_x)\xto t(\varGamma x,x)\), are simultaneously verified, where \(\free\pi_1(-)\) denotes the fundamental group of a Lie group modulo torsion:
\begin{enumerate}
\def\labelenumi{\rm\alph{enumi})}%
\itemsep=0pt%
 \item \(\pi_2(\varGamma^x)=0\).
 \item \(F:=\im{}\bigl(\pi_2(\varGamma x)\xto{\partial_2}\pi_1(\varGamma^x_x)\xto\pr\free\pi_1(\varGamma^x_x)\bigr)\) sits inside \(\free\pi_1(\varGamma^x_x)\) so that \(\Z e\cap F\neq 0\seq 2e\in F\) for all \(e\).
 \item \(\partial_1:\pi_1(\varGamma x)\to\pi_0(\varGamma^x_x)\) is an isomorphism of groups.
 \item \(\free\pi_1(c_g)=-\id\in\Aut\bigl(\free\pi_1(\varGamma^x_x)\bigr)\) for any \(g\in\varGamma^x_x\smallsetminus\idc{\varGamma^x_x}\), where \(c_g\in\Aut(\varGamma^x_x)\) stands for conjugation by \(g\) and \(\idc{\varGamma^x_x}\) indicates the identity component of \(\varGamma^x_x\).
\end{enumerate} \end{exmp}

\begin{exmp}[see Corollary~6.4 of \cite{2016b}]\label{exmp:16a.8.7} Combining \ref{exmp:16a.8.4}\textendash \ref{exmp:16a.8.6} with Theorem~\ref{prop:14A.5.2}, one gets the following result\emphpunct: Let \(\varGamma\tto M\) be a Lie groupoid which is source proper, source connected, and whose source fibers have finite fundamental groups. Suppose that \(\dim\varGamma x\leq 2\) for all \(x\in M\), and that for every \(x\in M_2\) the conditions a)~to d) of the preceding example are satisfied. Then the \(C^\infty\)\mdash space \(\Mcon(\varGamma)\) is non-emp\-ty and \(C^\infty\)\mdash path-con\-nect\-ed. \end{exmp}
% 11401 Dec 26 17:22
\appendix

\section{Uniform convergence topologies on spaces of sections}\label{sec:16a.A}

The material collected in this appendix is relevant for Sections~\ref{sec:16a.5}~and \ref{sec:16a.6}. We assume familiarity with a few very basic concepts from the theory of topological vector spaces\textemdash such as, for instance, the notion of \emph{locally convex topology generated by a family of seminorms} or the notion of \emph{Fr\'e\-chet space}\textemdash which are thoroughly discussed e.g.~in \cite[Chapter~2]{Schaef}. Throughout the appendix, \(X\) is going to denote a smooth manifold, and \(E\) a smooth vector bundle over \(X\). We are also going to let \(0\leq k\leq\infty\) denote an arbitrary, but fixed, order of differentiability.

We say that a cross-sec\-tion \(\xi:S\to E\) defined over an arbitrary subset \(S\subset X\) \emph{is of class\/ \(C^k\)} if for each point \(x\in S\) there is some open neighborhood \(B\) of \(x\) in \(X\) over which \(\xi\mathbin|S\cap B\) admits some extension to a cross-sec\-tion of \(E\) of class \(C^k\); notice that `class~\(C^0\)' does not agree with `continuous', in general, unless \(S\) is locally closed. We let \(\Gamma^k(S;E)\) denote the vector space formed by all the cross-sec\-tions of \(E\) over \(S\) of class \(C^k\). For any subset \(T\) of \(S\) we let
\begin{equation}
\label{eqn:16a.A.1}
	\res^S_T:\Gamma^k(S;E)\longto\Gamma^k(T;E)
\end{equation}
denote the linear map given by restriction from \(S\) to \(T\).

Let \(\varphi:U\approxto\varphi U\subset\R^n\) be any local \(C^\infty\) coordinate chart for \(X\). Also let \(\tau:E\mathbin|U\simto U\times\K^N\) be any \(C^\infty\) vec\-tor-bun\-dle trivialization for \(E\) over the domain of definition of the chart. We can express an arbitrary cross-sec\-tion \(\xi\in\Gamma^k(S;E)\) locally over \(S\cap U\) in terms of its component functions relative to \(\tau,\varphi\):
\begin{equation}
\label{eqn:12B.7.1}
	\xi^{\tau,\varphi}:=(\xi^{\tau,\varphi}_1,\dotsc,\xi^{\tau,\varphi}_N):=\pr_2\circ\tau\circ\xi\circ\varphi^{-1}:\varphi(S\cap U)\to\K^N.
\end{equation}
For every multi-in\-dex \(\alpha=(\alpha_1,\dotsc,\alpha_n)\in\N^n\) of order \(\abs\alpha:=\alpha_1+\dotsb+\alpha_n\leq k\) and for every function \(f:V\to\K\) of class \(C^k\) which is defined on some open domain \(V\subset\R^n\), we adopt the customary notational shorthand \(\partial^\alpha f:=\dfrac{\partial^{\alpha_1}}{\partial t_1^{\alpha_1}}\dotsm\dfrac{\partial^{\alpha_n}}{\partial t_n^{\alpha_n}}f\) (conventionally we set \(\dfrac{\partial^0}{\partial t_i^0}f:=f\)). Let \(K\) be an arbitrary compact subset of \(U\). For every natural number \(r\leq k\) and for every global cross-sec\-tion \(\xi\in\Gamma^k(X;E)\) we set
\begin{equation}
\label{eqn:12B.7.3}
	p^{\tau,\varphi,K,r}(\xi):=\max_{I=1,\dotsc,N}\left.\max_{\alpha\in\N^n,\abs\alpha\leq r}\left.\sup_{x\in K}\left.\within*|\partial^\alpha\xi^{\tau,\varphi}_I(\varphi x)|.\right.\right.\right.
\end{equation}
Evidently, this expression defines a seminorm \(p^{\tau,\varphi,K,r}\) on the vector space \(\Gamma^k(X;E)\). The \emph{topology of\/ \(k\)\mdash th~order local uniform convergence}\textemdash shortly, \emph{\(C^k\)\mdash topology}\textemdash on \(\Gamma^k(X;E)\) is the locally convex topology generated by all such seminorms \(p^{\tau,\varphi,K,r}\). Since \(E\) is locally trivial and \(X\) is locally compact, this topology is necessarily separated (i.e.~Hausdorff). The \(C^k\)\mdash topology makes \(\Gamma^k(X;E)\) into a Fr\'e\-chet space~(= complete, metrizable, locally convex, topological vector space). When \(X\) is compact and \(k\) is finite, the Fr\'e\-chet space \(\Gamma^k(X;E)\) turns out to be normable.

In practice, we want to work with a slightly more flexible notion of `generating seminorm for the \(C^k\)\mdash topology'. By a \emph{continuous norm} \(p\) on \(E\), we shall mean the datum, on each vec\-tor-bun\-dle fiber \(E_x\), of a norm \(p_x:E_x\to\R_{\geq0}\) depending on \(x\) in such a manner that the function on \(E\) defined by \(e\mapsto p_{\pr^E_Xe}(e)\) is continuous. For example, every (Riemannian or Hermitian, depending on whether \(E\) is real or complex) metric \(\phi\) on \(E\) gives rise to a continuous norm on \(E\) defined on each fiber \(E_x\) by the rule \(e\mapsto\sqrt{\phi_x(e,e)}\); also, given a continuous norm \(p\) on \(E\), and given a similar norm \(q\) on a second vector bundle \(F\) over \(X\), on the hom vector bundle \(L(E,F)\) there is a continuous norm defined on each fiber \(L(E,F)_x=L(E_x,F_x)\) by the rule \(\lambda\mapsto\sup_{p_x(e)=1}q_x(\lambda e)\).

\begin{lem*} Let\/ \(\xi_1,\dotsc,\xi_N\) be a local frame for\/ \(E\) defined over some open set\/ \(U\subset X\). Let\/ \(K\) be a compact subset of\/ \(U\). Let\/ \(p\) be a continuous vec\-tor-bun\-dle norm on\/ \(E\mathbin|U\). Then, there exists some constant\/ \(c>0\) such that if for any cross-sec\-tion\/ \(\xi\in\Gamma^0(U;E)\) we write\/ \(\xi=\sum_Ia_I\xi_I\) with\/ \(a_I\in C^0(U)\) then\/ \(\max_I\Norm{a_I}_K\leq c\sup_{x\in K}p_x\bigl(\xi(x)\bigr)\). \end{lem*}

\begin{proof} Put
\begin{equation*}
	c^{-1}=\inf_{x\in K}\left.\inf_{\abs{z_1}^2+\dotsb+\abs{z_N}^2=1}p_x\bigl(z_1\xi_1(x)+\dotsb+z_N\xi_N(x)\bigr).\right.
\end{equation*}
Whenever \(\abs{a_I(x)}>0\) for some point \(x\in K\) and for some index \(I\), say \(I=1\), put \(b_I:=a_I(x)/a_1(x)\) for all \(I\) and let \(\rho^2:=1+\abs{b_2}^2+\dotsb+\abs{b_N}^2\geq1\). Then \((1/\rho)^2+\abs{b_2/\rho}^2+\dotsb+\abs{b_N/\rho}^2=1\), whence \(c^{-1}\leq\rho c^{-1}\leq p_x\bigl(\xi_1(x)+b_2\xi_2(x)+\dotsb+b_N\xi_N(x)\bigr)\) and therefore \(c^{-1}\abs{a_1(x)}\leq p_x\bigl(\xi(x)\bigr)\). \end{proof}

Now, let \(\varOmega\) be any relatively compact open subset of \(X\); by definition, its closure \(\overline\varOmega\) is compact. Let \(p\) be any continuous norm on \(E\). Consider \(E\) as endowed with \(p\), and write \(E=(E,p)\). Finally, let \(\matheusm A=\{(\varphi_i,\tau_i)\}\) be a trivializing atlas for \(E\) consisting of local \(C^\infty\) coordinate charts \(\varphi_i:U_i\approxto\R^n\) and \(C^\infty\) vec\-tor-bun\-dle trivializations \(\tau_i:E\mathbin|U_i\simto U_i\times\K^N\). For every index \(i\), put \(B_i:=\varphi_i^{-1}\bigl(B_1(0)\bigr)\), and for every point \(u\in U_i\), let \(\norm{-}_{i,u}\) indicate the norm on \(\mathbb K^N\) corresponding to \(p_u\) under the linear bijection \(\tau_{i,u}:E_u\simto\mathbb K^N\). Assume that \(\matheusm A\) is \emph{locally finite}, in the sense that the open sets \(B_i\) form a locally finite cover of \(X\). Then, for any non-neg\-a\-tive integer \(r\) and for any cross-sec\-tion \(\xi\in\Gamma^r(\overline\varOmega;E)\), put
\begin{equation}
\label{eqn:12B.13.3}
	\Norm\xi_{C^r\overline\varOmega;E,\matheusm A}:=\max_i\left.\max_{\alpha\in\N^n,\abs\alpha\leq r}\left.\sup_{u\in B_i\cap\varOmega}\left.\within*|\partial^\alpha\xi^{\tau_i,\varphi_i}(\varphi_iu)|_{i,u}.\right.\right.\right.
\end{equation}
We call the function \(\Norm{-}_{C^r\overline\varOmega;E,\matheusm A}\) thus defined on \(\Gamma^r(\overline\varOmega;E)\) a \emph{standard\/ \(C^r\)\mdash norm}. As an easy consequence of the previous lemma, we see that any two standard \(C^r\)\mdash norms on \(\Gamma^r(\overline\varOmega;E)\) are equivalent. We refer to the (normable) locally convex topology on \(\Gamma^r(\overline\varOmega;E)\) generated by any standard \(C^r\)\mdash norm as the \emph{\(C^r\)\mdash norm topology}. By a \emph{\(C^r\)\mdash norm} on \(\Gamma^r(\overline\varOmega;E)\) we simply mean any norm which is equivalent to some standard \(C^r\)\mdash norm.

\begin{lem}\label{cor:12B.13.5} Let\/ \(X\) be a smooth manifold and let\/ \(E\) be a smooth vector bundle over\/ \(X\). Then for an arbitrary relatively compact open subset\/ \(\varOmega\subset X\) and an arbitrary non-neg\-a\-tive integer\/ \(r\leq k\leq\infty\) the restriction map\/ \(\res^X_{\overline\varOmega}:\Gamma^k(X;E)\to\Gamma^r(\overline\varOmega;E)\) is continuous relative to the\/ \(C^k\)\mdash topology on the first space and to the\/ \(C^r\)\mdash norm topology on the second space. If\/ \(\mathcal U\) is any open cover of\/ \(X\) by relatively compact open subsets, the\/ \(C^k\)\mdash topology on\/ \(\Gamma^k(X;E)\) coincides with the weak topology induced by the family of linear maps
\begin{equation*}
	\{\res^X_{\overline\varOmega}:\Gamma^k(X;E)\to\Gamma^r(\overline\varOmega;E)\mathrel|\varOmega\in\mathcal U\emphpunct, 0\leq r\text{ finite integer}\leq k\}\text;
\end{equation*}
in particular, a sequence\/ \(\{\xi_i\}\) of global cross-sec\-tions of\/ \(E\) is Cauchy within\/ \(\Gamma^k(X;E)\) if, and only if, so is the sequence\/ \(\{\res^X_{\overline\varOmega}(\xi_i)\}\) within\/ \(\Gamma^r(\overline\varOmega;E)\) for every\/ \(\varOmega\in\mathcal U\emphpunct, r\leq k\). \end{lem}

\begin{lem}\label{lem:12B.13.9} Let\/ \(\omega:E\to F\) be a fiberwise linear morphism between two smooth vector bundles\/ \(E\), \(F\) over a manifold\/ \(X\). For any relatively compact open subset\/ \(\varOmega\subset X\) and for any non-neg\-a\-tive integer\/ \(r\), the following linear map is\/ \(C^r\)\mdash continuous.
\begin{equation}
\label{eqn:12B.13.14}
	\Gamma^r(\overline\varOmega;E)\longto\Gamma^r(\overline\varOmega;F)\emphpunct, \xi\mapsto\omega\circ\xi
\end{equation} \end{lem}

For any smooth mapping \(f:Y\to X\), the pullback vector bundle \(f^*E\) has the fiber product \(Y\times_XE\) as its total manifold and the first projection \(Y\times_XE\to Y\) as its bundle projection onto \(Y\). By the universal property of the fiber product, for each \(C^k\) cross-sec\-tion \(\xi\) of \(E\) there exists a unique \(C^k\) cross-sec\-tion \(f^*\xi\) of \(f^*E\) such that \(\pr_E\circ f^*\xi=\xi\circ f\), where \(\pr_E\) denotes the projection \(Y\times_XE\to E\).

\begin{lem}\label{lem:12B.13.10} Let\/ \(f:Y\to X\) be a smooth mapping and let\/ \(E\) be a smooth vector bundle over\/ \(X\). Then, for any relatively compact open subset\/ \(\varOmega\subset X\) and for any similar subset\/ \(O\subset Y\) such that\/ \(f(O)\subset\varOmega\), the pullback operation on cross-sec\-tions gives rise for each non-neg\-a\-tive integer\/ \(r\) to a\/ \(C^r\)\mdash continuous linear map
\begin{equation}
\label{eqn:12B.13.15}
	\Gamma^r(\overline\varOmega;E)\longto\Gamma^r(\overline O;f^*E)\emphpunct, \xi\mapsto f^*\xi\mathbin|\overline O.
\end{equation} \end{lem}

We proceed to describe a notational device which will spare us the nuisance of keeping track of irrelevant scaling factors throughout. Let \(\mathscr S\) be an arbitrary set. We introduce a binary relation \(\preceq\) on the set \(\Func_{\geq0}(\mathscr S)\) of all non-neg\-a\-tive real valued functions on \(\mathscr S\) by defining \(f\preceq g\) to mean \em there exists some constant\/ \(C>0\) such that\/ \(f\leq Cg\)\em. Since this binary relation is reflexive and transitive, setting \(f\equiv g\aeq(f\preceq g\et g\preceq f)\) gives rise to an equivalence relation \(\equiv\) on \(\Func_{\geq0}(\mathscr S)\). Note that \(\preceq\) descends to a partial order \(\leq\) on the set of all \(\equiv\)~equivalence classes of functions. Also note that for \(f,g,h\in\Func_{\geq0}(\mathscr S)\)
\[f\equiv g\text{ entails }f+h\equiv g+h\text{ and }fh\equiv gh.\]
In addition, notice that if \(\lambda:\mathscr T\to\mathscr S\) is any mapping then \(f\equiv g\in\Func_{\geq0}(\mathscr S)\) implies \(f\circ\lambda\equiv g\circ\lambda\in\Func_{\geq0}(\mathscr T)\). Hence, the operations of sum, product, and pullback make sense for \(\equiv\)~equivalence classes of functions. Now, for \(E\), \(\varOmega\) and \(r\) as before, let \(\Norm{-}_{C^r\overline\varOmega;E}\) (or simply \(\Norm{-}_{C^r\overline\varOmega}\) or even \(\within\|-\|_{C^r}\) when there is no risk of confusion) denote the \(\equiv\)~class of any \(C^r\)\mdash norm within \(\Func_{\geq0}\bigl(\Gamma^r(\overline\varOmega;E)\bigr)\).

\begin{lem}\label{lem:12B.13.7} Let\/ \(E\), \(F\) and\/ \(G\) be any three smooth vector bundles over a given manifold\/ \(X\). Let\/ \(\varOmega\) be an arbitrary relatively compact open subset of\/ \(X\). Let\/ \(\eta\) denote a variable ranging over\/ \(\Gamma^r\bigl(\overline\varOmega;L(E,F)\bigr)\) and\/ \(\vartheta\) one ranging over\/ \(\Gamma^r\bigl(\overline\varOmega;L(F,G)\bigr)\) for some finite order of differentiability\/ \(r\geq 0\). Then, the following two estimates hold (the second of which for\/ \(r\geq 1\)), where\/ \(\vartheta\circ\eta\) stands for the cross-sec\-tion of\/ \(L(E,G)\) obtained by composing\/ \(\vartheta\) with\/ \(\eta\) pointwise.%
\begin{subequations}
\label{eqn:12B.13.5}
\begin{gather}
	\within\|\vartheta\circ\eta\|_{C^r}\leq\within\|\vartheta\|_{C^r}\within\|\eta\|_{C^r}
\label{eqn:12B.13.5a}
\\	\within\|\vartheta\circ\eta\|_{C^r}\leq\within\|\vartheta\|_{C^r}\within\|\eta\|_{C^{r-1}}+\within\|\vartheta\|_{C^{r-1}}\within\|\eta\|_{C^r}
\label{eqn:12B.13.5b}
\end{gather}
\end{subequations} \end{lem}

\begin{lem}\label{lem:12B.13.8} Let\/ \(E\), \(F\) be any two smooth vector bundles over a given manifold\/ \(X\). Let\/ \(\varOmega\) be an arbitrary relatively compact open subset of\/ \(X\). Let\/ \(\eta\) denote a variable ranging over\/ \(\Gamma^r\bigl(\overline\varOmega;\Lis(E,F)\bigr)\) for some finite order of differentiability\/ \(r\), where\/ \(\Lis(E,F)\) indicates the open subset of\/ \(L(E,F)\) formed by all the linear isomorphisms. Then, letting\/ \(\eta^{-1}\) stand for the cross-sec\-tion of\/ \(L(F,E)\) obtained by inverting\/ \(\eta\) pointwise, providing that\/ \(r\geq 1\),
\begin{equation}
\label{eqn:12B.13.9}
	\within*\|\eta^{-1}\|_{C^r}\leq\within*(\within*\|\eta^{-1}\|_{C^{r-1}})^2\within\|\eta\|_{C^r}.
\end{equation} \end{lem}
% 13727 Dec 28 18:27

\section{Haar integrals depending on parameters}\label{sec:16a.B}

The present appendix is a continuation of the preceding one. Again the material collected here is mostly needed in Sections~\ref{sec:16a.5}~and \ref{sec:16a.6}. The classical sources on Haar integration over (topological or Lie) groupoids include \cite{Pater,Ren,Tu}.

Let \(p:Y\to X\) be any smooth and submersive mapping. We say that a subset \(S\subset Y\) is \emph{properly located} (\emph{with respect to\/ \(p\)}) if the restriction of \(p\) to \(S\) is a proper mapping. We say that a function \(\alpha\) on \(Y\) is \emph{properly supported} if its support \(\supp\alpha\) is properly located. By a (\emph{volume}) \emph{density along the fibers} of \(p\), we mean any global (non-van\-ish\-ing) smooth section of the density (line) bundle associated with the vertical subbundle \(\ker Tp\subset TY\). For any such density \(\delta\), the following two statements are true:%
\begin{enumerate}
\def\labelenumi{\upshape(\roman{enumi})}
 \item If \(\alpha\) is any properly supported \(C^k\) function on \(Y\), the function \({\int}\alpha\delta\) on \(X\) obtained by integration along the fiber is also of class \(C^k\).
 \item For any properly located subset \(S\subset Y\), the operation of integration along the fiber gives rise to a \(C^k\)\mdash continuous linear map \(C^k_S(Y)\to C^k(X)\emphpunct, \alpha\mapsto{\int}\alpha\delta\), where \(C^k_S(Y)\) denotes the closed linear subspace of \(C^k(Y)\) formed by those functions \(\alpha\) such that \(\supp\alpha\subset S\).
\end{enumerate}
(Both claims are clear when \(Y\xto pX\) is the projection \(\R^m\times\R^n\xto\pr\R^m\); the general case is seen to reduce to this special case by a straightforward partition of unity argument based on Lemmas~\ref{cor:12B.13.5}\textendash\ref{lem:12B.13.10}.)

\begin{lem}\label{lem:12B.10.6} For any normalized Haar system\/ \(\nu\) on a proper Lie groupoid\/ \(\varGamma\tto M\) and any choice of ``parameter data'' \(\{P\xto fM\emphpunct, E\}\) the Haar integration functional\/ \(\vartheta\mapsto\integral\vartheta d\nu\) given by\/ \eqref{eqn:12B.10.6a} is a\/ \(C^k\)\mdash continuous linear map of\/ \(\Gamma^k(P\ftimes ft\varGamma\emphpunct*;\pr_P^*E)\) into\/ \(\Gamma^k(P;E)\). \end{lem}

\begin{proof} Since the integration can be done componentwise relative to any local vec\-tor-bun\-dle trivialization for \(E\) over the domain of definition of a local coordinate chart for \(P\), our task reduces at once to the case of functions (``trivial coefficients'').

Let \(\tau\) and \(c\) be a volume density along the target fibers and a non-neg\-a\-tive function on \(M\) as in the definition of `normalized Haar system' reviewed in \S\ref{sec:16a.1}. The projection map \(\pr_P:P\ftimes ft\varGamma\to P\) is a surjective submersion; also, the tangent map of the other projection \(\pr_\varGamma:P\ftimes ft\varGamma\to\varGamma\) induces an identification of vector bundles \(\ker(T\pr_P)\cong\pr_\varGamma^*(\ker Ts)\) which makes it possible to regard the pullback \(\pr_\varGamma^*\tau\) as a volume density along the fibers of \(\pr_P\). Now, the Haar integration map on functions can be expressed as the composition of two linear maps which by Lemma~\ref{lem:12B.13.9} and the above statement (ii) are already known to be \(C^k\)\mdash continuous, namely,
\begin{equation*}
	C^k(P\ftimes ft\varGamma)
 \mathrel{\xymatrix@1@C=4.67em{\ar[r]^{\alpha\,\mapsto\,(c\circ s\circ\pr_\varGamma)\alpha}&}}
	C^k_{\supp c\circ s\circ\pr_\varGamma}(P\ftimes ft\varGamma)
 \mathrel{\xymatrix@1@C=3.67em{\ar[r]^{\alpha\,\mapsto\,{\int}\alpha\mkern 2mu\pr_\varGamma^*\tau}&}}
	C^k(P). \qedhere
\end{equation*} \end{proof}

In the course of the previous proof we have tacitly made use of the equivalence between the following two properties, for any function \(c\) defined on the base \(M\) of a proper Lie groupoid \(\varGamma\tto M\):%
\begin{itemize}
\itemsep=0pt%
 \item The function \(c\circ s\) on \(\varGamma\) is properly supported with respect to \(t:\varGamma\to M\).
 \item The set\/ \(\supp c\cap\varGamma K\) is compact for every compact \(K\subset M\).
\end{itemize}
A function which enjoys these properties is commonly referred to as a ``cut-off'' function. The existence of cut-off functions on the base of any proper Lie groupoid, which implies the existence of normalized Haar systems, can easily be established by adapting the proof of \cite[Proposition~6.11]{Tu} from the continuous to the smooth case.

We shall say that a (non-emp\-ty) open subset \(U\) of the base \(M\) of a proper Lie groupoid \(\varGamma\tto M\) is \emph{adjusted} to a cut-off function \(c\) if \(\supp c\cap\varGamma U\subset U\). If \(U\) is adjusted to \(c\), then the restriction of \(c\) to \(U\) will be a cut-off function for \(\varOmega:=\varGamma\mathbin|U\tto U\). If \(\nu=(\tau,c)\) is a normalized Haar system on \(\varGamma\tto M\) and if \(U\subset M\) is any open subset which is adjusted to \(\nu\) in the sense that it is adjusted to \(c\), then \(\nu\mathbin|U:=(\tau\mathbin|\varOmega,c\mathbin|U)\) will be a normalized Haar system on \(\varOmega\tto U\); moreover, given \(f:P\to M\) and \(E\) as in the statement of Lemma~\ref{lem:12B.10.6}, for every continuous cross-sec\-tion \(\vartheta\in\Gamma^0(P\ftimes ft\varGamma\emphpunct*;\pr_P^*E)\),
\begin{equation}
\label{eqn:12B.18.3}
	\bigl(\integral*\vartheta d\nu\bigr)\mathbin|f^{-1}(U)=\integral*\left.\vartheta\mathbin|f^{-1}(U)\ftimes ft\varOmega\right.d(\nu\mathbin|U).
\end{equation}

\begin{lem}\label{lem:12B.20.1} Let\/ \(\varGamma\tto M\emphpunct, \nu\emphpunct, f:P\to M\) and\/ \(E\) be as in the statement of Lemma~\ref{lem:12B.10.6}. Let\/ \(U\) be any non-emp\-ty, relatively compact, open subset of\/ \(M\) which is adjusted to\/ \(\nu\), and let\/ \(V\) be any non-emp\-ty, relatively compact, open subset of\/ \(P\) such that\/ \(f(V)\subset U\). Then, for every finite order of differentiability\/ \(r\geq 0\), setting\/ \(\varOmega=s^{-1}(U)\cap t^{-1}(U)\), there exists a unique\/ \(C^r\)\mdash continuous linear map\/ \textcircled? that solves the commutativity problem below.
\begin{equation}
\label{eqn:12B.20.1}
\begin{split}
\xymatrix@C=3.33em@R=5ex{%
 \Gamma^r(P\ftimes ft\varGamma\emphpunct*;\pr_P^*E)
 \ar[d]^{\res^{P\ftimes ft\varGamma}_{\overline{V\ftimes ft\varOmega}}}
 \ar[r]^-{\vartheta\,\mapsto\,\integral\vartheta d\nu}_-{\eqref{eqn:12B.10.6a}}
 &	\Gamma^r(P;E)
	\ar[d]^{\res^P_{\overline V}}
\\ \Gamma^r(\overline{V\ftimes ft\varOmega}\emphpunct*;\pr_P^*E)
 \ar@{.>}[r]^-{\text{\textcircled?}}
 &	\Gamma^r(\overline V;E)
}\end{split}
\end{equation} \end{lem}

\begin{proof} We are going to make use the same notations as in the proof of Lemma~\ref{lem:12B.10.6} without further notice. A straightforward partition of unity argument shows that the restriction maps in \eqref{eqn:12B.20.1} are surjective. There will therefore be at most one solution to the problem represented by \eqref{eqn:12B.20.1}. To confirm that one such solution exists, one needs to verify that for every cross-sec\-tion \(\vartheta\in\Gamma^r(P\ftimes ft\varGamma\emphpunct*;\pr_P^*E)\) which vanishes on \(V\ftimes ft\varOmega\) the cross-sec\-tion \(\integral\vartheta d\nu\in\Gamma^r(P;E)\) also vanishes on \(V\). Now, for every \(y\in f^{-1}(U)\) we have \(c(sh)=0\) for \(h\in\varGamma_{f(y)}\smallsetminus s^{-1}(U)\) because \(U\) is adjusted to \(c\). So, whenever \(y\in V\),
\begin{equation*}
	\integral\vartheta(y,h)d\nu_{f(y)}(h)=\integral\vartheta(y,h)c(sh)d\mu_{f(y)}(h)=0.
\end{equation*}

Reduction to the case of trivial coefficients is clear when \(V\) is so small that its closure \(\overline V\) lies within the domain of definition of some local trivializing chart for the vector bundle \(E\). For general \(V\), we first choose a finite cover \(\{V_i\}\) of \(\overline V\) by relatively compact open sets so that \(E\) trivializes around each closure \(\overline{V_i}\), and then a partition of unity \(\{g_i\}\) over \(\overline V\) subordinated to this cover in the sense that \(\supp g_i\subset V_i\) and \(\sum_ig_i=1\) on \(\overline V\). Then
\begin{equation*}
\textstyle%
	\integral*\vartheta d\nu=\sum_ig_i\integral*\vartheta d\nu=\sum_i\integral*(g_i\circ\pr_P)\vartheta d\nu.
\end{equation*}
Now each correspondence \(\vartheta\mapsto\integral(g_i\circ\pr_P)\vartheta d\nu\) can be viewed as a composite linear map
\begin{multline*}
	\Gamma^r(\overline{V\ftimes ft\varOmega}\emphpunct*;\pr_P^*E)
 \mathrel{\xymatrix@1@C=4.33em{\ar[r]^{\vartheta\,\mapsto\,(g_i\circ\pr_P)\vartheta}&}}
	\Gamma^r(\overline{V\ftimes ft\varOmega}\emphpunct*;\pr_P^*E)
 \mathrel{\xymatrix@1@C=1.33em{\ar[r]^-\res &}}
\\	\Gamma^r(\overline{[V\cap V_i]\ftimes ft\varOmega}\emphpunct*;\pr_P^*E)
 \mathrel{\xymatrix@1@C=3.67em{\ar@{.>}[r]^-{\text{\textcircled? for $V\mathord\cap V_i$}}&}}
	\Gamma_{\overline V\cap\supp g_i}^r(\overline{V\cap V_i};E)
 \mathrel{\xymatrix@1@C=5.00em{\ar[r]^-{\text{extension by zero}}&}}
	\Gamma^r(\overline V;E),
\end{multline*}
the first, second, and last map being clearly \(C^r\)\mdash continuous. The problem is thus reduced to the case of functions.

Let \(g\) stand for the function \(c\circ s\circ\pr_\varGamma\) on \(\varGamma\), and let \(W\) stand for the relatively compact open subset \(V\ftimes ft\varOmega\) of \(P\ftimes ft\varGamma\). We contend that \(g\) is a properly supported function with respect to \(\pr_P\) and moreover that
\begin{equation*}
	\pr_P(W)\subset V\emphpunct{ }\text{and}\emphpunct{ }\pr_P^{-1}(V)\cap\supp g\subset W.
\tag{\dag}
\end{equation*}
To see this, notice that for every subset \(A\subset P\)
\begin{align*}
 \pr_P^{-1}(A)\cap\supp(c\circ s\circ\pr_\varGamma)
	&\subset\pr_P^{-1}(A)\cap\pr_\varGamma^{-1}\bigl(\supp(c\circ s)\bigr)
\\	&=A\ftimes ft\supp(c\circ s).
\tag{\ddag}
\end{align*}
Now, if \(A\) is compact then \(t^{-1}\bigl(f(A)\bigr)\cap\supp(c\circ s)\) is compact too, since \(c\circ s\) is properly supported with respect to \(t\) because \(c\) is a cut-off function. This shows that \(g\) has to be properly supported with respect to \(\pr_P\). If on the other hand in (\ddag) we take \(A=V\) then, since \(f(V)\subset U\) and since \(U\) is adjusted to \(c\),
\begin{equation*}
	\supp(c\circ s)\cap t^{-1}\bigl(f(V)\bigr)\subset s^{-1}(\supp c)\cap t^{-1}(U)\subset\varOmega,
\end{equation*}
which establishes (\dag).

Now, let \(\delta:=\pr_\varGamma^*\tau\) (regarded as a volume density along the fibers of \(\pr_P\)). Our Haar integration functional \textcircled? on functions coincides with the linear map
\begin{equation*}
\textstyle%
	C^r(\overline W)\to C^r(\overline V)\emphpunct, \alpha\mapsto{\int}\alpha g\delta
\end{equation*}
described as follows\emphpunct: (1)~extend \(\alpha g\) to some properly supported \(C^r\) function defined on some ``tube'' \(W_1=\pr_P^{-1}(V_1)\) with \(V_1\supset\overline V\) open [this is always possible in virtue of (\dag)]\emphpunct; (2)~integrate the extended function against \(\delta\mathbin|W_1\)\emphpunct; (3)~restrict the result of the integration from \(V_1\) to \(\overline V\). We leave it as an exercise for the reader to verify that the linear map thus obtained is \(C^r\)\mdash continuous. ({\it Hint\emphpunct:}\/ By using a suitable partition of unity, reduce the proof to a computation in local coordinates.) \end{proof}
% 11346 Dec 31 11:30

% references %%%%%%%%%%%%%%%%%%%%%%%%%%%%%%%%%%%%
%%
{\footnotesize
\bibliographystyle{abbrv}
\bibliography{bib/2016a,bib/gtrentin}

\begin{thebibliography}{10}

\bibitem{Bl06}
A.~D. Blaom.
\newblock Geometric structures as deformed infinitesimal symmetries.
\newblock {\em Trans. Amer. Math. Soc.}, 358(8):3651--3671, 2006.

\bibitem{Bl12}
A.~D. Blaom.
\newblock Lie algebroids and {C}artan's method of equivalence.
\newblock {\em Trans. Amer. Math. Soc.}, 364(6):3071--3135, 2012.

\bibitem{Bl13}
A.~D. Blaom.
\newblock The infinitesimalization and reconstruction of locally homogeneous
  manifolds.
\newblock {\em SIGMA Symmetry Integrability Geom. Methods Appl.}, 9(074):1--19,
  2013.

\bibitem{Bla16}
A.~D. Blaom.
\newblock Cartan connections on {L}ie groupoids and their integrability.
\newblock {\em SIGMA Symmetry Integrability Geom. Methods Appl.},
  12(114):1--26, 2016.

\bibitem{Bl16a}
A.~D. Blaom.
\newblock Pseudogroups via pseudoactions: unifying local, global, and
  infinitesimal symmetry.
\newblock {\em J. Lie Theory}, 26(2):535--565, 2016.

\bibitem{Chern}
S.~S. Chern.
\newblock The geometry of {$G$}\mdash structures.
\newblock {\em Bull. Amer. Math. Soc.}, 72(2):167--219, 1966.

\bibitem{CSS}
M.~Crainic, M.~A. Salazar, and I.~Struchiner.
\newblock Multiplicative forms and {S}pencer operators.
\newblock {\em Math. Z.}, 279:939--979, 2015.

\bibitem{CS}
M.~Crainic and I.~Struchiner.
\newblock On the linearization theorem for proper {L}ie groupoids.
\newblock {\em Ann. Sci. \'Ec. Norm. Sup\'er. (4)}, 46(5):723--746, 2013.

\bibitem{CrampS}
M.~Crampin and D.~Saunders.
\newblock {\em Cartan geometries and their symmetries. A {L}ie algebroid
  approach}.
\newblock Number~4 in Atlantis Studies in Variational Geometry. Atlantis Press,
  Paris, 2016.

\bibitem{delaHK}
P.~de~la Harpe and M.~Karoubi.
\newblock Repr\'esentations approch\'ees d'un groupe dans une alg\`ebre de
  {B}anach.
\newblock {\em Manuscripta Math.}, 22(3):293--310, 1977.

\bibitem{Ehr50}
C.~Ehresmann.
\newblock Les connexions infinit\'esimales dans un espace fibr\'e
  diff\'erentiable.
\newblock In {\em Colloque de topologie (espaces fibr\'es), {B}ruxelles, 1950},
  pages 29--55. Georges Thone, Li\`ege; Masson et Cie., Paris, 1951.

\bibitem{Ehr58}
C.~Ehresmann.
\newblock Sur les pseudogroupes de {L}ie de type fini.
\newblock {\em C. R. Acad. Sci. Paris}, 246:360--362, 1958.

\bibitem{GKR}
K.~Grove, H.~Karcher, and E.~A. Ruh.
\newblock Group actions and curvature.
\newblock {\em Invent. Math.}, 23:31--48, 1974.

\bibitem{Haefli}
A.~Haefliger.
\newblock Feuilletages sur les vari\'et\'es ouvertes.
\newblock {\em Topology}, 9:183--194, 1970.

\bibitem{JO}
M.~Jotz~Lean and C.~Ortiz.
\newblock Foliated groupoids and infinitesimal ideal systems.
\newblock {\em Indag. Math. (N.S.)}, 25(5):1019--1053, 2014.

\bibitem{Koba}
S.~Kobayashi.
\newblock {\em Transformation groups in differential geometry}.
\newblock Classics in Mathematics. Sprin\-ger-Ver\-lag, Ber\-lin, New~York,
  1972.

\bibitem{MM}
I.~Moerdijk and J.~Mr{\v c}un.
\newblock {\em Introduction to foliations and {L}ie groupoids}.
\newblock Number~91 in Cambridge Studies in Advanced Mathematics. Cambridge
  University Press, Cambridge, 2003.

\bibitem{PS}
R.~S. Palais and T.~E. Stewart.
\newblock Deformations of compact differentiable transformation groups.
\newblock {\em Amer. J. Math.}, 82(4):935--937, 1960.

\bibitem{Pater}
A.~L.~T. Paterson.
\newblock {\em Groupoids, inverse semigroups, and their operator algebras}.
\newblock Number 170 in Progress in Mathematics. Birk\-h\"au\-ser Boston, Inc.,
  Boston, MA, 1999.

\bibitem{Ren}
J.~Renault.
\newblock {\em A groupoid approach to\/ $C^*$\mdash algebras}.
\newblock Number 793 in Lecture Notes in Mathematics. Sprin\-ger-Ver\-lag,
  Ber\-lin, Hei\-del\-berg, New~York, 1980.

\bibitem{Salaz}
M.~A. Salazar~Pinzon.
\newblock {\em Pfaffian groupoids}.
\newblock PhD thesis, Utrecht University, 2013.

\bibitem{Schaef}
H.~H. Schaefer.
\newblock {\em Topological vector spaces}.
\newblock Number~3 in Graduate Texts in Mathematics. Sprin\-ger-Ver\-lag,
  New~York, 1971.

\bibitem{Sharpe}
R.~W. Sharpe.
\newblock {\em Differential geometry. {C}artan's generalization of {K}lein's
  {E}rlangen program}.
\newblock Number 166 in Graduate Texts in Mathematics. Sprin\-ger-Ver\-lag,
  New~York, 1997.

\bibitem{Steen}
N.~Steenrod.
\newblock {\em The topology of fibre bundles}.
\newblock Number~14 in Princeton Mathematical Series. Princeton University
  Press, Princeton, NJ, 1951.

\bibitem{Stern}
S.~Sternberg.
\newblock {\em Lectures on differential geometry}.
\newblock Chelsea Publishing Co., New York, second edition, 1983.

\bibitem{Tang}
X.~Tang.
\newblock Deformation quantization of pseudo-symplectic ({P}oisson) groupoids.
\newblock {\em Geom. Funct. Anal.}, 16(3):731--766, 2006.

\bibitem{Thur74}
W.~Thurston.
\newblock The theory of foliations of codimension greater than one.
\newblock {\em Comment. Math. Helv.}, 49:214--231, 1974.

\bibitem{Thur76}
W.~P. Thurston.
\newblock Existence of codimension-one foliations.
\newblock {\em Ann. of Math. (2)}, 104(2):249--268, 1976.

\bibitem{Tre6}
G.~Trentinaglia.
\newblock Reduced smooth stacks?
\newblock {\em Theory Appl. Categ.}, 30(31):1032--1066, July 2015.

\bibitem{2016b}
G.~Trentinaglia.
\newblock Regular {C}artan groupoids and longitudinal representations.
\newblock Preprint arXiv:\hskip.pt\relax 1508.\hskip.pt\relax
  00489\hskip.pt\relax v2 [math.DG], 12~Nov. 2017.

\bibitem{Tu}
J.~L. Tu.
\newblock La conjecture de {N}ovikov pour les feuille\-tages
  hy\-per\-bo\-liques.
\newblock {\em $K$\mdash Theory}, 16(2):129--184, 1999.

\bibitem{Wein00}
A.~Weinstein.
\newblock Almost invariant submanifolds for compact group actions.
\newblock {\em J. Eur. Math. Soc. (JEMS)}, 2(1):53--86, 2000.

\bibitem{Wein}
A.~Weinstein.
\newblock Linearization of regular proper groupoids.
\newblock {\em J. Inst. Math. Jussieu}, 1(3):493--511, 2002.

\bibitem{Yudil}
O.~Yudilevich.
\newblock {\em Lie pseudogroups \`a la {C}artan from a modern perspective}.
\newblock PhD thesis, Utrecht University, 2016.

\bibitem{Zung}
N.~T. Zung.
\newblock Proper groupoids and momentum maps: linearization, affinity, and
  convexity.
\newblock {\em Ann. Sci. \'Ecole Norm. Sup. (4)}, 39(5):841--869, 2006.

\end{thebibliography}
}%
\end{document}